\documentclass{svproc}

\usepackage{color}
\usepackage[dvipsnames]{xcolor}

\usepackage{amsmath}
\usepackage{amssymb}
\usepackage{amsfonts}
\usepackage{stmaryrd}
\usepackage{soul}
\usepackage{array}
\usepackage{empheq}
\usepackage{xfrac}
\usepackage{minibox}
\usepackage{booktabs}
\usepackage{enumitem}
	\setlist{nosep} \usepackage{color}
\usepackage{hyperref}
\usepackage{cleveref}
\usepackage{arydshln}
\usepackage{subcaption}
\usepackage{tikz}
\usepackage{ulem}
\usepackage{pgfplots}
	\pgfplotsset{compat=newest}
	\pgfplotsset{plot coordinates/math parser=false}

\usepackage{placeins}\usetikzlibrary{shapes,patterns,snakes}

\usepackage{makecell}
\usepackage{hhline}

\newlength\figureheight
\newlength\figurewidth

\usepackage{lineno}
\makeatletter
\let\LN@align\align
\let\LN@endalign\endalign
\renewcommand{\align}{\linenomath\LN@align}
\renewcommand{\endalign}{\LN@endalign\endlinenomath}
\let\LN@gather\gather
\let\LN@endgather\endgather
\renewcommand{\gather}{\linenomath\LN@gather}
\renewcommand{\endgather}{\LN@endgather\endlinenomath}
\makeatother

\allowdisplaybreaks

\usepackage{url}

\begin{document}
\mainmatter              \title{Tight two-level convergence of Linear Parareal and MGRIT: Extensions and implications
in practice}
\titlerunning{Two-level convergence of Linear Parareal and MGRIT}
\author{Ben S. Southworth\inst{1} \and
	Wayne Mitchell\inst{2} \and
	Andreas Hessenthaler \and
	Federico Danieli\inst{3}}
\authorrunning{Southworth, Mitchell, Hessenthaler, and Danieli} \tocauthor{Ben S. Southworth, Wayne Mitchell, Andreas Hessenthaler, and Federico Danieli}
\institute{Los Alamos National Laboratory, Los Alamos NM, 87544, USA,
	\email{southworth@lanl.gov}
\and
	Heidelberg University, Heidelberg, Germany
\and
	University of Oxford, Oxford, England}

\maketitle              
\begin{abstract}
Two of the most popular parallel-in-time methods are Parareal and
multigrid-reduction-in-time (MGRIT). Recently, a general convergence theory was
developed in Southworth (2019) for linear two-level MGRIT/Parareal that provides
necessary and sufficient conditions for convergence, with tight bounds on
worst-case convergence factors. This paper starts by providing a new and
simplified analysis of linear error and residual propagation of Parareal,
wherein the norm of error or residual propagation is given by one over the
minimum singular value of a certain block bidiagonal operator. New discussion is
then provided on the resulting necessary and sufficient conditions for
convergence that arise by appealing to block Toeplitz theory as in Southworth
(2019). Practical applications of the theory are discussed, and the convergence
bounds demonstrated to predict convergence in practice to high accuracy on two
standard linear hyperbolic PDEs: the advection(-diffusion) equation, and the
wave equation in first-order form.
\end{abstract}

\section{Background} \label{sec:background}

Two of the most popular parallel-in-time methods are Parareal \cite{Parareal} and
multigrid-reduction-in-time (MGRIT) \cite{Falgout:2014}.
Convergence of Parareal/two-level MGRIT has been considered in a number of papers
\cite{Dobrev:2017,Bal:2005cw,Gander:2007jq,Gander:2008bt,Ruprecht:2014gd,Wu:2015ht,WU15,MGRIT19,19c_mgrit}.
Recently, a general convergence theory was developed for linear two-level MGRIT/Parareal
that provides necessary and sufficient conditions for convergence, with tight bounds
on worst-case convergence factors \cite{southworth19}, and does not rely on assumptions
of diagonalizability of the underlying operators. \Cref{sec:simple} provides a simplified
analysis of linear Parareal and MGRIT that expresses the norm of error or residual propagation of two-level
linear Parareal and MGRIT precisely as one over the minimum singular value of a
certain block bidiagonal operator (rather than the more involved pseudoinverse
approach used in \cite{southworth19}). We then provide new discussion on the
resulting necessary and sufficient conditions for convergence that arise by
appealing to block Toeplitz theory \cite{southworth19}. In addition, the
framework developed in \cite{southworth19} is extended to provide necessary
conditions for the convergence of a single iteration, followed by a discussion
putting this in the context of multilevel convergence in \Cref{sec:theory:single}.
Practical applications of the theory are discussed in \Cref{sec:imag}, and the convergence bounds
demonstrated to predict convergence in practice to high accuracy on two standard linear
hyperbolic PDEs, the advection(-diffusion) equation, and the wave equation
in first-order form.

\section{Two-level convergence} \label{sec:simple}

\subsection{A linear algebra framework} \label{sec:simple:framework}

Consider time-integration with $N$ discrete time points. Let $\Phi(t)$ be a time-dependent, linear,
and stable time-propagation operator, with subscript $\ell$ denoting $\Phi_\ell := \Phi(t_\ell)$,
and let $\mathbf{u}_\ell$ denote the (spatial) solution at time-point $t_\ell$. Then, consider the
resulting space-time linear system,
\begin{align}\label{eq:system2}
A\mathbf{u} = \begin{bmatrix} I \\ -\Phi_1 & I \\ & -\Phi_2 & I \\ & & \ddots & \ddots \\
	& & & -\Phi_{N-1} & I \end{bmatrix}
	\begin{bmatrix} \mathbf{u}_0 \\ \mathbf{u}_1\\ \mathbf{u}_2 \\ \vdots \\ \mathbf{u}_{N-1}\end{bmatrix} = \mathbf{f},
\end{align}
Clearly \eqref{eq:system2} can be solved directly using block forward substitution, which
corresponds to standard sequential time stepping. Linear Parareal and MGRiT are reduction-based
multigrid methods, which solve \eqref{eq:system2} in a parallel iterative manner. 

First, note there is a closed form inverse for matrices with the form in \eqref{eq:system2},
which will prove useful for further derivations. Excusing the slight abuse of notation,
define $\Phi_{i}^{j} := \Phi_i\Phi_{i-1}...\Phi_j$. Then,
\begin{align}\label{eq:geninv}
\begin{bmatrix} I \\ -\Phi_1 & I \\ & -\Phi_2 & I \\ & & \ddots & \ddots \\ & & & -\Phi_{N-1} & I \end{bmatrix}^{-1}
	& =
\begin{bmatrix} I \\ \Phi_1 & I \\
	\Phi_2^1 & \Phi_2 & I \\
	\Phi_3^1 & \Phi_3^2 & \Phi_3 & I \\ 
	\vdots & \vdots  & & \ddots  & \ddots \\
	\Phi_{N-1}^1 & \Phi_{N-1}^2 & ... & ... & \Phi_{N-1} & I
\end{bmatrix}.
\end{align}
Now suppose we coarsen in time by a factor of $k$, that is, every $k$th time-point we denote a C-point,
and the $k-1$ points between each C-point are considered F-points (it is not necessary that $k$ be
fixed across the domain, rather this is a simplifying assumption for derivations and notation).
Then, using the inverse in \eqref{eq:geninv} and analogous matrix derivations as in \cite{southworth19},
we can eliminate F-points from \eqref{eq:system2} and arrive at a Schur complement
of $A$ over C-points, given by
\begin{align}\label{eq:schur}
A_\Delta = \begin{bmatrix} I \\ -\Phi_k^1 & I \\ & -\Phi_{2k}^{k+1} & I \\ & & \ddots & \ddots \\
	& & & -\Phi_{(N_c-1)k}^{(N_c-2)k+1} & I \end{bmatrix}.
\end{align}
Notice that the Schur-complement coarse-grid operator in the time-dependent case does exactly what it
does in the time-independent case: it takes $k$ steps on the fine grid, in this case using the appropriate
sequence of time-dependent operators.

A Schur complement arises naturally in reduction-based methods when we eliminate certain
degrees-of-freedom (DOFs). In this case, even computing the action of the Schur complement
\eqref{eq:schur}
is too expensive to be considered tractable. Thus, parallel-in-time integration is based on
a Schur complement approximation, where we let $\Psi_i$ denote a non-Galerkin approximation
to $\Phi_{ik}^{(i-1)k+1}$ and define the ``coarse-grid'' time integration operator 
$B_\Delta\approx A_\Delta$ as
\begin{align*}
B_\Delta = \begin{bmatrix} I \\ -\Psi_1 & I \\ & -\Psi_2 & I \\ & & \ddots & \ddots \\
	& & & -\Psi_{N_c-1} & I \end{bmatrix}.
\end{align*}

Convergence of iterative methods is typically considered by analyzing the error and
residual-propagation operators, say $\mathcal{E}$ and $\mathcal{R}$. Necessary and
sufficient conditions for convergence of an iterative method are that $\mathcal{E}^p,
\mathcal{R}^p\to 0$ with $p$. In this case, eigenvalues of $\mathcal{E}$ and $\mathcal{R}$
are a poor measure of convergence, so we consider the $\ell^2$-norm. For notation, assume
we block partition $A$ \eqref{eq:system2} into C-points and F-points and reorder
$A \mapsto \begin{bmatrix}A_{ff} & A_{fc} \\ A_{cf} & A_{cc}\end{bmatrix}$. Letting
subscript $F$ denote F-relaxation and subscript FCF denote FCF-relaxation, error and
residual propagation operators for two-level Parareal/MGRiT are derived in
\cite{southworth19} to be:
\begin{align}\label{eq:err_res}
\begin{split}
\mathcal{E}_F^p & = \begin{bmatrix} -A_{ff}^{-1}A_{fc} \\ I \end{bmatrix}
	\begin{bmatrix} \mathbf{0} & ({I - B_\Delta^{-1}A_\Delta} )^p \end{bmatrix}, \\
\mathcal{E}_{FCF}^p & = \begin{bmatrix} -A_{ff}^{-1}A_{fc} \\ I \end{bmatrix}
	\begin{bmatrix} \mathbf{0} & \left({(I - B_\Delta^{-1}A_\Delta })
		(I - A_\Delta)\right)^p \end{bmatrix}, \\
\mathcal{R}_F^p & = \begin{bmatrix} \mathbf{0} \\ (I - A_\Delta B_\Delta^{-1})^p \end{bmatrix}
	\begin{bmatrix} -A_{cf}A_{ff}^{-1} & I \end{bmatrix}, \\
\mathcal{R}_{FCF}^p & = \begin{bmatrix}\mathbf{0} \\ 
		\left((I - A_\Delta B_\Delta^{-1})(I - A_\Delta)\right)^p\end{bmatrix}
	\begin{bmatrix} -A_{cf}A_{ff}^{-1} & I \end{bmatrix}.
\end{split}
\end{align}
In \cite{southworth19}, the leading terms involving $A_{cf}A_{ff}^{-1}$ and
$A_{ff}^{-1}A_{fc}$ (see \cite{southworth19} for representation in $\Phi$) are
shown to be bounded in norm $\leq k$. Thus, as we iterate $p > 1$, convergence of
error and residual propagation operators is defined by iterations on the coarse
space, e.g., $({I - B_\Delta^{-1}A_\Delta} )$ for $\mathcal{E}_F$. 
To that end, define
\begin{align}
\begin{split}
\widetilde{\mathcal{E}}_F & = {I - B_\Delta^{-1}A_\Delta} , \hspace{5ex}
\widetilde{\mathcal{E}}_{FCF} = (I - B_\Delta^{-1}A_\Delta)(I - A_\Delta), \\
\widetilde{\mathcal{R}}_F & = I - A_\Delta B_\Delta^{-1}, \hspace{5ex}
\widetilde{\mathcal{R}}_{FCF} = (I - A_\Delta B_\Delta^{-1})(I - A_\Delta).
\end{split}
\end{align}

\subsection{A closed form for $\|\widetilde{\mathcal{E}}\|$ and $\|\widetilde{\mathcal{R}}\|$}
\label{sec:simple:norm}

\subsubsection{F-relaxation:} \label{sec:simple:norm:F}

Define the shift operators and block diagonal matrix
{\small
\begin{align*}
I_L & = \begin{bmatrix} \mathbf{0} \\ I & \mathbf{0} \\ & \ddots & \ddots \\ && I & \mathbf{0} \end{bmatrix},
\hspace{1ex}
I_z & = \begin{bmatrix} I & \\ & \ddots  \\ && I  \\ &&& \mathbf{0} \end{bmatrix},
\hspace{1ex}
D = \begin{bmatrix} \Phi_k^1 - \Psi_1 \\ & \ddots \\ &&  \Phi_{(N_c-1)k}^{(N_c-2)k+1} - \Psi_{N_c-1} \\
	&&& \mathbf{0} \end{bmatrix}
\end{align*}
}and notice that $I_LD = (B_\Delta - A_\Delta)$ and $I_L^TI_L = I_z$. Further define $\widetilde{D}$
and $\widetilde{B}_\Delta$ as the leading principle submatrices of $D$ and $B_\Delta$, respectively,
obtained by eliminating the last (block) row and column, corresponding to the final coarse-grid
time step, $N_c-1$.

Now note that $\widetilde{\mathcal{R}}_F := I-A_\Delta B_\Delta^{-1} =
(B_\Delta - A_\Delta)B_\Delta^{-1} = I_LDB_\Delta^{-1}$ and $I_L^TI_L = I_z$. Then,
\begin{align}\label{eq:Rf0}
\|\mathcal{R}_F\|^2 & = \sup_{\mathbf{x}\neq\mathbf{0}} \frac{\langle I_LDB_\Delta^{-1} \mathbf{x},
	I_LDB_\Delta^{-1}\mathbf{x}\rangle}{\langle \mathbf{x}, \mathbf{x}\rangle}
= \sup_{\mathbf{x}\neq\mathbf{0}} \frac{\langle I_zDB_\Delta^{-1} \mathbf{x},
	I_zDB_\Delta^{-1}\mathbf{x}\rangle}{\langle \mathbf{x}, \mathbf{x}\rangle}.
\end{align}
Because $DB_\Delta^{-1}$ is lower triangular, setting the last row to zero by multiplying
by $I_z$ also sets the last column to zero, in which case $I_zDB_\Delta^{-1} = 
I_zDB_\Delta^{-1}I_z$. Continuing from \eqref{eq:Rf0}, we have
\begin{align*}
\|\mathcal{R}_F\|^2
& =  \sup_{\mathbf{y}\neq\mathbf{0}} \frac{\langle \widetilde{D}\widetilde{B}_\Delta^{-1}\mathbf{y},
	\widetilde{D}\widetilde{B}_\Delta^{-1}\mathbf{y}\rangle}{\langle \mathbf{y}, \mathbf{y}\rangle}
= \|\widetilde{D}\widetilde{B}_\Delta^{-1}\|^2,
\end{align*}
where $\mathbf{y}$ is defined on the lower-dimensional space corresponding to the
operators $\widetilde{D}$ and $\widetilde{B}_\Delta$. Recalling that the $\ell^2$-norm
is defined by $\|A\| = \sigma_{\max}(A) = 1/\sigma_{\min}(A^{-1})$, for maximum and minimum
singular values, respectively,
\begin{align*}
\|\widetilde{\mathcal{R}}_F\| & = \sigma_{\max}\left(\widetilde{D}\widetilde{B}_\Delta^{-1}\right)
	= \max_{\mathbf{x}\neq\mathbf{0}} \frac{\|\widetilde{D}\mathbf{x}\|}{\|\widetilde{B}_\Delta\mathbf{x}\|}
	= \frac{1}{\sigma_{\min}\left(\widetilde{B}_\Delta\widetilde{D}^{-1}\right)},
\end{align*}
where
\begin{align*}
\widetilde{B}_\Delta\widetilde{D}^{-1}:&= \begin{bmatrix} I \\ -\Psi_1 & I \\ & \ddots & \ddots \\
		&& -\Psi_{N_c-2} & I \end{bmatrix}
	\begin{bmatrix} \left(\Phi_{k}^{1}-\Psi_{1}\right)^{-1} \\ 
	& \ddots \\ && \left(\Phi_{(N_c-1)k}^{(N_c-2)k+1}-\Psi_{N_c-1}\right)^{-1} \end{bmatrix}.
\end{align*}

Similarly, $\widetilde{\mathcal{E}}_F := I - B_\Delta^{-1}A_\Delta = B_\Delta^{-1}I_LD$.
Define $\widehat{B}_\Delta$ as the principle submatrix of $B_\Delta$ obtained by eliminating
the \textit{first} row and column. Similar arguments as above then yield
\begin{align*}
\|\widetilde{\mathcal{E}}_F\| & = \sigma_{\max}\left(\widehat{B}_\Delta^{-1}\widetilde{D}\right)
	= \max_{\mathbf{x}\neq\mathbf{0}} \frac{\|\widehat{B}_\Delta^{-1}\mathbf{x}\|}
		{\|\widetilde{D}^{-1}\mathbf{x}\|}
	= \frac{1}{\sigma_{\min}\left(\widetilde{D}^{-1}\widehat{B}_\Delta\right)},
\end{align*}
where
\begin{align*}
\widetilde{D}^{-1}\widehat{B}_\Delta:&= \begin{bmatrix} \left(\Phi_{k}^{1}-\Psi_{1}\right)^{-1} \\ 
	& \ddots \\ && \left(\Phi_{(N_c-1)k}^{(N_c-2)k+1}-\Psi_{N_c-1}\right)^{-1} \end{bmatrix}
	\begin{bmatrix} I \\ -\Psi_2 & I \\ & \ddots & \ddots \\
		&& -\Psi_{N_c-1} & I \end{bmatrix}.
\end{align*}

\subsubsection{FCF-relaxation:} \label{sec:simple:norm:F}

Adding the effects of (pre)FCF-relaxation, $\mathcal{E}_{FCF} =
B_\Delta^{-1}(B_\Delta - A_\Delta)(I - A_\Delta)$, where
$(B_\Delta - A_\Delta)(I - A_\Delta)$ is given by the block diagonal matrix
\begin{align*}
= I_L^2 \begin{bmatrix} \left(\Phi_k^1 - \Psi_L\right)\Phi_k^1 \\ & \ddots \\ &&  
		\left(\Phi_{(N_c-2)k}^{(N_c-3)k+1}-\Psi_{N_c-2}\right)\Phi_{(N_c-2)k}^{(N_c-3)k+1}\\
	&&& \mathbf{0} \\ &&&&\mathbf{0} \end{bmatrix}.
\end{align*}
Again analogous arguments as for F-relaxation can pose this as a problem on a nonsingular
matrix of reduced dimensions. Let
\begin{align*}
&\widetilde{D}_{fcf}^{-1}\overline{B}_\Delta := \\
&	\begin{bmatrix} \left(\Phi_k^1 - \Psi_L\right)\Phi_k^1 \\ & \ddots \\ &&  
		\left(\Phi_{(N_c-2)k}^{(N_c-3)k+1}-\Psi_{N_c-2}\right)\Phi_{(N_c-2)k}^{(N_c-3)k+1}\end{bmatrix}^{-1}
	\begin{bmatrix} I \\ -\Psi_3 & I \\ & \ddots & \ddots \\
		&& -\Psi_{N_c-1} & I \end{bmatrix}.
\end{align*}
Then
\begin{align*}
\|\widetilde{\mathcal{E}}_{FCF}\| & = \sigma_{\max}\left(\overline{B}_\Delta^{-1}\widetilde{D}_{fcf}\right)
	= \max_{\mathbf{x}\neq\mathbf{0}} \frac{\|\overline{B}_\Delta^{-1}\mathbf{x}\|}
		{\|\widetilde{D}_{fcf}^{-1}\mathbf{x}\|}
 	= \frac{1}{\sigma_{\min}\left(\widetilde{D}_{fcf}^{-1}\overline{B}_\Delta\right)}.
\end{align*}
Here we have an interesting thing to note -- in adding FCF-relaxation, our coarse-grid
approximation, e.g., $(\Phi_k^1 - \Psi_1)$ must be accurate with respect to the later in time 
operator $I - \Psi_3$. The case of FCF relaxation with residual propagation produces operators with a more complicated definition, in particular, requiring diagonal scalings on both sides of $I-A_\Delta B_\Delta^-1$. The analysis, however, proceeds in a similar fashion. Moreover, considering only post-FCF relaxation again shifts all the scalings on a single side.

\subsection{Convergence and the temporal approximation property}\label{sec:theory:tap}

Now assume that $\Phi_i=\Phi_j$ and $\Psi_i=\Psi_j$ for all $i,j$, that is,
$\Phi$ and $\Psi$ are independent of time. Then, the block matrices derived in
the previous section defining convergence of two-level MGRiT all fall under
the category of block-Toeplitz matrices. By appealing to block-Toepliz theory, 
in \cite{southworth19}, tight bounds are derived on the appropriate minimum
and maximum singular values appearing in the previous section, which are
exact as the number of coarse-grid time steps $\to\infty$.
The fundamental concept is the ``temporal approximation property'' (TAP), as introduced
below, which provides necessary and sufficient conditions for two-level convergence
of Parareal and MGRiT. Moreover, the constant with which the TAP is
satisfied, $\varphi_F$ or $\varphi_{FCF}$, provides a tight upper bound
on convergence factors, that is asymptotically exact as $N_c\to\infty$.
We present a simplified/condensed version of the theoretical results
derived in \cite{southworth19} in \Cref{th:conv}.

\begin{definition}[Temporal approximation property]
Let $\Phi$ denote a fine-grid time-stepping operator and $\Psi$ denote a coarse-grid time-stepping
operator, for all time points, with coarsening factor $k$. Then, $\Phi$ satisfies an F-relaxation
temporal approximation property (F-TAP), with respect to $\Psi$, with constant $\varphi_{F}$,
if, for all vectors $\mathbf{v}$,
\begin{align}\label{eq:tap_f}
\|(\Psi - \Phi^k)\mathbf{v}\| \leq \varphi_{F} \left[\min_{x\in[0,2\pi]}
	\left\| (I - e^{\mathrm{i}x}\Psi )\mathbf{v}\right\| \right].
\end{align}
Similarly, $\Phi$ satisfies an FCF-relaxation temporal approximation property (FCF-TAP), with
respect to $\Psi$, with constant $\varphi_{FCF}$, if, for all vectors $\mathbf{v}$,
\begin{align}\label{eq:tap_fcf}
\|(\Psi - \Phi^k)\mathbf{v}\| \leq \varphi_{FCF}\left[\min_{x\in[0,2\pi]}
	\left\| (\Phi^{-k}(I - e^{\mathrm{i}x}\Psi) )\mathbf{v}\right\| \right].
\end{align}
\end{definition}

\begin{theorem}[Necessary and sufficient conditions]\label{th:conv}
Suppose $\Phi$ and $\Psi$ are linear, stable ($\|\Phi^p\|,\|\Psi^p\| < 1$ for some $p$), and independent of time; and that $(\Psi-\Phi^k)$ is invertible. Further suppose that $\Phi$ satisfies an F-TAP with respect
to $\Psi$, with constant $\varphi_{F}$, and $\Phi$ satisfies an FCF-TAP with respect to $\Psi$, with constant
$\varphi_{FCF}$. Then, \textit{worst-case convergence} of residual is exactly bounded by
\begin{align*}
\frac{\varphi_{F}}{1+O(1/\sqrt{N_c})} & \leq
	\frac{\|\mathbf{r}_{i+1}^{(F)}\|}{\|\mathbf{r}_{i}^{(F)}\|} < \varphi_{F}, \\
\frac{\varphi_{FCF}}{1+O(1/\sqrt{N_c})} & \leq
	\frac{\|\mathbf{r}_{i+1}^{(FCF)}\|}{\|\mathbf{r}_{i}^{(FCF)}\|} < \varphi_{FCF}
\end{align*}
for iterations $i > 1$ (i.e., not the first iteration).
\end{theorem}

Broadly, the TAP defines how well $k$ steps of the the fine-grid time-stepping
operator, $\Phi^k$, must approximate the action of the coarse-grid operator,
$\Psi$, for convergence.\footnote{
There are some nuances regarding error vs. residual and powers of operators/multiple
iterations. We refer the interested reader to
\cite{southworth19} for details.}
An interesting property of the TAP is the introduction of a complex scaling of $\Psi$,
even in the case of real-valued operators. If we suppose $\Psi$ has imaginary eigenvalues,
the minimum over $x$ can be thought of as rotating this eigenvalue to the real axis.
Recall from \cite{southworth19}, the TAP for
a fixed $x_0$ can be computed as the largest generalized singular value of
$\left\{ \Psi - \Phi^k, I - e^{\mathrm{i}x_0}\Psi\right\}$ or, equivalently, the largest
singular value of $(\Psi - \Phi^k)(I - e^{\mathrm{i}x_0}\Psi)^{-1}$. Although there
are methods to compute this directly or iteratively, e.g., \cite{zwaan:2017}, minimizing
the TAP for all $x\in[0,2\pi]$ is expensive and often impractical. The following lemma
and corollaries introduce a closed form for the minimum over $x$, and simplified
sufficient conditions for convergence that do not require a minimization or complex
operators.

\begin{lemma}\label{lem:complex}
Suppose $\Psi$ is real-valued. Then,
\begin{align}
\min_{x\in[0,2\pi]} \|(I - e^{\mathrm{i}x}\Psi)\mathbf{v}\|^2 & =
	 \|\mathbf{v}\|^2 + \|\Psi\mathbf{v}\|^2 -
	 	2\left|\langle\Psi\mathbf{v},\mathbf{v} \rangle\right|, \label{eq:TAPmin} \\
\min_{x\in[0,2\pi]} \|\Phi^{-k}(I - e^{\mathrm{i}x}\Psi)\mathbf{v}\|^2 & =
	 \|\Phi^{-k}\mathbf{v}\|^2 + \|\Phi^{-k}\Psi\mathbf{v}\|^2 -
	 	2\left|\langle\Phi^{-k}\Psi\mathbf{v},\Phi^{-k}\mathbf{v} \rangle\right|. \nonumber
\end{align}
\end{lemma}
\begin{proof}
Consider satisfying the TAP for real-valued operators and complex vectors. Expanding in inner products,
the TAP is given by
\begin{align*}
\min_{x\in[0,2\pi]} \|(I - e^{\mathrm{i}x}\Psi)\mathbf{v}\|^2 & = \min_{x\in[0,2\pi]}  \|\mathbf{v}\|^2 + \|\Psi\mathbf{v}\|^2 -
  e^{\mathrm{i}x}\langle\Psi\mathbf{v},\mathbf{v}\rangle - e^{-\mathrm{i}x}\langle\mathbf{v},\Psi\mathbf{v}\rangle.
\end{align*}
Now decompose $\mathbf{v}$ into real and imaginary components, $\mathbf{v} = \mathbf{v}_r
+ \mathrm{i}\mathbf{v}_i$ for $\mathbf{v}_i,\mathbf{v}_r\in\mathbb{R}^n$, and note that
\begin{align*}
\langle\Psi\mathbf{v},\mathbf{v} \rangle & = \langle\Psi\mathbf{v}_i,\mathbf{v}_i \rangle +
  \langle\Psi\mathbf{v}_r,\mathbf{v}_r \rangle + \mathrm{i}\langle\Psi\mathbf{v}_i,\mathbf{v}_r \rangle -
  \mathrm{i}\langle\Psi\mathbf{v}_r,\mathbf{v}_i \rangle := \mathcal{R} - \mathrm{i}\mathcal{I}, \\
\langle\mathbf{v},\Psi\mathbf{v} \rangle & = \langle\mathbf{v}_i,\Psi\mathbf{v}_i \rangle +
  \langle\mathbf{v}_r,\Psi\mathbf{v}_r \rangle + \mathrm{i}\langle\mathbf{v}_i,\Psi\mathbf{v}_r \rangle -
  \mathrm{i}\langle\mathbf{v}_r,\Psi\mathbf{v}_i \rangle := \mathcal{R} + \mathrm{i}\mathcal{I}.
\end{align*}
Expanding with $e^{\mathrm{i}x} = \cos(x) + \mathrm{i}\sin(x)$ and
$e^{-\mathrm{i}x} = \cos(x) - \mathrm{i}\sin(x)$ yields
\begin{align}\label{eq:sincos}
\begin{split}
e^{\mathrm{i}x} \langle\Psi\mathbf{v},\mathbf{v} \rangle + e^{-\mathrm{i}x}\langle\mathbf{v},\Psi\mathbf{v} \rangle & =
  2\cos(x)\mathcal{R} + 2\sin(x)\mathcal{I}.
\end{split}
\end{align}
To minimize in $x$, differentiate and set the derivative equal to zero to obtain roots $\{x_0,x_1\}$, where
$x_0 = \arctan \left( \mathcal{I} / \mathcal{R}\right)$ and $x_1 = x_0 + \pi$. Plugging in above yields
\begin{align*}
\min_{x\in[0,2\pi]} \pm \left(e^{\mathrm{i}x} \langle\Psi\mathbf{v},\mathbf{v} \rangle + e^{-\mathrm{i}x}\langle\mathbf{v},\Psi\mathbf{v} \rangle\right)
  & = -2\sqrt{\mathcal{R}^2 + \mathcal{I}^2} \\
&\hspace{-20ex} = -2 \sqrt{\left(\langle\mathbf{v}_i,\Psi\mathbf{v}_i \rangle +
	\langle\mathbf{v}_r,\Psi\mathbf{v}_r \rangle\right)^2 +
    \left(\langle\Psi\mathbf{v}_i,\mathbf{v}_r \rangle - 
    \langle\Psi\mathbf{v}_r,\mathbf{v}_i \rangle\right)^2 } \\
&\hspace{-20ex} = -2\sqrt{\left(\langle\Psi\mathbf{v},\mathbf{v} \rangle\right)^*
	\langle\Psi\mathbf{v},\mathbf{v} \rangle} \\
&\hspace{-20ex} = -2\left|\langle\Psi\mathbf{v},\mathbf{v} \rangle\right|.
\end{align*}
Then,
\begin{align*}
\min_{x\in[0,2\pi]} \|(I - e^{\mathrm{i}x}\Psi)\mathbf{v}\|^2 & = \min_{x\in[0,2\pi]}  \|\mathbf{v}\|^2 + \|\Psi\mathbf{v}\|^2 -
  e^{\mathrm{i}x}\langle\Psi\mathbf{v},\mathbf{v}\rangle - e^{-\mathrm{i}x}\langle\mathbf{v},\Psi\mathbf{v}\rangle \\
& =   \|\mathbf{v}\|^2 + \|\Psi\mathbf{v}\|^2 -2\left|\langle\Psi\mathbf{v},\mathbf{v} \rangle\right|.
\end{align*}
Analogous derivations hold for the FCF-TAP with a factor of $\Phi^{-k}$.
\end{proof}

\begin{corollary}[Symmetry in $x$]\label{cor:symm}
For real-valued operators, the TAP is symmetric in $x$ when considered over all $\mathbf{v}$,
that is, it is sufficient to consider ${x\in[0,\pi]}$.
\end{corollary}
\begin{proof}
From the above proof, suppose that $\mathbf{v} := \mathbf{v}_r + \mathrm{i}\mathbf{v}_i$
is minimized by $x_0 = \arctan(\mathcal{I}/\mathcal{R})$. Then, swap the sign on
$\mathbf{v}_i \mapsto \hat{\mathbf{v}}:= \mathbf{v}_r - \mathrm{i}\mathbf{v}_i$,
which yields $\hat{\mathcal{I}} = -\mathcal{I}$ and $\hat{\mathcal{R}} = \mathcal{R}$,
and $\hat{\mathbf{v}}$ is minimized at $\hat{x}_0 = \arctan(\hat{\mathcal{I}}/\hat{\mathcal{R}})
= \arctan(-\mathcal{I}/\mathcal{R}) = -\arctan(\mathcal{I}/\mathcal{R}) = -x_0$. 
Further note the equalities
$\min_{x\in[0,2\pi]} \|(I - e^{\mathrm{i}x}\Psi)\mathbf{v}\| = 
\min_{x\in[0,2\pi]} \|(I - e^{\mathrm{i}x}\Psi)\hat{\mathbf{v}}\|$
and $\|(\Psi - \Phi^k)\mathbf{v}\| = \|(\Psi - \Phi^k)\hat{\mathbf{v}}\|$ and, thus,
$\mathbf{v}$ and $\hat{\mathbf{v}}$ satisfy the F-TAP with the same constant, and
$x$-values $x_0$ and $-x_0$. Similar derivations hold for the FCF-TAP.
\end{proof}

\begin{corollary}[A sufficient condition for the TAP]
For real-valued operators, sufficient conditions to satisfy the F-TAP and FCF-TAP,
respectively, are that for all vectors $\mathbf{v}$,
\begin{align}\label{eq:tap_approx_f}
\|(\Psi - \Phi^k)\mathbf{v}\| & \leq \varphi_{F} ( \|\mathbf{v}\| - \|\Psi\mathbf{v}\| ), \\
\|(\Psi - \Phi^k)\Phi^k\mathbf{v}\| & \leq \varphi_{FCF} ( \|\mathbf{v}\| -
	 \|\Phi^{-k}\Psi\Phi^k\mathbf{v}\| ).\label{eq:tap_approx_fcf}
\end{align}
\end{corollary}
\begin{proof}
Note 
\begin{align*}
\min_{x\in[0,2\pi]} \|(I - e^{\mathrm{i}x}\Psi)\mathbf{v}\|^2 
& \geq \|\mathbf{v}\|^2 + \|\Psi\mathbf{v}\|^2 -2\|\Psi\mathbf{v}\|\|\mathbf{v}\| \\
& = \left(\|\mathbf{v}\| - \|\Psi\mathbf{v}\|\right)^2.
\end{align*}
Then, $\|(\Psi - \Phi^k)\mathbf{v}\| \leq \varphi_{F} ( \|\Psi\mathbf{v}\| - \|\mathbf{v}\| )
\leq \varphi_F \min_{x\in[0,2\pi]} \|(I - e^{\mathrm{i}x}\Psi)\mathbf{v}\|$. Similar
derivations hold for $\min_{x\in[0,2\pi]} \|\Phi^{-k}(I - e^{\mathrm{i}x}\Psi)\mathbf{v}\|$.
\end{proof}

For all practical purposes, a computable approximation to the TAP is sufficient, because the underlying purpose
is to understand convergence of Parareal and MGRIT, and help pick or develop effective coarse-grid propagators.
To that end, one can approximate the TAP by only considering it for specific $x$. For example, one
could only consider real-valued $e^{\mathrm{i}x}$, with $x \in\{0,\pi\}$ (or, equivalently, only real-valued $\mathbf{v}$).
Let $\Psi = \Psi_s + \Psi_k$, where $\Psi_s := (\Psi + \Psi^T)/2$ and $\Psi_k := (\Psi - \Psi^T)/2$ are the symmetric
and skew-symmetric parts of $\Psi$. Then, from \Cref{lem:complex} and Remark 4 in \cite{southworth19}, the TAP
restricted to $x\in\{0,\pi\}$ takes the simplified form
\begin{align}
\min_{x\in\{0,\pi\}} \|(I - e^{\mathrm{i}x}\Psi)\mathbf{v}\|^2 & =
  \|\mathbf{v}\|^2 + \|\Psi\mathbf{v}\|^2 -2\left|\langle\Psi_s\mathbf{v},\mathbf{v} \rangle\right|\nonumber \\
& =  \begin{cases} \|(I + \Psi)\mathbf{v}\|^2  &
  \text{if } \langle \Psi\mathbf{v},\mathbf{v}\rangle \leq 0 \\
  \|(I - \Psi)\mathbf{v}\|^2 & \text{if } \langle \Psi\mathbf{v},\mathbf{v}\rangle > 0 \end{cases}. \label{eq:tap_symm}
\end{align}
Here, we have approximated the numerical range $|\langle\Psi\mathbf{v},\mathbf{v}\rangle|$ in \eqref{eq:TAPmin}
with the numerical range restricted to the real axis (that is, the numerical range of the symmetric component of $\Psi$).
If $\Psi$ is symmetric, $\Psi_s = \Psi$, $\Psi$ is unitarily diagonalizable, and the eigenvalue-based convergence bounds
of \cite{southworth19} immediately pop out from \eqref{eq:tap_symm}. Because the numerical range is convex and
symmetric across the real axis, \eqref{eq:tap_symm} provides a reasonable approximation when $\Psi$ has a
significant symmetric component.

Now suppose $\Psi$ is skew symmetric. Then $\Psi_s = \mathbf{0}$, and the numerical range lies exactly on the
imaginary axis (corresponding with the eigenvalues of a skew symmetric operator). This suggests an alternative
approximation to \eqref{eq:TAPmin} by letting $x \in\{\pi/2,3\pi/2\}$, which yields $e^{\mathrm{i}x} =\pm \mathrm{i}$.
Similar to above, this yields a simplified version of the TAP,
\begin{align}\label{eq:tap_skew}
\min_{x\in\{\pi/2,3\pi/2\}} \|(I - e^{\mathrm{i}x}\Psi)\mathbf{v}\|^2 & = \|\mathbf{v}\|^2 + \|\Psi\mathbf{v}\|^2 -2\left|\langle\Psi_k\mathbf{v},\mathbf{v} \rangle\right| \\
  & =  \begin{cases} \|(I + \mathrm{i}\Psi)\mathbf{v}\|^2  &
  \text{if } \langle \Psi\mathbf{v},\mathbf{v}\rangle \leq 0 \nonumber \\
    \|(I -  \mathrm{i}\Psi)\mathbf{v}\|^2 & \text{if } \langle \Psi\mathbf{v},\mathbf{v}\rangle > 0 \end{cases}.
\end{align}
Recall skew symmetric operators are normal and unitarily diagonalizable. Pulling out the eigenvectors in
\eqref{eq:tap_skew} and doing a maximum over eigenvalues again yields exactly the two-grid eigenvalue bounds
derived in \cite{southworth19}. In particular, the $e^{\mathrm{i}x}$ is a rotation of the purely imaginary eigenvalues
to the real axis, and corresponds to the denominator $1 - |\mu_i|$ of two-grid eigenvalue convergence bounds
\cite{southworth19}.

\subsection{Two-level results, and why multilevel is harder}\label{sec:theory:single}

\Cref{th:conv} covers the case of two-level convergence in a fairly general
(linear) setting. In practice, however, multilevel methods are often superior to
two-level methods, so a natural question is what these results mean in the
multilevel setting. Two-grid convergence usually provides a lower bound on
possible multilevel convergence factors. For MGRIT, it
is known that F-cycles can be advantageous, or even necessary, for fast,
scalable (multilevel) convergence \cite{Falgout:2014}, better approximating
a two-level method on each level of the hierarchy than a normal V-cycle.
However, because MGRIT uses non-Galerkin coarse grid operators, the relationship
between two-level and multilevel convergence is complicated, and it is not
obvious that two-grid convergence does indeed provide a lower bound on multilevel
convergence.

\Cref{th:conv} and results in \cite{Dobrev:2017,southworth19}, can be
interpreted in two ways. The first, and the interpretation used here, is that
the derived tight bounds on worst-case convergence factor hold for all but the
first iteration. In \cite{Dobrev:2017}, a different perspective is taken, where
bounds in \Cref{th:conv} hold for propagation of C-point error on all iterations.\footnote{In
\cite{Dobrev:2017} it is assumed that $\Phi$ and $\Psi$ are unitarily
diagonalizable, and upper bounds on convergence are derived. These bounds were
generalized in \cite{southworth19}, and shown to be tight in the unitarily
diagonalizable setting.} In the two-level setting,
either case provides necessary and sufficient conditions for convergence,
assuming that the other iteration is bounded (which it is \cite{southworth19}).

In the multilevel setting we are interested in convergence over
\textit{all} points for a \textit{single} iteration.
Consider a three-level MGRIT V-cycle, with levels $0, 1$, and
$2$. On level 1, one iteration of two-level MGRIT is applied as an approximate
residual correction for the problem on level 0. Suppose \Cref{th:conv} ensures
convergence for one iteration, that is, $\varphi<1$. Because we
are only performing one iteration, we must take the perspective that
\Cref{th:conv} ensures a decrease in C-point error, but a possible increase in
F-point error on level 1. However, if the \textit{total} error on level 1 has
increased, coarse-grid correction on level 0 interpolates a \textit{worse}
approximation to the desired (exact) residual correction than interpolating no
correction at all! If divergent behavior occurs somewhere in the middle of the
multigrid hierarchy, it is likely that the larger multigrid iteration will also
diverge. Given that multilevel convergence is usually worse than two-level in
practice, this motivates a stronger two-grid result that can ensure two-grid
convergence for all points on all iterations.

The following theorem introduces stronger variations in the TAP that provide
necessary conditions for a two-level method with F- and FCF-relaxation to
converge every iteration on C-points and F-points. \Cref{cor:new2grid}
strengthens this result in the case of simultaneous diagonalization of $\Phi$
and $\Psi$, deriving necessary and sufficient conditions for convergence, with
tight bounds in norm.

\begin{theorem}\label{th:new2grid}
    Let $\mathcal{E}_F$ and $\mathcal{E}_{FCF}$ denote error propagation of two-level MGRIT with F-relaxation
    and FCF-relaxation, respectively. Define
    $\mathcal{W}_F := \sqrt{\sum_{\ell=0}^{k-1} \Phi^\ell (\Phi^\ell)^*}$,
    $\mathcal{W}_{FCF} := \sqrt{\sum_{\ell=k}^{2k-1} \Phi^\ell(\Phi^\ell)^*}$, and $\varphi_F^{A^*A}$ and
    $\varphi_{FCF}^{A^*A}$ as the minimum constants such that, for all $\mathbf{v}$,
        \begin{align*}
        \| (\Psi - \Phi^{k}) \mathbf{v} \|
        &\leq \widehat{\varphi}_F \left[ \min_{x \in [0, 2 \pi]} \left\| \mathcal{W}_F^{-1} (I - e^{ix}\Psi) \mathbf{v} \right\| + O (1/\sqrt{N_c}) \right], \\
        \| (\Psi - \Phi^{k}) \mathbf{v} \|
        &\leq \widehat{\varphi}_{FCF} \left[ \min_{x \in [0, 2 \pi]} \left\|\mathcal{W}_{FCF}^{-1} (I - e^{ix} \Psi) \mathbf{v} \right\| + O(1/\sqrt{N_c}) \right].
    \end{align*}
        Then, $ \| \mathcal{R}_F \| \geq \widehat{\varphi}_{F}$ and $\| \mathcal{R}_{FCF} \|\geq
    \widehat{\varphi}_{FCF}$.
\end{theorem}
\begin{proof}
The proof follows the theoretical framework developed in \cite{southworth19} and can be found in the appendix.
\end{proof}
\begin{corollary}\label{cor:new2grid}
    Assume that $\Phi$ and $\Psi$ are simultaneously diagonalizable with eigenvectors $U$
    and eigenvalues $\{\lambda_i\}_i$ and $\{\mu_i\}_i$, respectively. Denote error- and residual-propagation operators
    of two-level MGRIT as $\mathcal{E}$ and $\mathcal{R}$, respectively, with subscripts indicating relaxation scheme.
    Let $\widetilde{U}$ denote a block-diagonal matrix with diagonal blocks given by $U$. Then,
        \begin{align}\label{eq:tight_2grid}
    \begin{split}
        \| \mathcal{R}_F \|_{(\widetilde{U} \widetilde{U}^*)^{-1}}
        = \|\mathcal{E}_F\|_{(\widetilde{U} \widetilde{U}^*)^{-1}}
        &= \max_i \sqrt{\frac{1-|\lambda_i|^{2 k}}{1 - |\lambda_i|^2}} \frac{|\mu_i - \lambda_i^{k}|}{(1 - |\mu_i|) + O(1/N_c)}, \\
        \| \mathcal{R}_{FCF} \|_{(\widetilde{U} \widetilde{U}^*)^{-1}}
        = \|\mathcal{E}_{FCF}\|_{(\widetilde{U} \widetilde{U}^*)^{-1}}
        &= \max_i \sqrt{\frac{1-|\lambda_i|^{2 k}}{1 - |\lambda_i|^2}} \frac{|\lambda_i|^{k}|\mu_i - \lambda_i^{k}|}{(1 - |\mu_i|) + O(1/N_c)}.
    \end{split}
    \end{align}
\end{corollary}
\begin{proof}
The proof follows the theoretical framework developed in \cite{southworth19} and can be found in the appendix.
\end{proof}
For a detailed discussion on the simultaneously diagonalizable assumption, see \cite{southworth19}.

Note in \Cref{th:new2grid} and \Cref{cor:new2grid}, there is an additional
scaling compared with results obtained on two-grid convergence for all but the
first iteration (see \Cref{th:conv} and \cite{southworth19}), which makes the
convergence bounds larger in all cases. This factor may be relevant to why
multilevel convergence is more difficult than two-level.
\Cref{fig:neumann} demonstrates the impact of this additional scaling by
plotting the difference between worst-case two-grid convergence on all points on
all iterations (\Cref{cor:new2grid}) 
vs. error on all points on all but one iteration (\Cref{th:conv}). Plots are a
function of $\delta t$ times the spatial eigenvalues in the complex plane.
Note, the color map is strictly positive because worst-case error propagation on
the first iteration is strictly greater than that on further iterations.

\begin{figure}[!th]
    \centering
    \begin{center}
        \begin{subfigure}[c]{0.25\textwidth}
            \includegraphics[height=\textwidth]{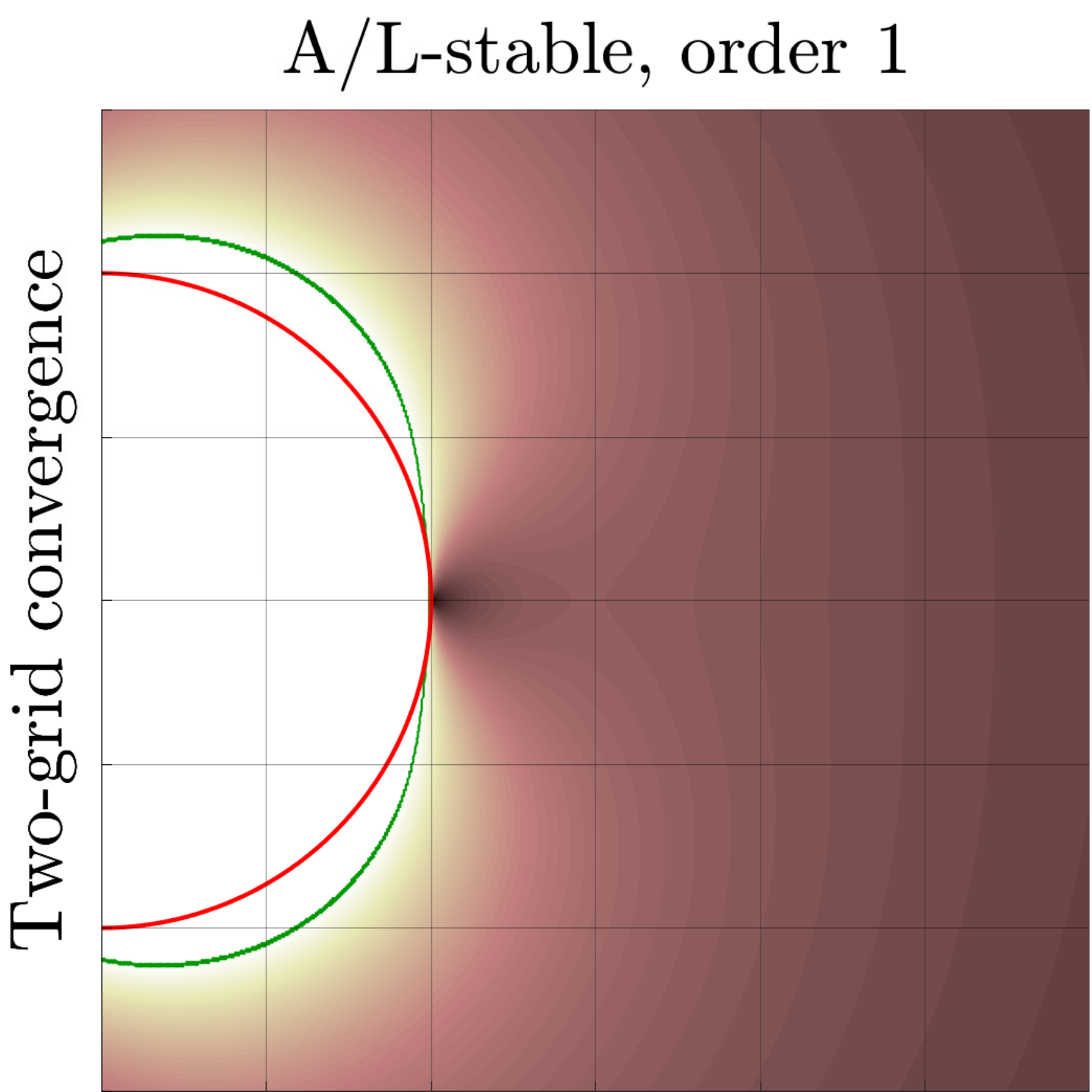}
        \end{subfigure}\hspace{1ex}
        \begin{subfigure}[c]{0.25\textwidth}
            \includegraphics[height=\textwidth]{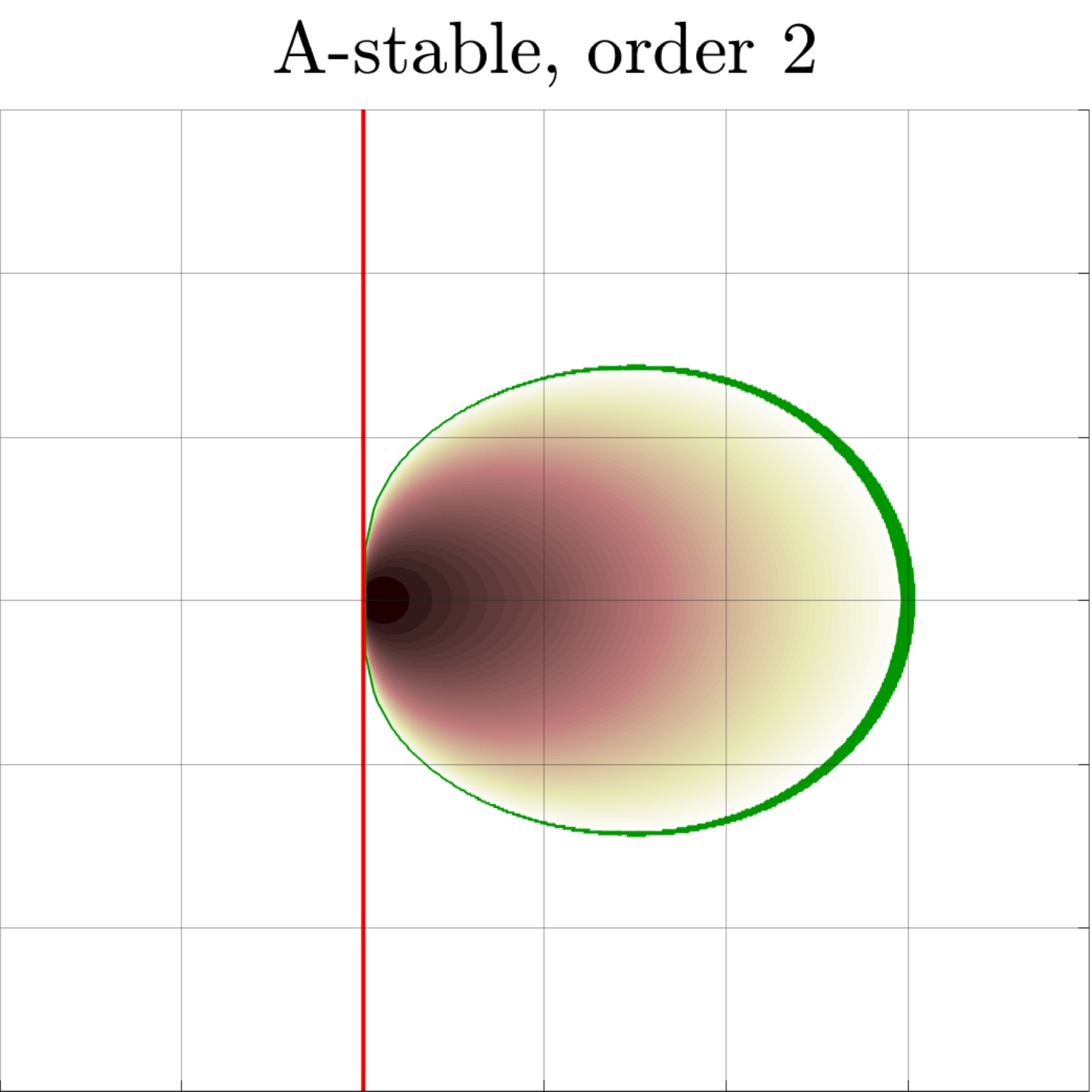}
        \end{subfigure}\hspace{1ex}
        \begin{subfigure}[c]{0.25\textwidth}
            \includegraphics[height=\textwidth]{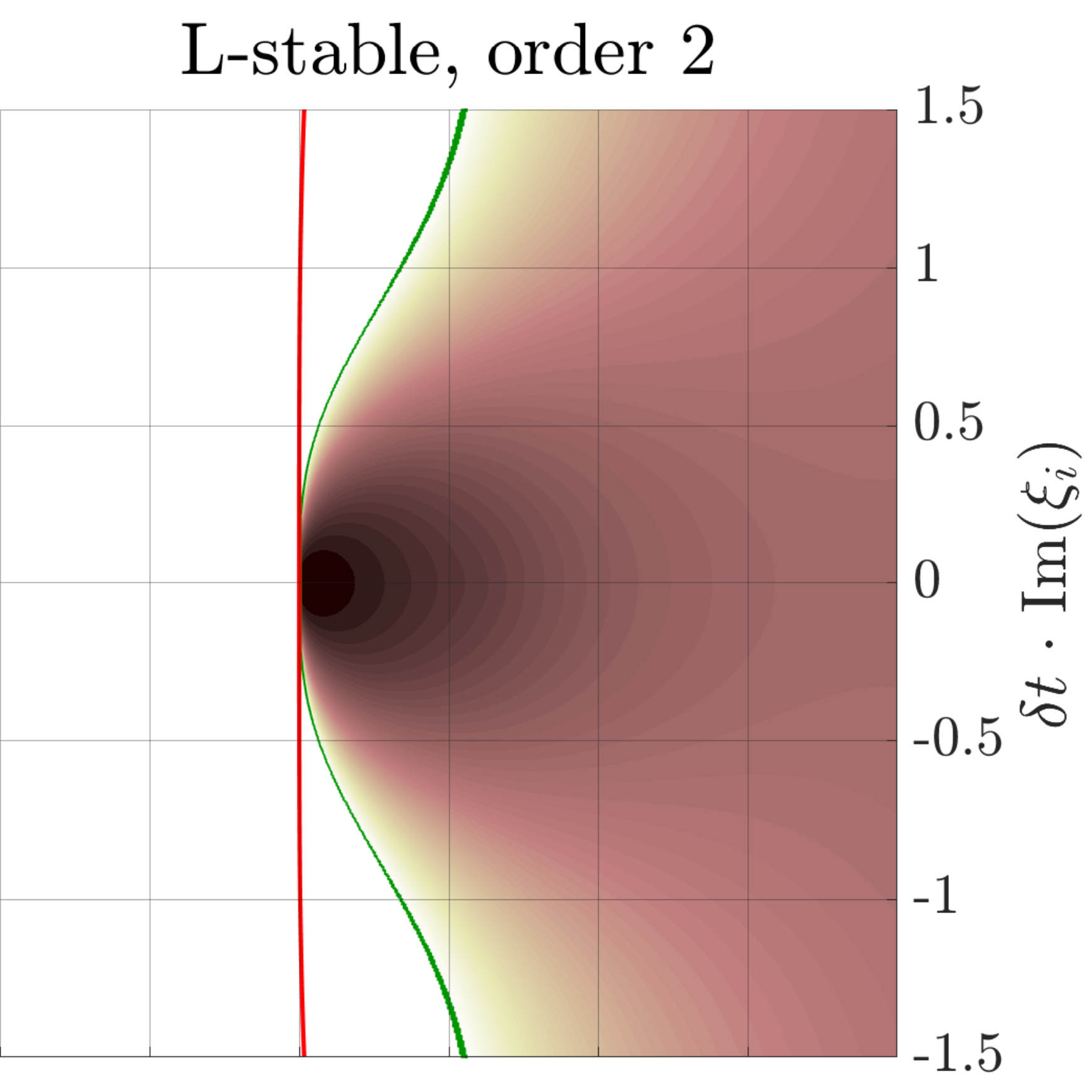}
        \end{subfigure}
    \end{center}
    \begin{center}
        \begin{subfigure}[c]{0.25\textwidth}
            \includegraphics[height=\textwidth]{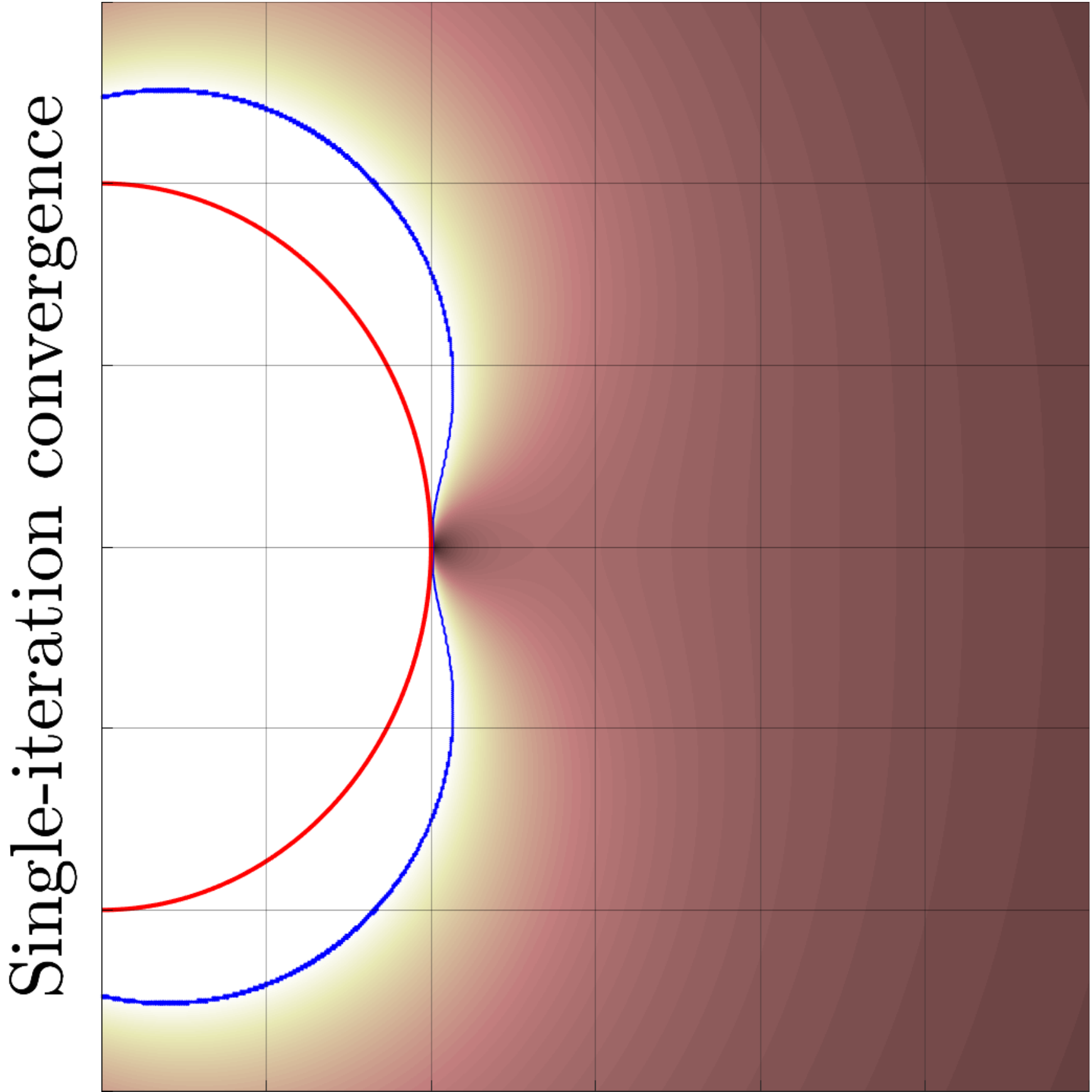}
        \end{subfigure}\hspace{1ex}
        \begin{subfigure}[c]{0.25\textwidth}
            \includegraphics[height=\textwidth]{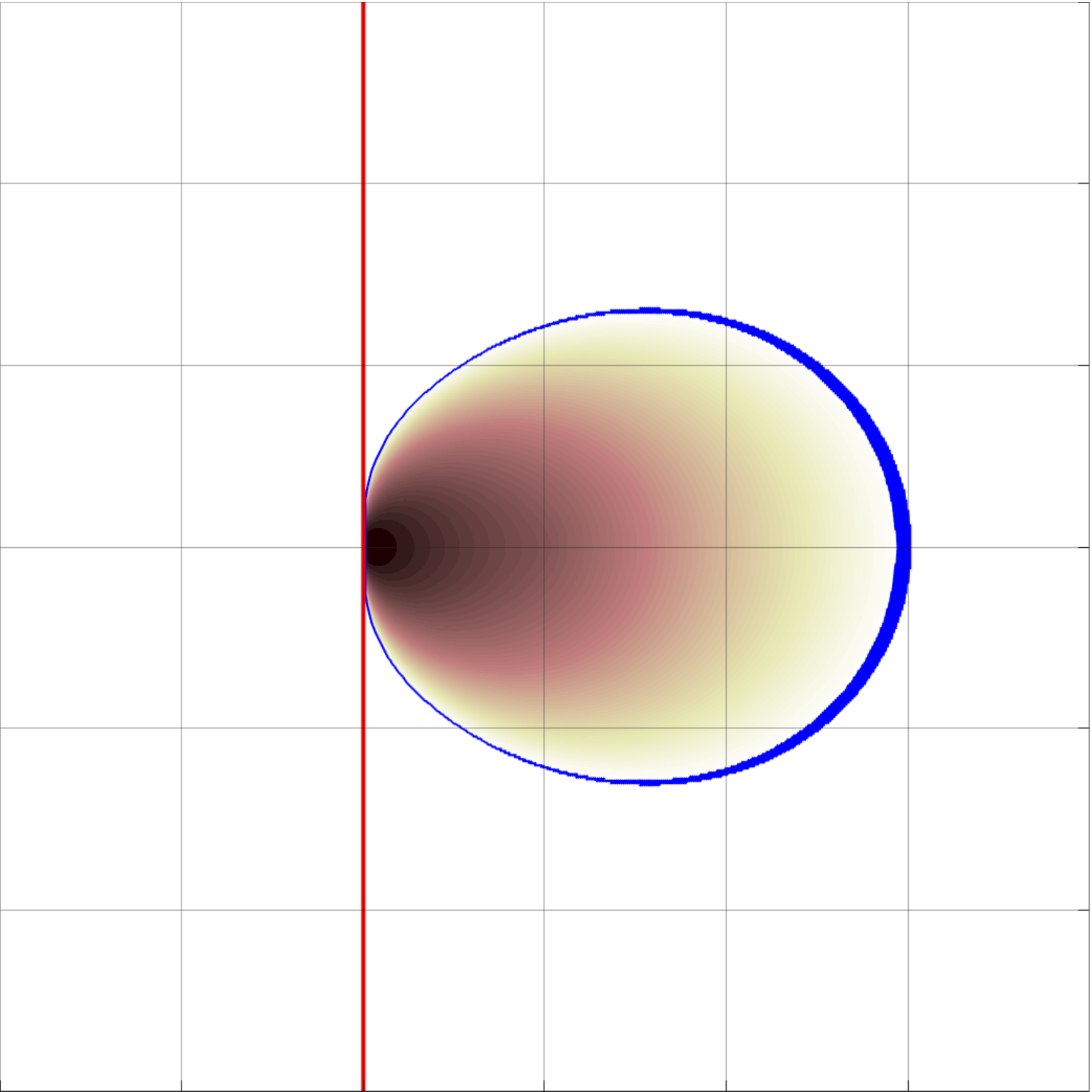}
        \end{subfigure}\hspace{1ex}
        \begin{subfigure}[c]{0.25\textwidth}
            \includegraphics[height=\textwidth]{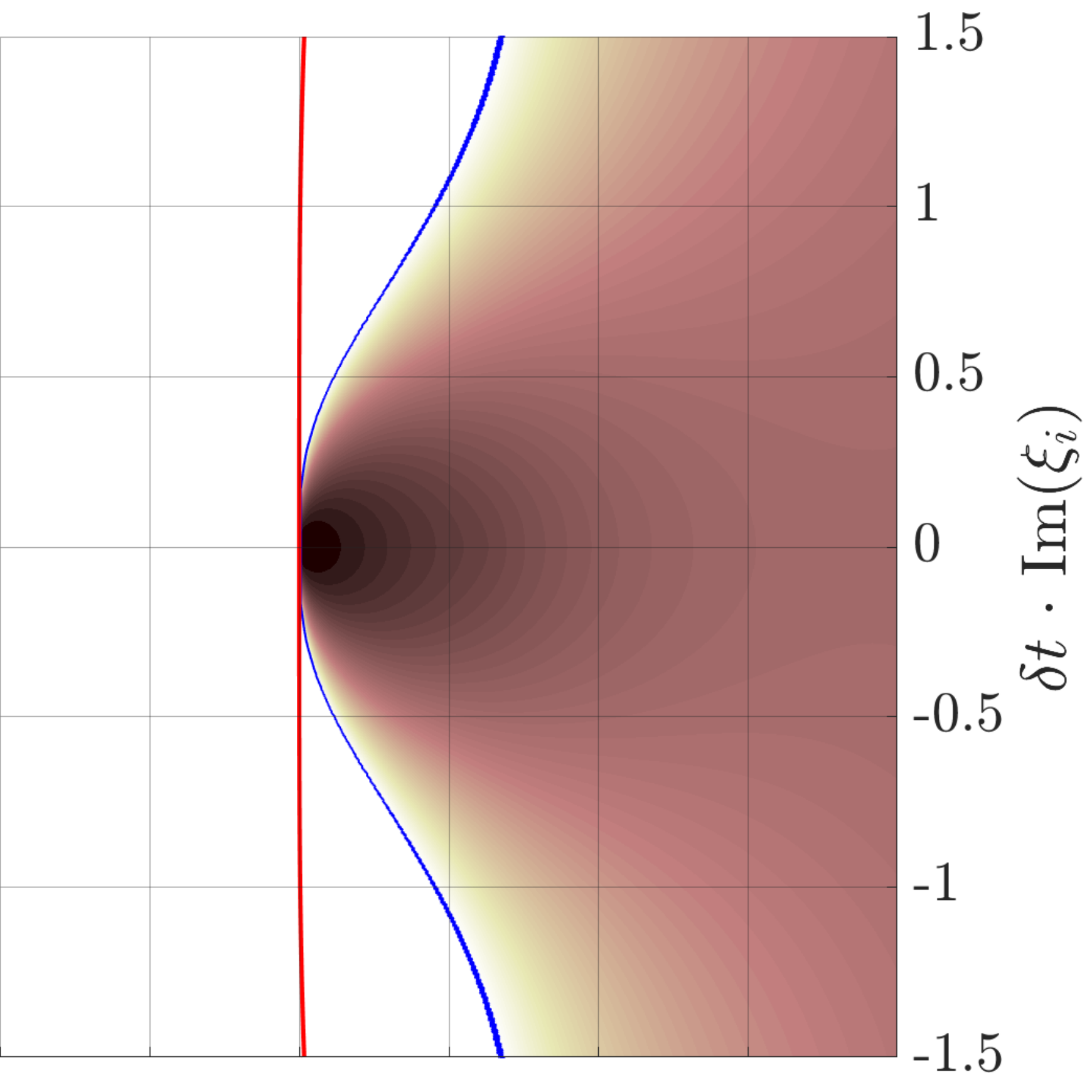}
        \end{subfigure}
    \end{center}
    \begin{center}
        \begin{subfigure}[b]{0.2575\textwidth}
            \includegraphics[width=\textwidth]{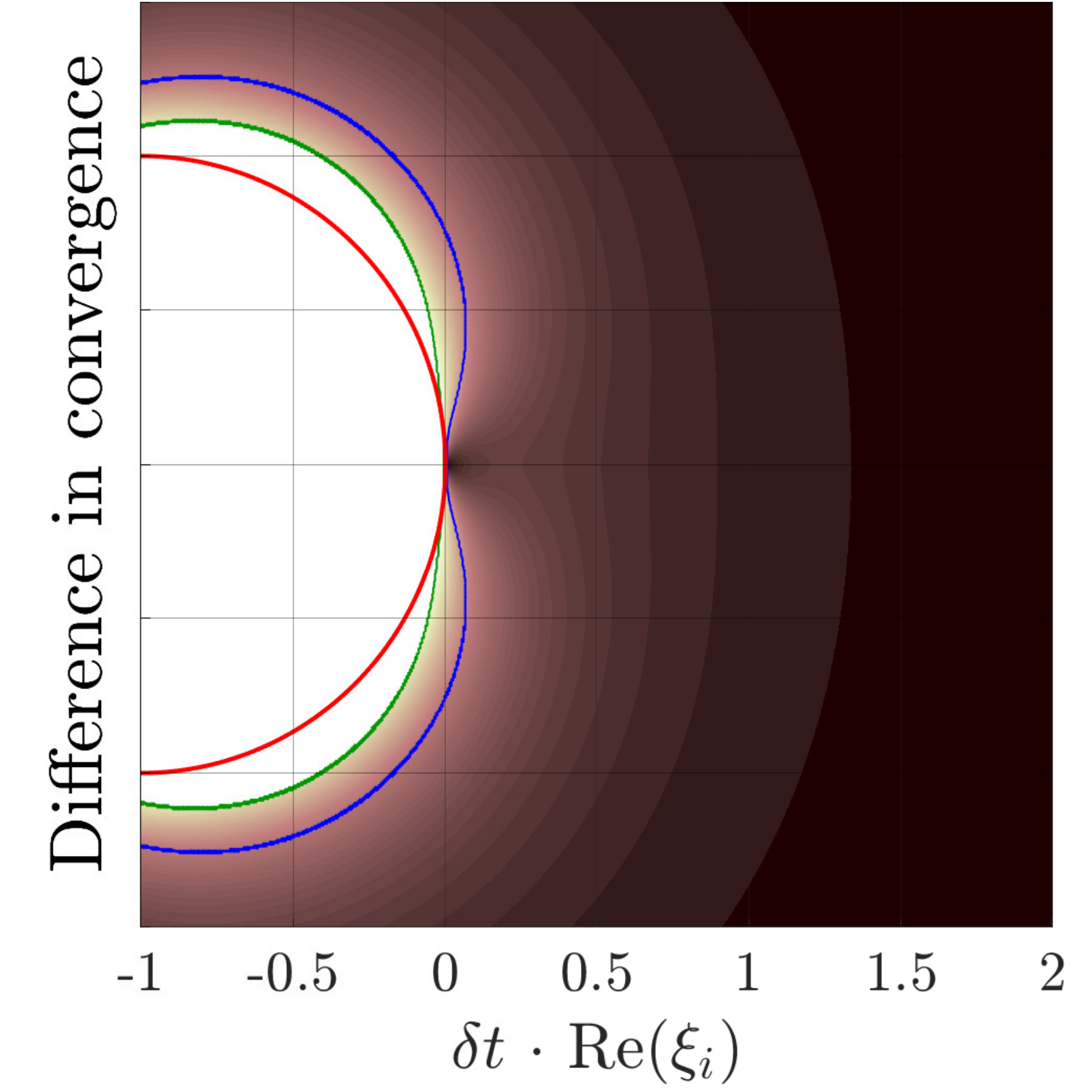}
        \end{subfigure}\hspace{1ex}
        \begin{subfigure}[b]{0.2575\textwidth}
            \includegraphics[width=\textwidth]{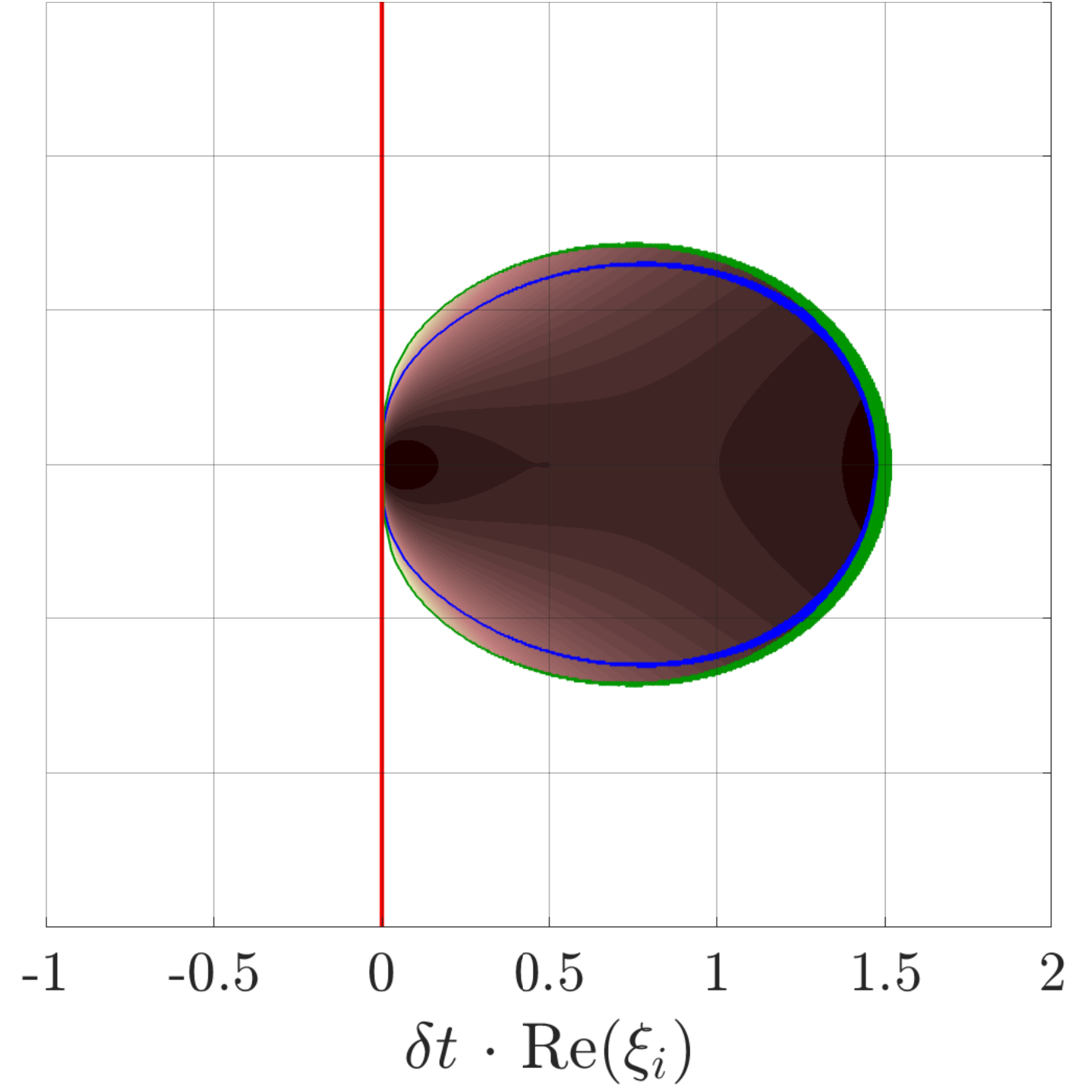}
        \end{subfigure}\hspace{1ex}
        \begin{subfigure}[b]{0.2575\textwidth}
            \includegraphics[width=\textwidth]{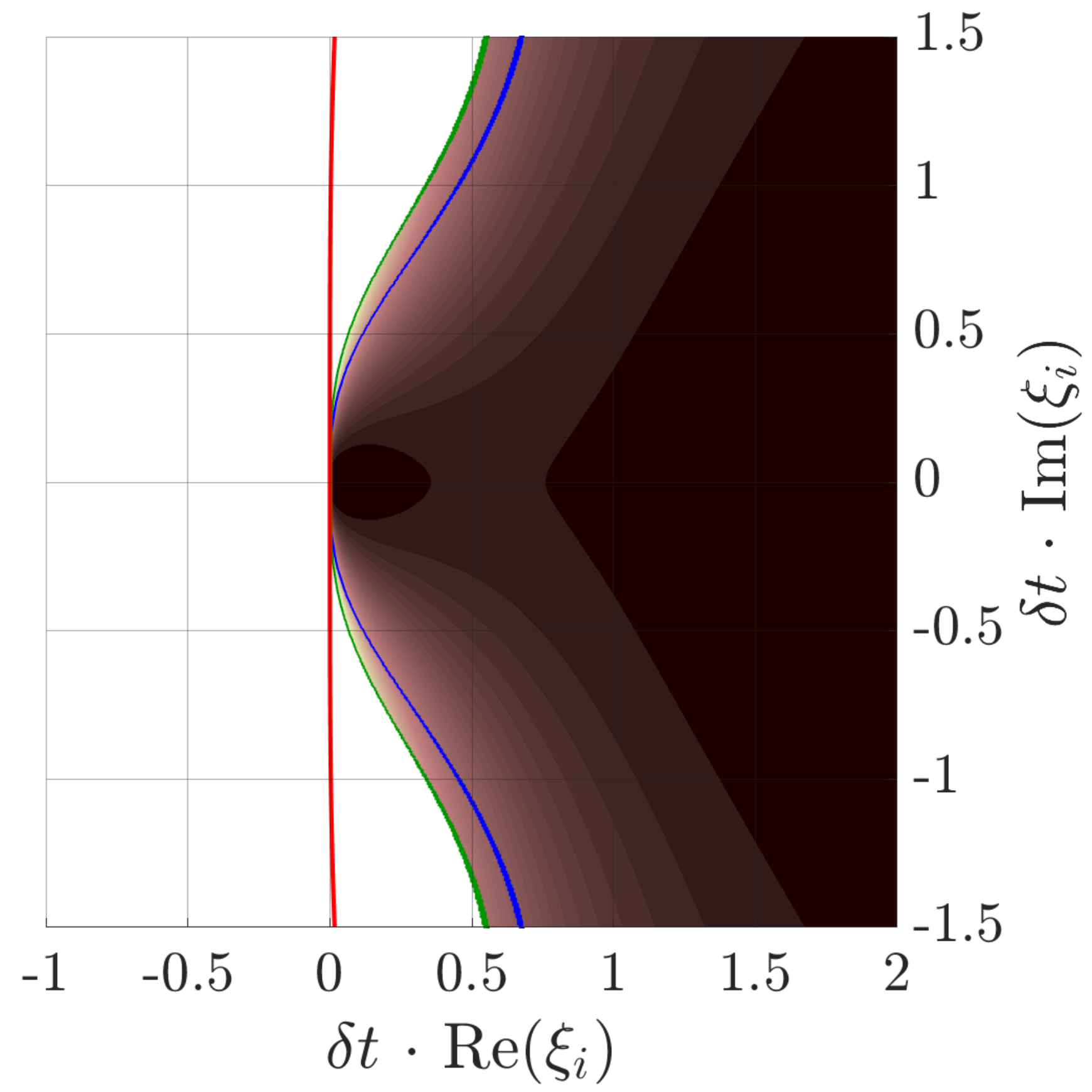}
        \end{subfigure}
        \end{center}
    \includegraphics[trim={0 0.5cm 0 3.5cm},clip,width=5in]{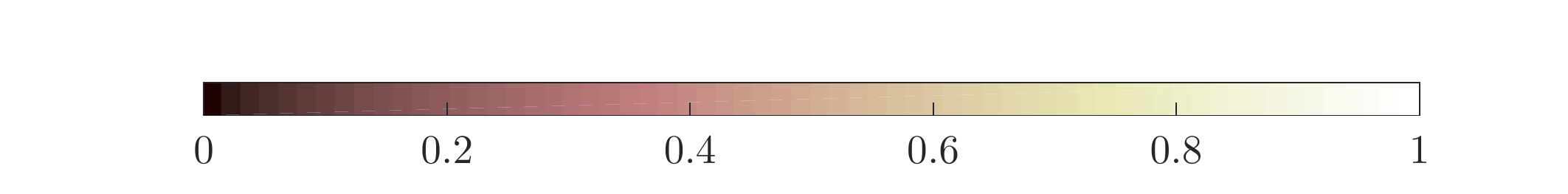}
    \caption{
        Convregence bounds for two-level MGRIT with F-relaxation and $k=4$, for A- and L-stable SDIRK schemes, of order $1$ and $2$, as a function
        of spatial eigenvalues $\{\xi_i\}$ in the complex plane. Red lines indicate the stability region of the integration scheme
        (stable left of the line). The top row shows two-grid convergence rates for all but one iteration \eqref{eq:tap_f}, with
        the green line marking the boundary of convergence. Similarly, the second row shows single-iteration two-grid
        convergence \eqref{eq:tight_2grid}, with blue line marking the boundary of convergence. The final row shows the
        difference in convergence between single-iteration and further two-grid iterations.
        }
    \label{fig:neumann}
\end{figure}

There are a few interesting points to note from \Cref{fig:neumann}. First, the 2nd-order L-stable
scheme yields good convergence over a far larger range of spatial eigenvalues and time steps than
the A-stable scheme. The better performance of L-stable schemes is discussed in detail in \cite{19c_mgrit},
primarily for the case of SPD spatial operators. However, some of the results extend to the complex
setting as well. In particular, if $\Psi$ is L-stable, then as $\delta_t|\xi_i| \to\infty$ two-level MGRIT
is convergent. This holds even on the imaginary axis, a spatial eigenvalue regime known to cause
difficulties for parallel-in-time algorithms. As a consequence, it is possible there are compact regions in the positive
half plane where two-level MGRIT will not converge, but convergence is guaranteed at the origin and
outside of these isolated regions. Convergence for large time steps is particularly relevant for multilevel schemes, because the coarsening procedure increases $\delta t$ exponentially. Such a result
does not hold for A-stable schemes, as seen in \Cref{fig:neumann}.

Second, \Cref{fig:neumann} indicates (at least one reason) why multilevel convergence is hard for hyperbolic
PDEs. It is typical for discretizations of hyperbolic PDEs to have spectra that push up against the imaginary
axis. From the limit of a skew symmetric matrix with purely imaginary
eigenvalues to more diffusive discretizations with nonzero real eigenvalue parts, it is usually
the case that there are small (in magnitude) eigenvalues with dominant imaginary components.
This results in eigenvalues pushing against the imaginary axis close to the origin. In the two-level setting,
backward Euler still guarantees convergence in the right half plane (see top left of \Cref{fig:neumann}),
regardless of imaginary eigenvalues. Note that to our knowledge, no other implicit scheme is convergent
for the entire right half plane. However, when considering single iteration convergence as a proxy for multilevel
convergence, we see in \Cref{fig:neumann} that even backward Euler is not convergent in a
small region along the imaginary axis. In particular, this region of non-convergence corresponds to
small eigenvalues with dominant imaginary parts, exactly the eigenvalues that typically arise in
discretizations of hyperbolic PDEs.

Other effects of imaginary components of eigenvalues are discussed in the following section, and
can be seen in results in \Cref{sec:adv}.

\section{Theoretical bounds in practice}\label{sec:imag}

\subsection{Convergence and imaginary eigenvalues}

One important point that follows from \cite{southworth19} is that convergence of MGRIT and Parareal
\textit{only depends on the discrete spatial and temporal problem}. Hyperbolic PDEs are known to
be difficult for PinT methods. However, for MGRIT/Parareal, it is not directly the (continuous)
hyperbolic PDE that causes difficulties, but rather its discrete representation. Spatial
discretizations of hyperbolic PDEs (when diagonalizable) often have eigenvalues with relatively large
imaginary components compared to the real component, particularly as magnitude $|\lambda| \to 0$.
In this section, we look at why eigenvalues with dominant imaginary part are difficult for MGRIT
and Parareal. The results are limited to diagonalizable operators, but give new insight on (the
known fact) that diffusivity of the backward Euler scheme makes it more amenable to PinT methods.
We also look at the relation of temporal problem size and coarsening factor to
convergence, which is particularly important for such discretizations, as well as the disappointing acceleration
of FCF-relaxation. Note, least-squares discretizations have been developed for hyperbolic PDEs that
result in an SPD spatial matrix (e.g., \cite{manteuffel1999boundary}),
which would immediately overcome the problems that arise with
complex/imaginary eigenvalues discussed in this section. Whether a given
discretization provides the desired
approximation to the continuous PDE is a different question.

\subsubsection{Problem size, and FCF-relaxation:}

Consider the exact solution to the linear time propagation problem, $\mathbf{u}_t = \mathcal{L}\mathbf{u}$, given by
\begin{align*}
\hat{\mathbf{u}} := e^{-\mathcal{L}t}\mathbf{u}.
\end{align*}
Then, an exact time step of size $\delta t$ is given by $\mathbf{u} \mapsfrom
e^{-\mathcal{L}\delta t}\mathbf{u}$. Runge-Kutta schemes are designed to
approximate this (matrix) exponential as a rational function, to order $\delta
t^p$ for some $p$.  Now suppose $\mathcal{L}$ is diagonalizable. Then,
propagating the $i$th eigenvector, say $\mathbf{v}_i$, forward in time by
$\delta t$, is achieved through the action $e^{-\delta t\xi_i}\mathbf{v}_i$,
where $\xi_i$ is the $i$th corresponding eigenvalue. Then the exact solution to
propagating $\mathbf{v}_i$ forward in time by $\delta t$ is given by
\begin{align}\label{eq:exp}
e^{-\delta t\xi_i} =e^{-\delta t \textnormal{Re}(\xi_i)} \Big( \cos(\delta t\textnormal{Im}(\xi_i)) - \mathrm{i}\sin(\delta t\textnormal{Im}(\xi_i))\Big).
\end{align}
If $\xi_i$ is purely imaginary, raising \eqref{eq:exp} to a power $k$,
corresponding to $k$ time steps, yields the function $e^{\pm\mathrm{i}k\delta
t|\xi_i|} = \cos(k\delta t|\xi_i|) \pm \mathrm{i}\sin(k\delta t|\xi_i|)$. This
function is (i) magnitude one for all $k,\delta t$, and $\xi$, and (ii) performs
exactly $k$ revolutions around the origin in a circle of radius one. Even though
we do not compute the exact exponential when integrating in time, we do
approximate it, and this perspective gives some insight on why operators with
imaginary eigenvalues tend to be particularly difficult for MGRiT and Parareal.

Recall the convergence bounds developed and stated in \Cref{sec:simple} as well
as \cite{southworth19} have a $\mathcal{O}(1/\sqrt{N_c})$ term in the
denominator. In many cases, the $\mathcal{O}(1/\sqrt{N_c})$ term is in some
sense arbitrary and the bounds in \Cref{th:conv} are relatively tight for
practical $N_c$. However, for some problems with spectra or field-of-values
aligned along the imaginary axis, convergence can differ significantly as a
function of $N_c$ (see \cite{19c_mgrit}). To that end, we restate Theorem 30
from \cite{southworth19}, which provides tight upper and lower bounds, including
the constants on $N_c$, for the case of diagonalizable operators.
\begin{theorem}[Tight bounds -- the diagonalizable case]\label{th:diag_tight}
Let $\Phi$ denote the fine-grid time-stepping operator and $\Psi$ denote the coarse-grid time-stepping operator,
with coarsening factor $k$, and $N_c$ coarse-grid time points. Assume that $\Phi$ and $\Psi$ commute and are
diagonalizable, with eigenvectors as columns of $U$, and spectra $\{\lambda_i\}$ and $\{\mu_i\}$, respectively.
Then, \textit{worst-case convergence} of error (and residual) in the $(UU^*)^{-1}$-norm is exactly bounded by
{\small
\begin{align*}
\sup_j \left( \frac{|\mu_j - \lambda_j^k|}{\sqrt{ (1 - |\mu_j|)^2 + \frac{\pi^2|\mu_j|}{N_c^2}}}\right)
  & \leq \frac{\|\mathbf{e}_{i+1}^{(F)}\|_{(UU^*)^{-1}}}{\|\mathbf{e}_i^{(F)}\|_{(UU^*)^{-1}}}
  \leq \left(\sup_j  \frac{|\mu_j - \lambda_j^k|}{\sqrt{(1 - |\mu_j|)^2 + \frac{\pi^2|\mu_j|}{6N_c^2}} } \right), \\
\sup_j \left(\frac{|\lambda_j^k||\mu_j - \lambda_j^k|}{\sqrt{ (1 - |\mu_j|)^2 + \frac{\pi^2|\mu_j|}{N_c^2}}}\right)
  & \leq \frac{\|\mathbf{e}_{i+1}^{(FCF)}\|_{(UU^*)^{-1}}}{\|\mathbf{e}_i^{(FCF)}\|_{(UU^*)^{-1}}}
  \leq \left( \sup_j  \frac{|\lambda_j^k||\mu_j - \lambda_j^k|}{\sqrt{(1 - |\mu_j|)^2 + \frac{\pi^2|\mu_j|}{6N_c^2}} } \right),
\end{align*}
}
for all but the last iteration (or all but the first iteration for residual).
\end{theorem}

The counter-intuitive nature of MGRiT and Parareal convergence is that
convergence is defined by how well $k$ steps on the fine-grid
\textit{approximate the coarse-grid operator}. That is, in the case of
eigenvalue bounds, we must have $|\mu_i - \lambda_i^k|^2 \lessapprox [(1 -
|\mu_i|)^2 + 10/N_c^2$, for each coarse-grid eigenvalue $\mu_i$. Clearly the
important cases for convergence are $|\mu_i|\approx 1$. Returning to
\eqref{eq:exp}, purely imaginary spatial eigenvalues typically lead to $|\mu_i|,
|\lambda_j| \approx 1$, particularly for small $|\delta t\xi_i|$ (because the RK
approximation to the exponential is most accurate near the origin). This has
several effects on convergence:
\begin{enumerate}
\item Convergence will deteriorate as the number of time steps increases,
\item $\lambda_i^k$ must approximate $\mu_i$ highly accurately for many eigenvalues,
\item FCF-relaxation will offer little to no improvement in convergence, and
\item Convergence will be increasingly difficult for larger coarsening factors.
\end{enumerate}

For the first and second points, notice from bounds in \Cref{th:diag_tight} that
the order of $N_c \gg 1$ is generally only significant when $|\mu_i| \approx 1$.
With imaginary eigenvalues, however, this leads to a moderate part of the
spectrum in which $\lambda_i^k$ must approximate $\mu_i$ with accuracy
$\approx1/N_c$, which is increasingly difficult as $N_c\to\infty$. Conversely,
introducing real parts to spatial eigenvalues introduces dissipation in
\eqref{eq:exp}, particularly when raising to powers. Typically the result of
this is that $|\mu_i|$ decreases, leading to fewer coarse-grid eigenvalues that
need to be approximated well, and a lack of dependence on $N_c$. 

In a similar vein, the third point above follows because in terms of
convergence, FCF-relaxation adds a power of $|\lambda_i|^k$ to convergence
bounds compared with F-relaxation. Improved convergence (hopefully due to
FCF-relaxation) is needed when $|\mu_i| \approx 1$. In an unfortunate cyclical
fashion, however, for such eigenvalues, it must be the case that $\lambda_i^k
\approx \mu_i$. But if $|\mu_i| \approx 1$, and $\lambda_i^k \approx \mu_i$, the
additional factor lent by FCF-relaxation, $|\lambda_i|\approx 1$, which tends
to offer at best marginal improvements in convergence. Finally, point four, which
specifies that it will be difficult to observe nice convergence with larger coarsening
factors, is a consequence of points two and three. As $k$ increases, $\Psi$ must
approximate a rational function, $\Phi^k$, of polynomials with progressively
higher degree. When $\Psi$ must approximate $\Phi$ well for many eigenmodes and,
in addition, FCF-relaxation offers minimal improvement in convergence, a
convergent method quickly becomes intractable. 

\subsubsection{Convergence in the complex plane:}

Although eigenvalues do not always provide a good measure of convergence (for
example, see \Cref{sec:adv}), they can provide invaluable information on 
choosing a good time-integration scheme for Parareal/MGRIT. Some of the
properties of a ``good'' time integration scheme transfer from the analysis
of time integration to parallel-in-time, however, some integration schemes
are far superior for parallel-in-time integration, without an obvious/intuitive
reason why. This section demonstrates how we can analyze time-stepping
schemes by considering convergence of two-level MGRIT/Parareal as a function
of eigenvalues in the complex plane. \Cref{fig:irk} and \Cref{fig:erk} plot
the real and imaginary parts of eigenvalues $\lambda\in\sigma(\Phi)$ and
$\mu\in\sigma(\Psi)$ as a function of $\delta t$ time the spatial eigenvalue,
as well as the corresponding two-level convergence for F- and FCF-relaxation,
for an A-stable ESDIRK-33 Runge-Kutta scheme and a 3rd-order explicit Runge-Kutta
scheme, respectively. 

\begin{figure}[!h]
    \centering
    \begin{center}
        \begin{subfigure}[c]{0.25\textwidth}
            \includegraphics[height=\textwidth]{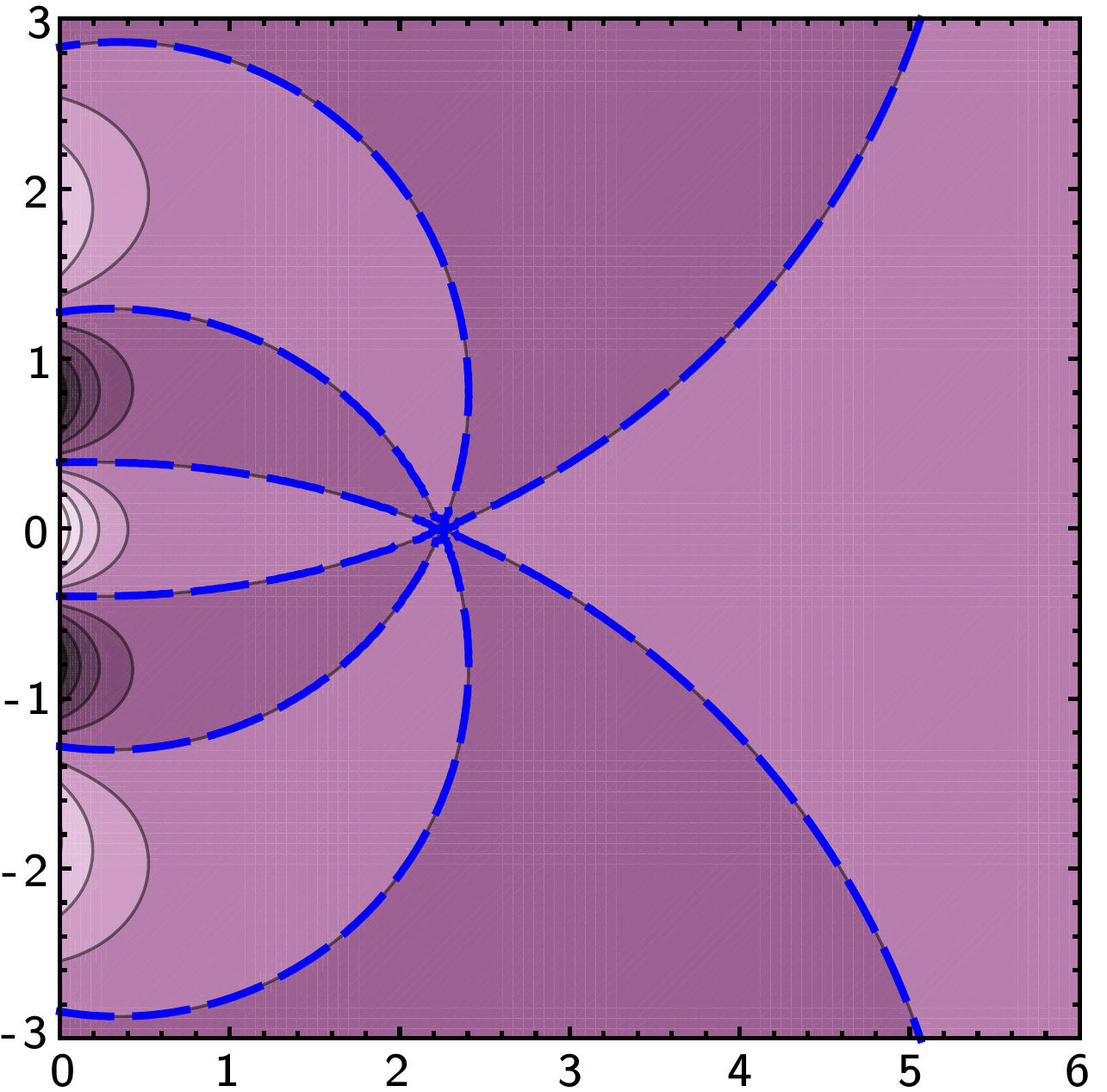}
            \caption{Re$(\lambda^4)$}
        \end{subfigure}
        \begin{subfigure}[c]{0.25\textwidth}
            \includegraphics[height=\textwidth]{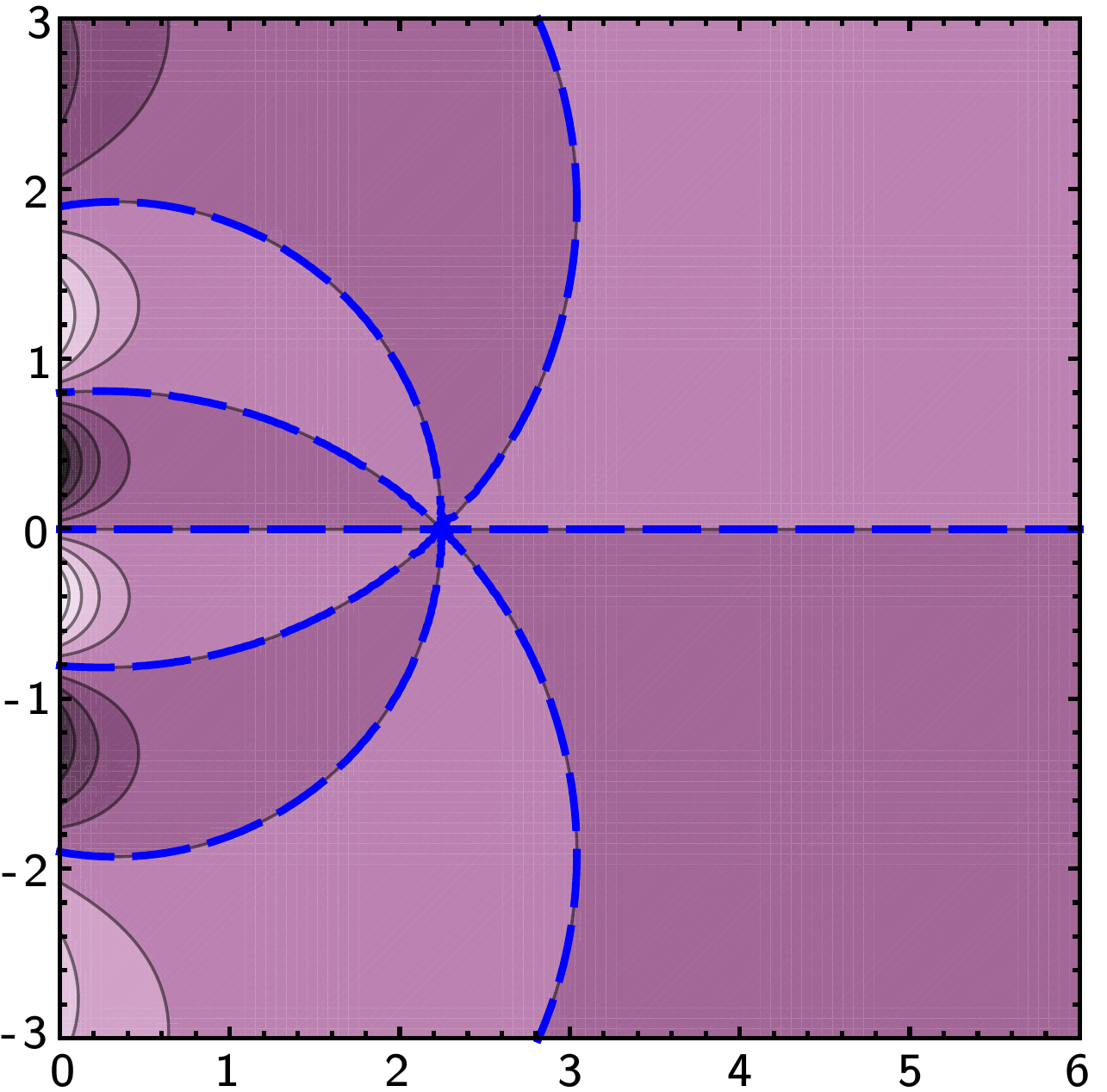}
            \caption{Im$(\lambda^4)$}
        \end{subfigure}
        \begin{subfigure}[b!]{0.055\textwidth}
            \includegraphics[width=\textwidth]{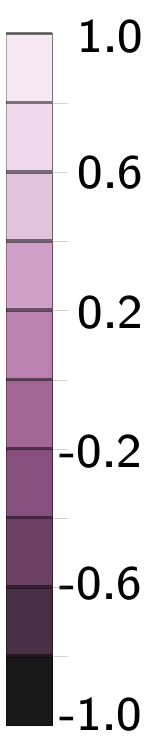}\vspace{5ex}
        \end{subfigure}
        \hspace{1ex}
        \begin{subfigure}[c]{0.25\textwidth}
            \includegraphics[height=\textwidth]{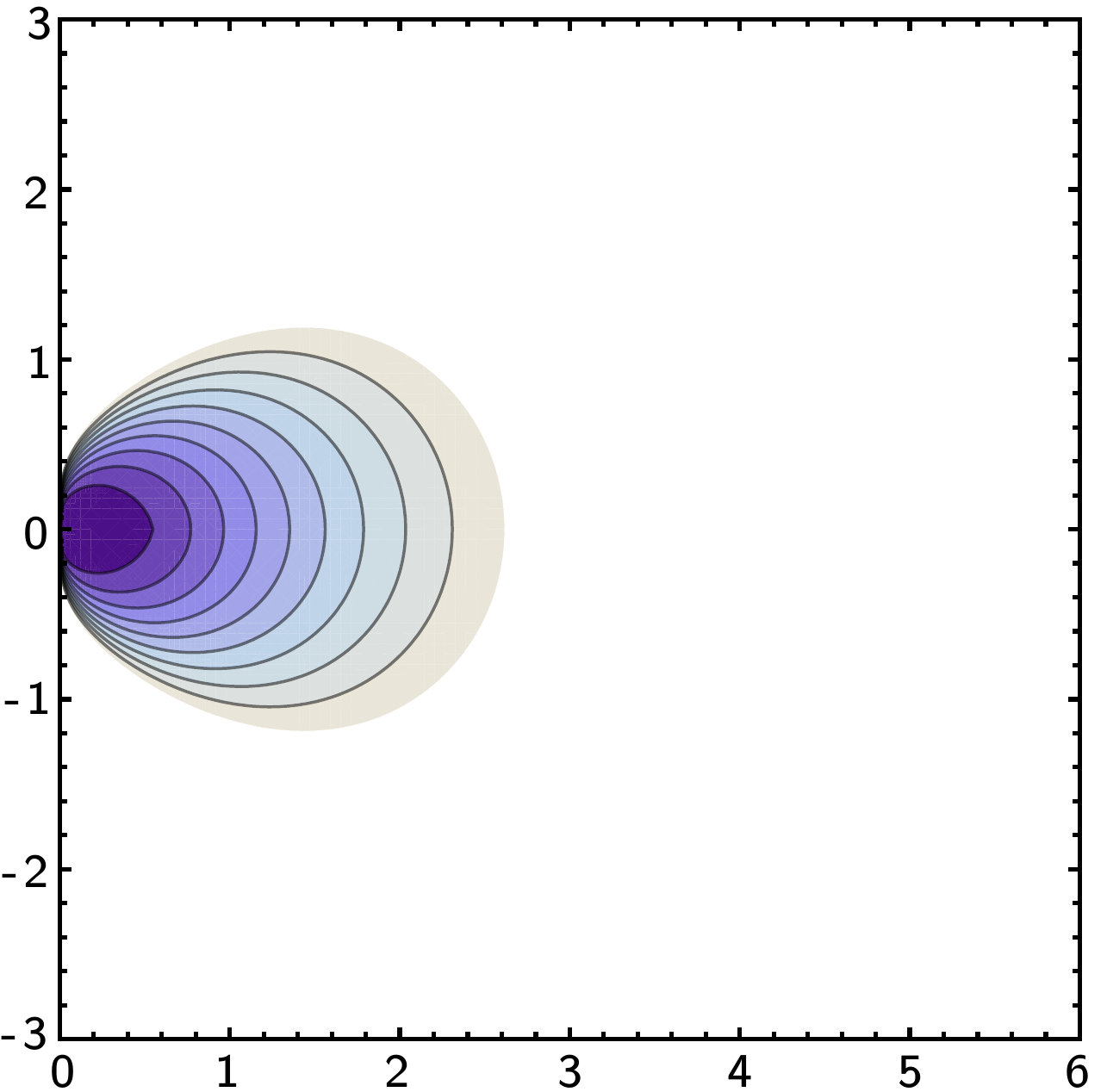}
            \caption{$\varphi_F$}
        \end{subfigure}
\\
        \begin{subfigure}[c]{0.25\textwidth}
            \includegraphics[height=\textwidth]{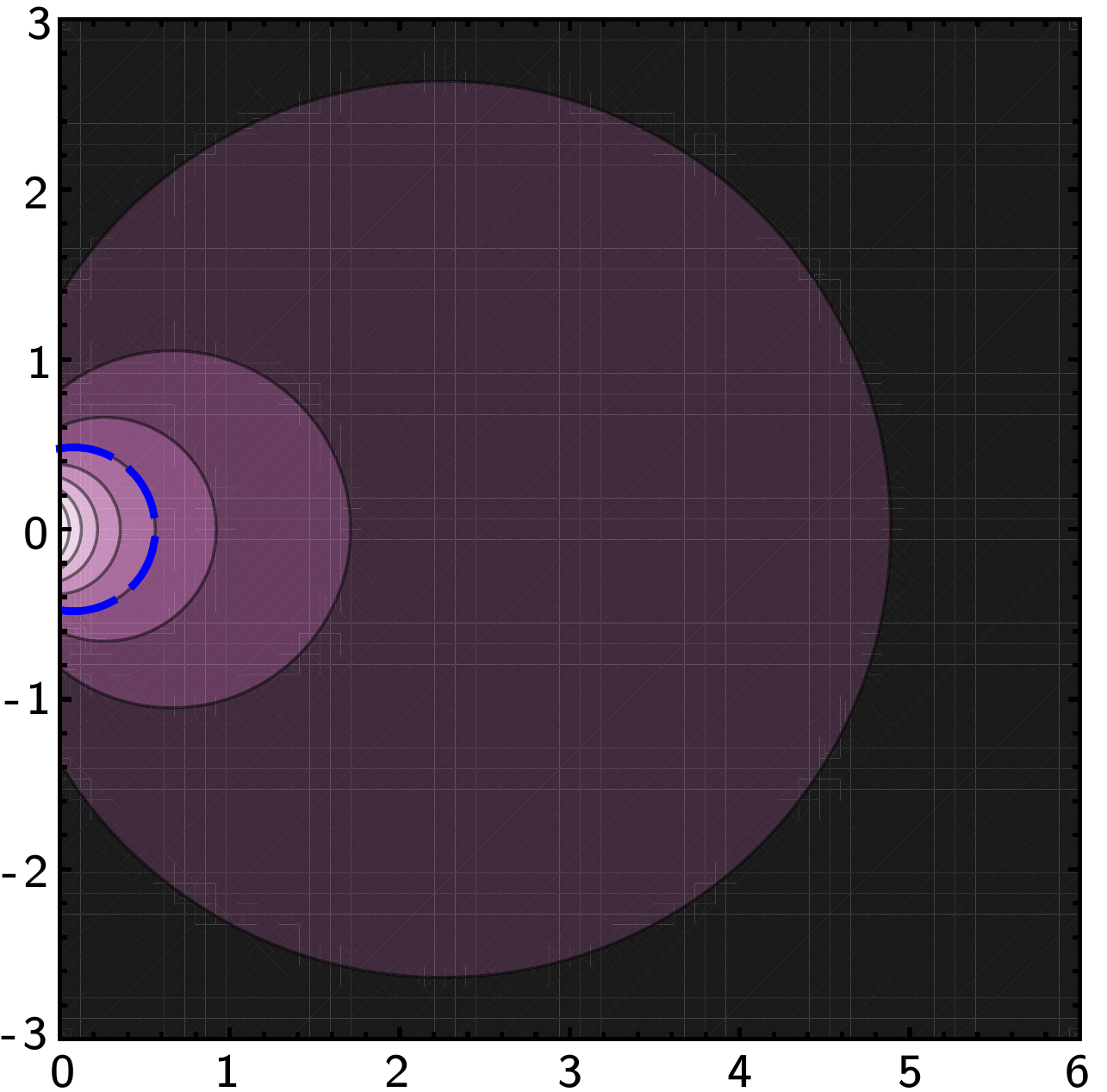}
            \caption{Re$(\mu_4)$}
        \end{subfigure}
        \begin{subfigure}[c]{0.25\textwidth}
            \includegraphics[height=\textwidth]{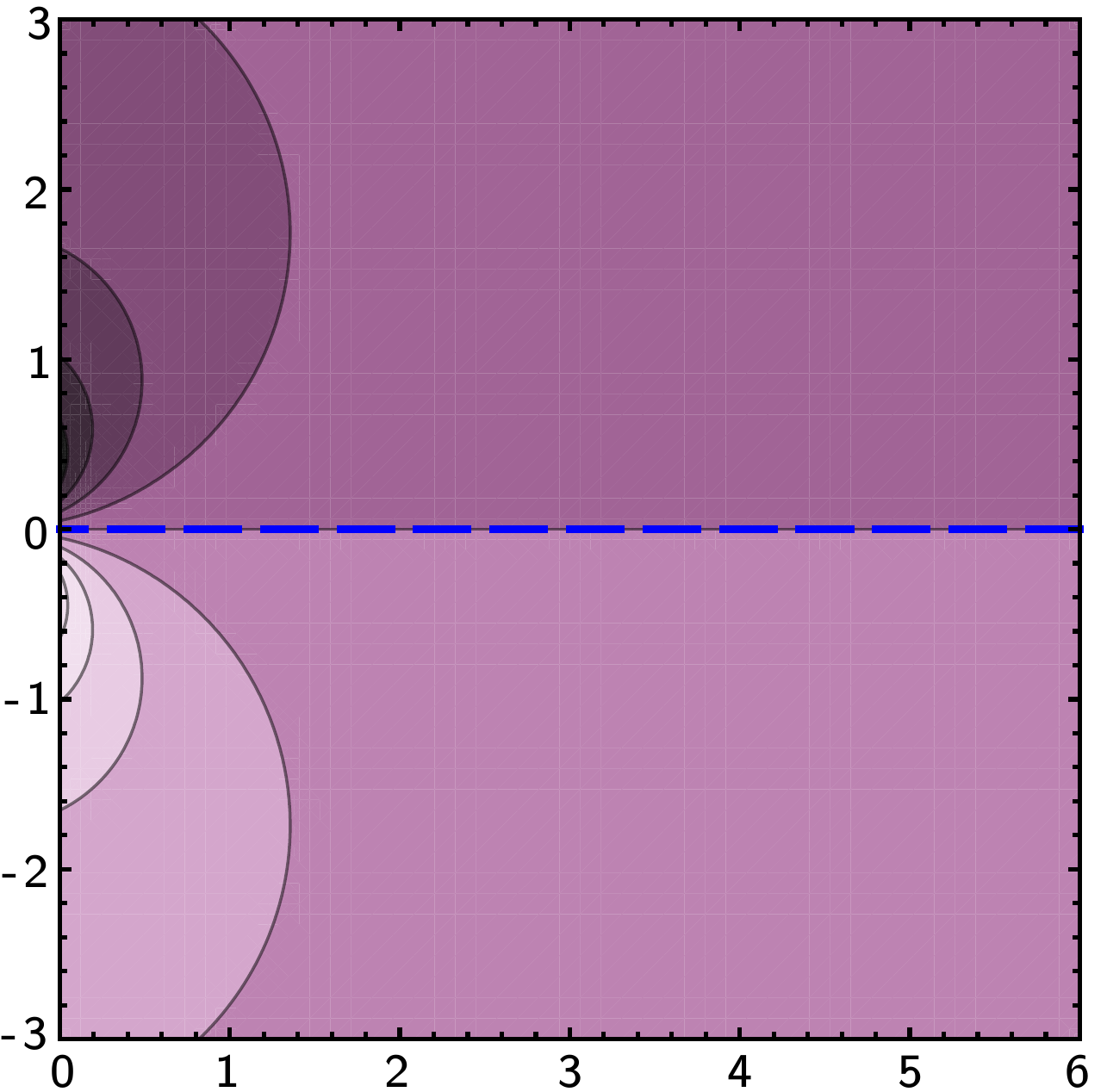}
            \caption{Im$(\mu_4)$}
        \end{subfigure}
        \hspace{1ex}
        \begin{subfigure}[b!]{0.0575\textwidth}
            \includegraphics[width=\textwidth]{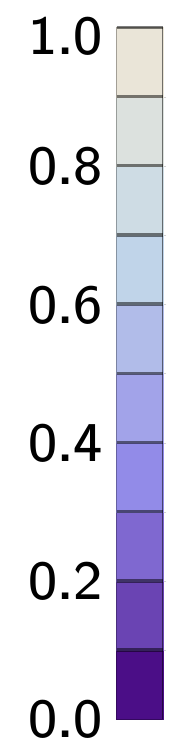}\vspace{5ex}
        \end{subfigure}
        \begin{subfigure}[c]{0.25\textwidth}
            \includegraphics[height=\textwidth]{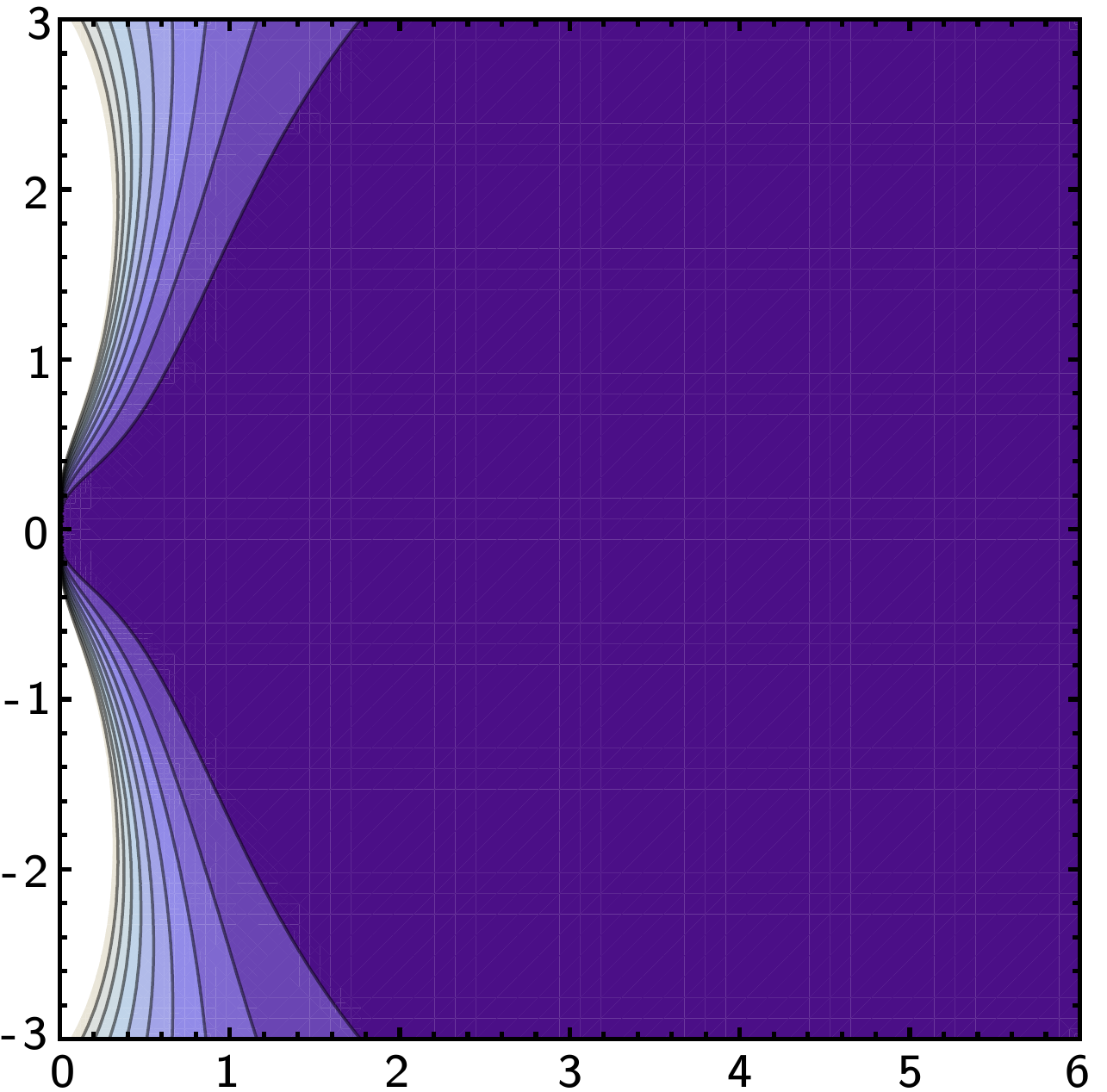}
            \caption{$\varphi_{FCF}$}
        \end{subfigure}
    \end{center}
    \caption{Eigenvalues and convergence bounds for ESDIRK-33, $p=3$ and $k=4$.
    Dashed blue lines indicate sign changes.
    {Note, if we zoom out on the FCF plot, it actually resembles that of F- with
    a diameter of about 100 rather than 2. Thus, for $\delta t \xi \gg 1$,
    even FCF-relaxation does not converge well.}}
    \label{fig:irk}
\end{figure}
\begin{figure}[!h]
    \centering
    \begin{center}
        \begin{subfigure}[c]{0.25\textwidth}
            \includegraphics[height=\textwidth]{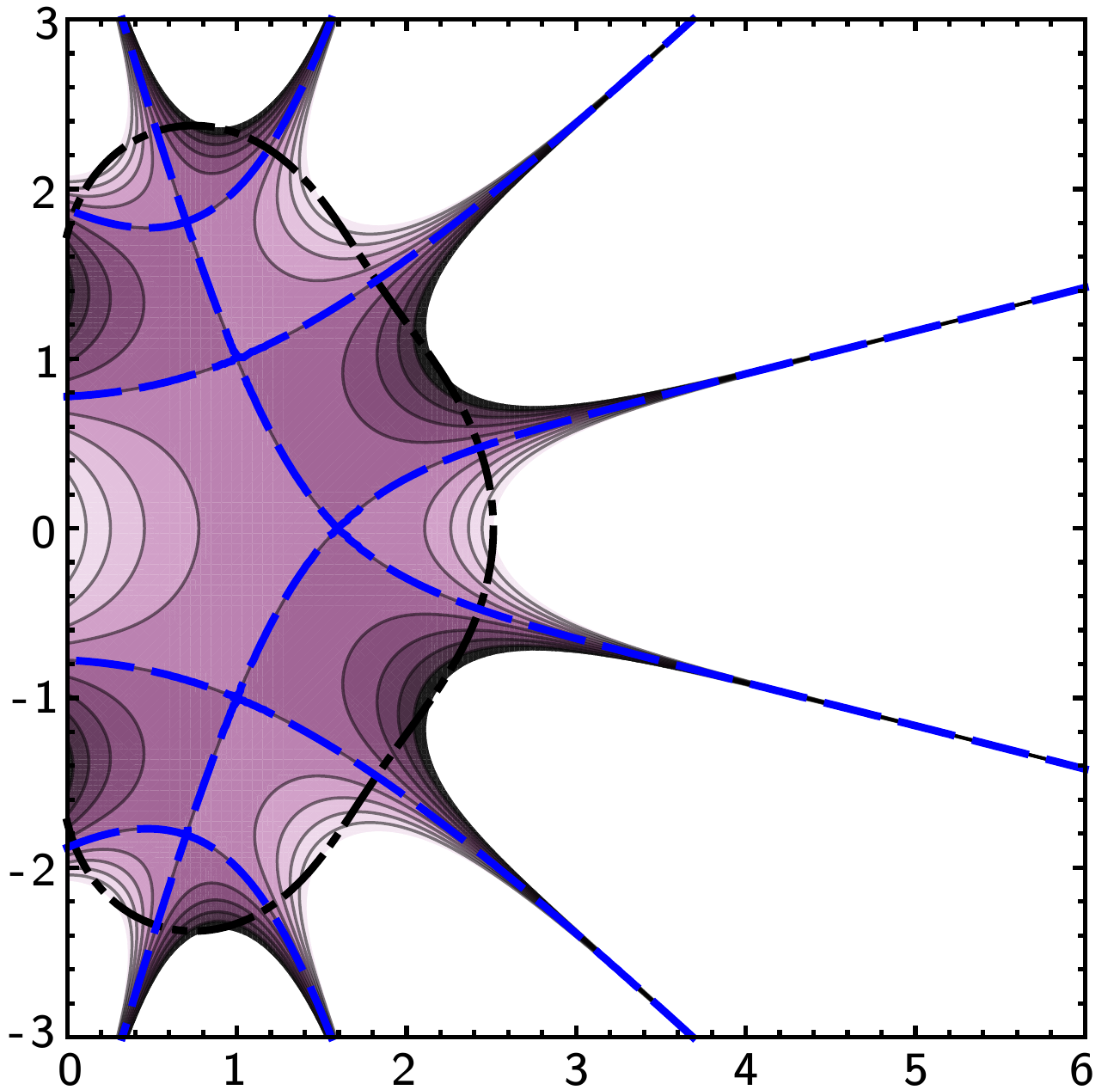}
            \caption{Re$(\lambda^2)$}
        \end{subfigure}
        \begin{subfigure}[c]{0.25\textwidth}
            \includegraphics[height=\textwidth]{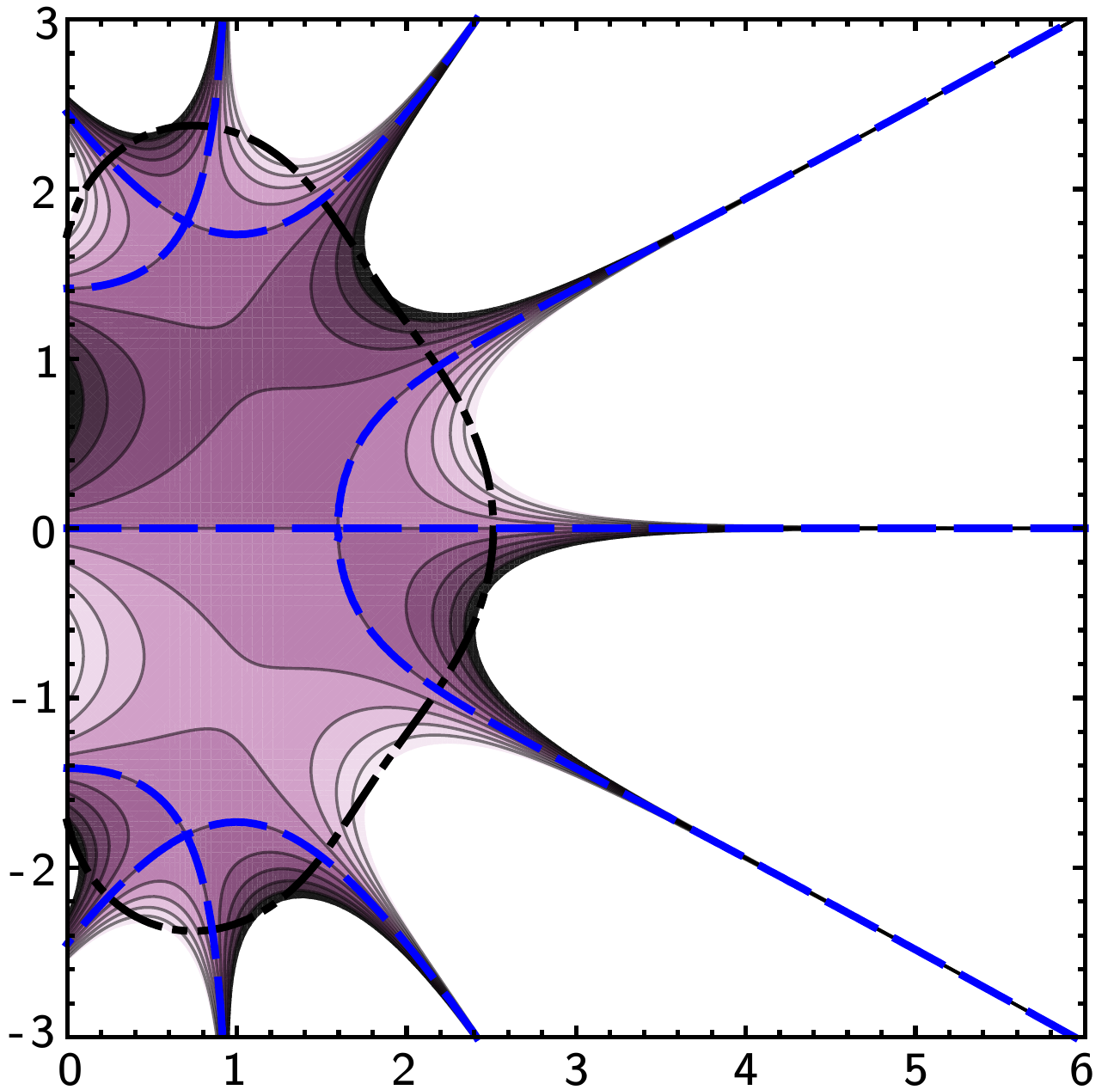}
            \caption{Im$(\lambda^2)$}
        \end{subfigure}
        \begin{subfigure}[b!]{0.055\textwidth}
            \includegraphics[width=\textwidth]{colorbar_-1-1.pdf}\vspace{6ex}
        \end{subfigure}
        \hspace{1ex}
        \begin{subfigure}[c]{0.25\textwidth}
            \includegraphics[height=\textwidth]{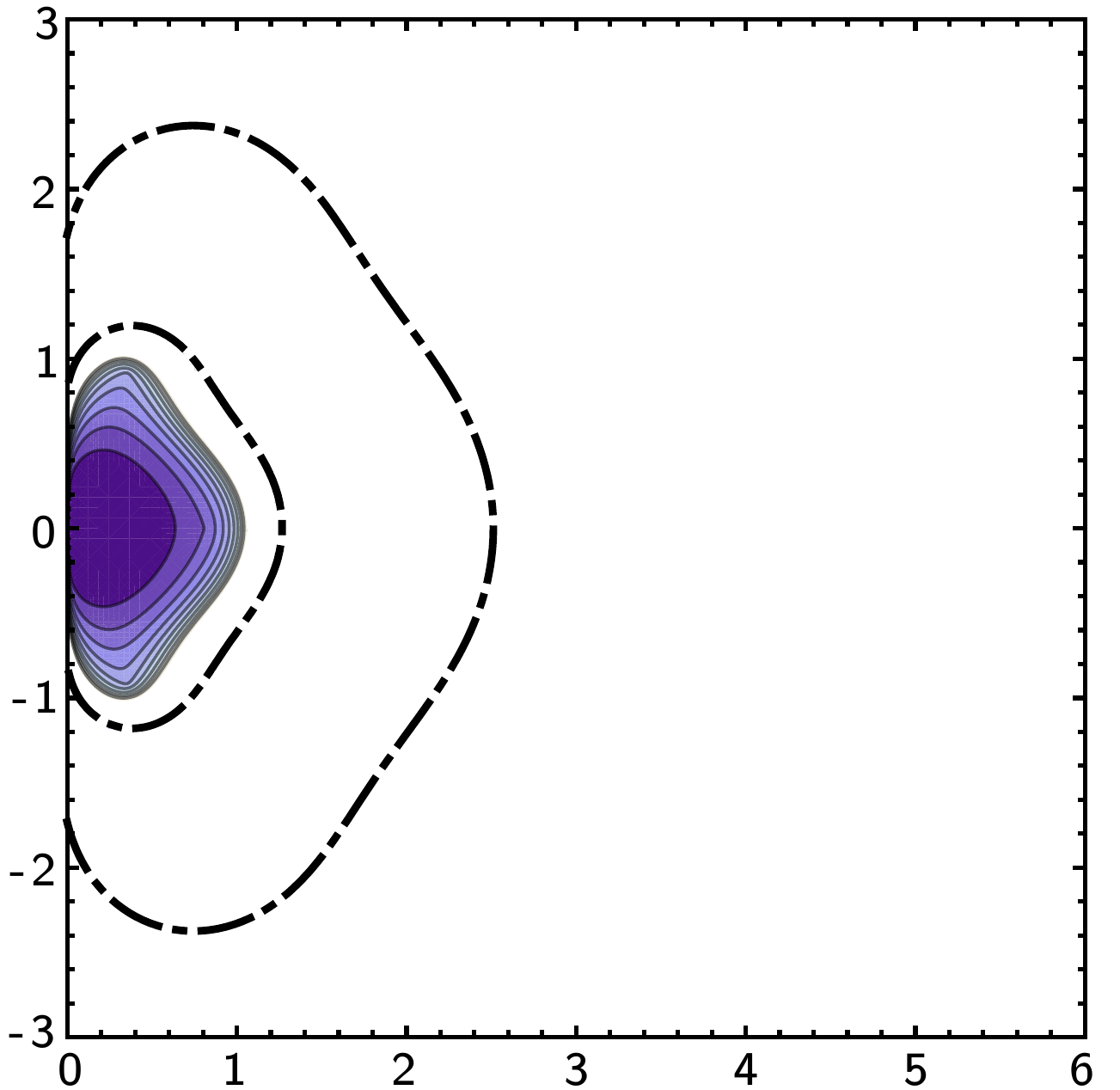}
            \caption{$\varphi_F$}
        \end{subfigure}
\\
        \begin{subfigure}[c]{0.25\textwidth}
            \includegraphics[height=\textwidth]{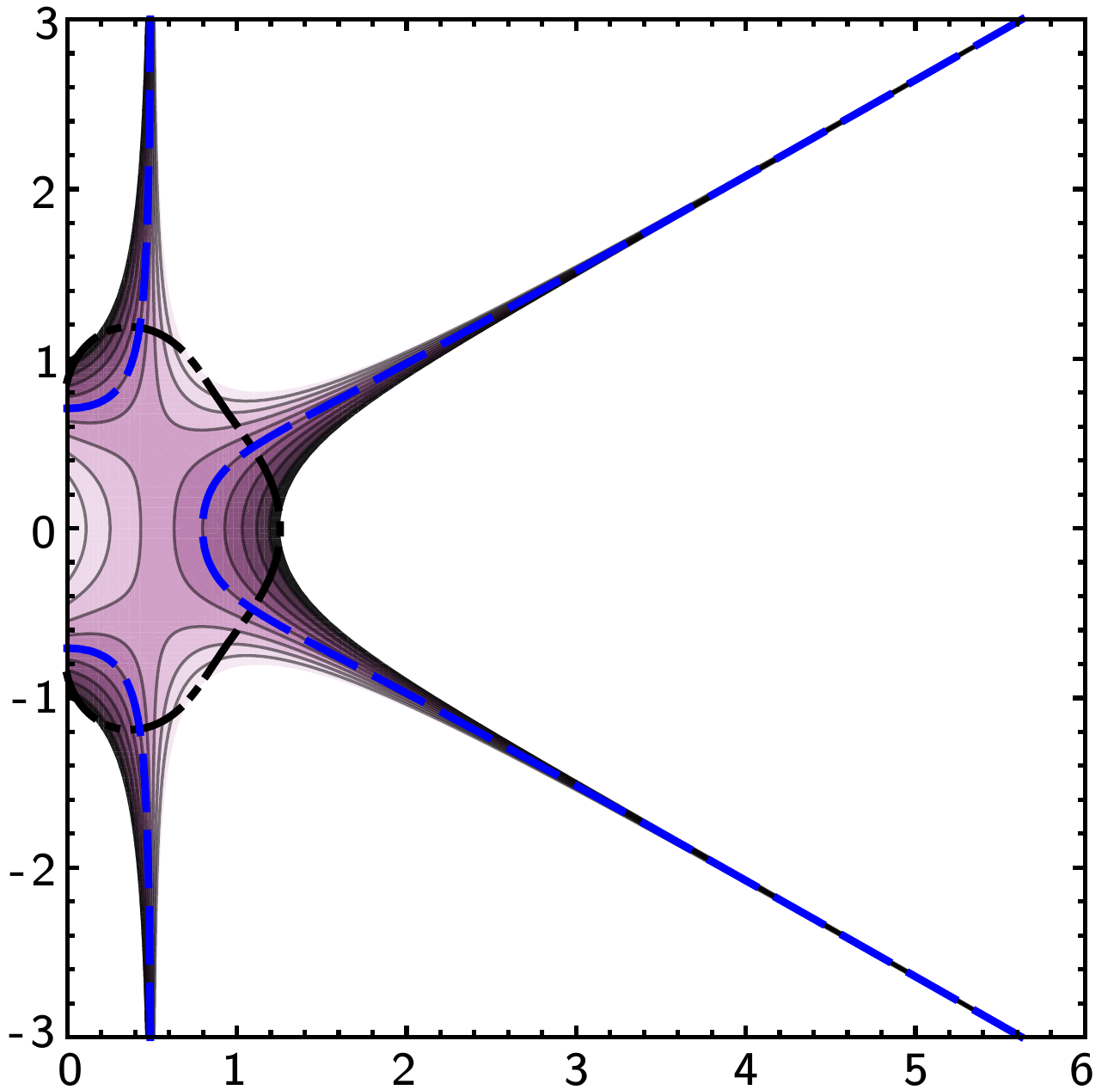}
            \caption{Re$(\mu_2)$}
        \end{subfigure}
        \begin{subfigure}[c]{0.25\textwidth}
            \includegraphics[height=\textwidth]{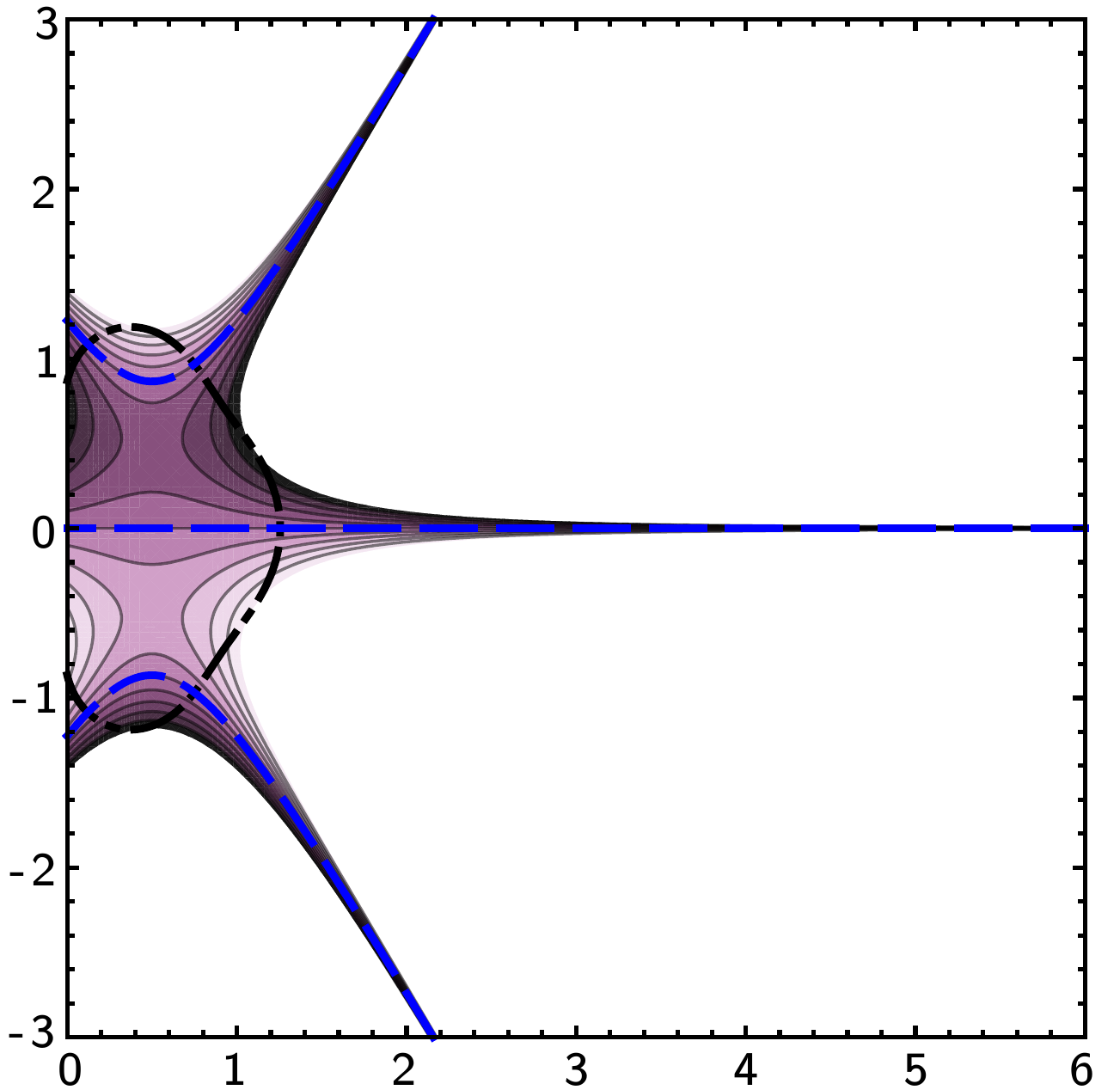}
            \caption{Im$(\mu_2)$}
        \end{subfigure}
        \hspace{1ex}
        \begin{subfigure}[b!]{0.0575\textwidth}
            \includegraphics[width=\textwidth]{colorbar_0-1.pdf}\vspace{6ex}
        \end{subfigure}
        \begin{subfigure}[c]{0.25\textwidth}
            \includegraphics[height=\textwidth]{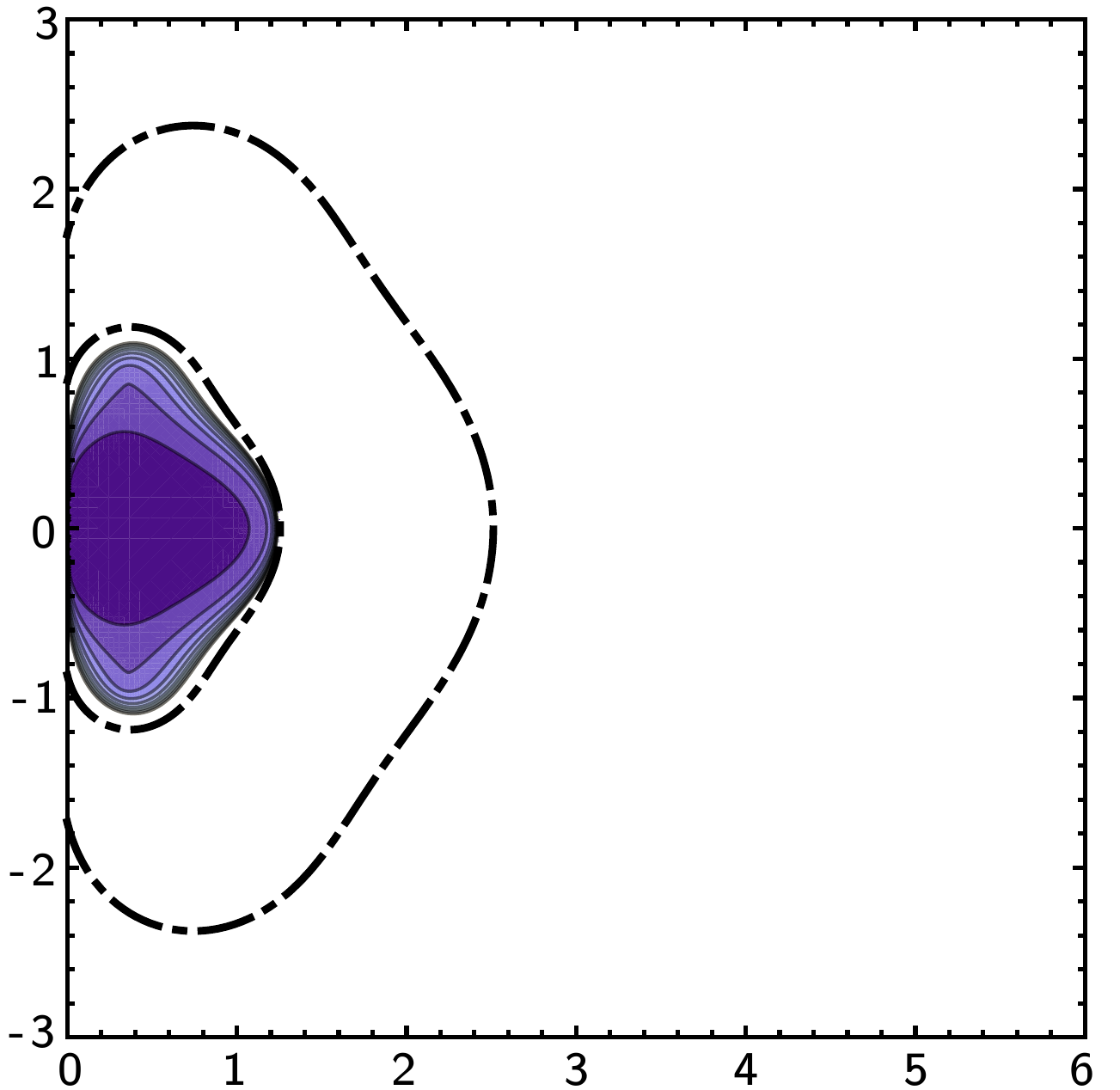}
            \caption{$\varphi_{FCF}$}
        \end{subfigure}
    \end{center}
    \caption{Eigenvalues and convergence bounds for ERK, $p=3$ and $k=2$.
    Dashed blue lines indicate sign changes.}
    \label{fig:erk}
\end{figure}

There are a few interesting things to note. First, FCF-relaxation expands the 
region of convergence in the complex plane dramatically for ESDIRK-33. However,
there are other schemes (not shown here, for brevity) where FCF-relaxation provides
little to no improvement. Also, note that the fine eigenvalue $\lambda^4$ changes
sign many times along the imaginary axis (in fact, the real part of $\lambda^k$
changes signs $2k$ times and the imaginary part $2k-1$). Such behavior is very
difficult to approximate with a coarse-grid time-stepping scheme, and provides
another way to think about why imaginary eigenvalues and hyperbolic PDEs can
be difficult for Parareal MGRIT. On a related note, using explicit time-stepping
schemes in Parareal/MGRIT is inherently limited by ensuring a stable time
step on the coarse grid, which makes naive application rare in numerical PDEs.
However, when stability is satisfied on the coarse grid, \Cref{fig:erk} (and similar
plots for other explicit schemes)
suggests that the domain of convergence pushes much closer against the imaginary
axis for explicit schemes than implicit. Such observations may be useful in
applying Parareal and MGRIT methods to systems of ODEs with
imaginary eigenvalues, but less stringent stability requirements,
wherein explicit integration may be a better choice than implicit.

\subsection{Test case: The wave equation} \label{sec:wave}

The previous section considered the effects of imaginary eigenvalues on convergence of MGRiT and
Parareal. Although true skew-symmetric operators are not particularly common, a similar character
can be observed in other discretizations. In particular, writing the 2nd-order wave equation in
first-order form and discretizing often leads to a spatial operator that is nearly skew-symmetric.
Two-level and multilevel convergence theory for MGRIT based on eigenvalues was demonstrated to provide
moderately accurate convergence estimates for small-scale discretizations of the second-order wave equation
in \cite{MGRIT19}. Here, we investigate the second-order wave equation further in the context of a finer
spatiotemporal discretization, examining why eigenvalues provide reliable information on convergence,
looking at the single-iteration bounds from \Cref{cor:new2grid}, and discussing the broader
implications.

The second-order wave equation in two spatial dimensions over domain $\Omega = (0, 2\pi) \times (0, 2\pi )$
is given by $\partial_{tt} u = c^2 \Delta u$ for $\mathbf{x} \in \Omega,  t \in (0, T]$
with scalar solution $u(\mathbf{x},t)$ and wave speed $c = \sqrt{10}$.
This can be written equivalently as a system of PDEs that are first-order in time,
\begin{align}
\begin{bmatrix} u \\ v\end{bmatrix}_t - \begin{bmatrix} {0} & I \\ c^2\Delta & 0 \end{bmatrix}\begin{bmatrix} u \\ v\end{bmatrix}
  = \begin{bmatrix} 0 \\ 0 \end{bmatrix}, \hspace{4ex}
    \text{for } \mathbf{x} \in \Omega,  t \in (0, T],
    \label{wave-system-eqn}
\end{align}
with initial and boundary conditions
\begin{alignat*}{4}
    u (\cdot, 0) &= \sin(x) \sin(y), \quad v (\cdot, 0) = 0,\qquad &&\text{for } \mathbf{x} \in \Omega \cup \partial \Omega, \\
    u (\mathbf{x}, \cdot) &= v (\mathbf{x}, \cdot) = 0, \qquad &&\text{for } \mathbf{x} \in \partial \Omega.
\end{alignat*}

\subsubsection{Why eigenvalues:}\label{wave:eig}

Although the operator that arises from writing the second-order wave equation as first-order
in time is not skew-adjoint, one can show that it
(usually) has purely imaginary eigenvalues. Moreover, although not unitarily diagonalizable, the
eigenvectors are only ill-conditioned in a specific, localized sense. As a result, the eigenvector
convergence bounds provide an accurate measure of convergence.

Let $\{\mathbf{w}_\ell, \zeta_\ell\}$ be an eigenpair of the discretization of $-c^2\Delta u = 0$
used in \eqref{wave-system-eqn}, for $\ell = 0,...,n-1$. For most standard
discretizations, we have $\zeta_\ell > 0$ $\forall \ell$ and the set $\{\mathbf{w}_\ell\}$ forms an orthonormal
basis of eigenvectors. Suppose this is the case. Expanding the block eigenvalue problem
$\mathbf{A}\mathbf{u}_j = \xi_j\mathbf{u}_j$ corresponding to \eqref{wave-system-eqn},
\begin{align*}
\begin{bmatrix} \mathbf{0} & I \\ c^2\Delta & \mathbf{0} \end{bmatrix} \begin{bmatrix}\mathbf{x}_j \\ \mathbf{v}_j \end{bmatrix}
  & = \xi_j \begin{bmatrix}\mathbf{x}_j \\ \mathbf{v}_j \end{bmatrix},
\end{align*}
yields a set of $2n$ eigenpairs, grouped in conjugate pairs of the corresponding purely imaginary
eigenvalues,
\begin{align*}
\{\mathbf{u}_{2\ell}, \xi_{2\ell}\} & := \left\{ \frac{1}{\sqrt{1+\zeta_\ell}}\begin{bmatrix} \mathbf{w}_{\ell} \\
  \mathrm{i}\sqrt{\zeta_\ell}\mathbf{w}_\ell\end{bmatrix}, \mathrm{i}\sqrt{\zeta_{\ell}}\right\}, \\
\{\mathbf{u}_{2\ell+1}, \xi_{2\ell+1}\} & := \left\{ \frac{1}{\sqrt{1+\zeta_\ell}}\begin{bmatrix} \mathbf{w}_{\ell} \\
  -\mathrm{i}\sqrt{\zeta_\ell}\mathbf{w}_\ell\end{bmatrix}, -\mathrm{i}\sqrt{\zeta_{\ell}}\right\},
\end{align*}
for $\ell=0,...,n-1$. Although the $(UU^*)^{-1}$-norm can be expressed in closed form, it is rather
complicated. Instead, we claim that eigenvalue bounds (theoretically tight in the $(UU^*)^{-1}$-norm)
provide a good estimate of $\ell^2$-convergence by considering the conditioning of eigenvectors.

Let $U$ denote a matrix with columns given by eigenvectors $\{\mathbf{u}_j\}$,
ordered as above for $\ell=0,...,n-1$. We can consider the conditioning of eigenvectors through the product
\begin{align}\label{eq:U*U}
U^*U = \begin{bmatrix} 1 & \frac{1 - \zeta_0}{1+\zeta_0} \\ \frac{1 - \zeta_0}{1+\zeta_0} & 1 \\ && \ddots \\
  &&& 1 & \frac{1 - \zeta_{n-1}}{1+\zeta_{n-1}} \\ &&& \frac{1 - \zeta_{n-1}}{1+\zeta_{n-1}} & 1 \end{bmatrix}.
\end{align}
Notice that $U^*U$ is a block-diagonal matrix with $2\times 2$ blocks corresponding to conjugate pairs
of eigenvalues. The eigenvalues of the $2\times 2$ block are given by $\{2\zeta_\ell/(1+\zeta_\ell), 2/(1+\zeta_\ell)\}$.
Although \eqref{eq:U*U} can be ill-conditioned for large $\zeta_{\ell} \sim 1/h^2$, for spatial mesh size $h$,
the ill-conditioning is only between conjugate pairs of eigenvalues, and eigenvectors are otherwise orthogonal.
Furthermore, the following Proposition proves that convergence bounds are symmetric across the real axis,
that is, the convergence bound for eigenvector with spatial eigenvalue $\xi$ is equivalent to that for its conjugate
$\bar{\xi}$. Together, these facts suggest the ill-conditioning between eigenvectors of conjugate pairs will
not significantly affect the accuracy of bounds, and that tight eigenvalue convergence bounds for MGRIT in the
$(UU^*)^{-1}$-norm provide accurate estimates of performance in practice.

\begin{proposition}\label{prop:symm}
Let $\Phi$ and $\Psi$ correspond to Runge-Kutta discretizations in time, as a function of a diagonalizable
spatial operator $\mathcal{L}$, with eigenvalues $\{\xi_i\}$. Then (tight) two-level convergence bounds of
MGRIT derived in \cite{southworth19} as a function $\delta t\xi$ are symmetric across the real axis.
\end{proposition}
\begin{proof}
Recall from \cite{southworth19,19c_mgrit}, eigenvalues of $\Phi$ and $\Psi$ correspond to the Runge-Kutta
stability functions applied to $\delta t\xi$ and $k\delta t\xi$, respectively, for coarsening factor $k$. Also note
that the stability function of a Runge-Kutta method can be written as a rational function of two polynomials
with real coefficients, $P(z)/Q(z)$ \cite{JCButcher_2016}. As a result of the fundamental theorem of linear
algebra $P(\bar{z})/Q(\bar{z}) = \overline{P(z)/Q(z)}$. Thus for spatial eigenvalue $\xi$, $|\lambda(\xi)| =
|\lambda(\bar{\xi})|$, $|\mu(\xi)| = |\mu(\bar{\xi})|$, and $|\mu(\bar{\xi}) - \lambda(\bar{\xi})^k|
= |\overline{\mu(\xi) - \lambda(\xi)^k}| = |\mu(\xi) - \lambda(\xi)^k|$, which implies that convergence
bounds in \Cref{th:diag_tight} are symmetric across the real axis.
\end{proof}

\subsubsection{Observed convergence vs. bounds:}\label{wave:mgrit}

The first-order form~\eqref{wave-system-eqn} is implemented in MPI/C++ using second-order
finite differences in space and various L-stable SDIRK time integration schemes
(see~\cite[Section SM3]{MGRIT19}).
We consider 4096 points in the temporal domain and $41$~points in the spatial domain, with
$4096\delta_t = 40\delta_x = T = 2\pi$ and $4096\delta_t = 10 \cdot 40\delta_x = T = 20\pi$,
such that $\delta_t \approx 0.1 \delta_x / c^2$ and $\delta_t \approx \delta_x / c^2$, respectively.
Two-level MGRIT with FCF-relaxation is tested for temporal coarsening factors
$k \in \{ 2, 4, 8, 16, 32 \}$~\cite{XBraid}, with a random initial space-time guess,
an absolute convergence tolerance of $10^{-10}$, and a maximum of $200$ and $1000$
iterations for SDIRK and ERK schemes, respectively. 
\Cref{fig:wave0} reports the geometric average (``Ave CF'') and worst-case
(``Worst CF'') convergence factors for XBraid runs,
along with estimates for the ``Eval single it'' bound from \Cref{cor:new2grid}
by letting $N_c\to\infty$, the ``Eval bound'' as the eigenvalue form of the GSVD
upper bound in \Cref{th:conv}, and the upper/lower bound from \Cref{th:diag_tight}.

\setlength{\figurewidth}{0.22\linewidth}
\setlength{\figureheight}{0.22\linewidth}
\colorlet{evalSingleIt}{ForestGreen}
\colorlet{evalBound}{Orange}
\colorlet{worstCF}{Maroon}
\colorlet{aveCF}{Purple}
\colorlet{upperBound}{NavyBlue}
\colorlet{lowerBound}{Gray}
\begin{figure}[!ht]
    \centering
    \begin{center}
                                                                                                                                                                {
        \pgfplotsset{every tick label/.append style={font=\footnotesize}}
        \begin{subfigure}[c]{0.3\textwidth}
            \begin{tikzpicture}

\begin{axis}[width=\figurewidth,
height=0.849\figureheight,
at={(0\figurewidth,0\figureheight)},
scale only axis,
unbounded coords=jump,
xmin=2,
xmax=32,
xtick={2, 4, 8, 16, 32},
xlabel style={font=\color{white!15!black}},
xlabel={\footnotesize Coarsening factor $k$},
ymode=log,
ymin=1e-1,
ymax=1e1,
yminorticks=true,
ylabel style={font=\color{white!15!black}},
ylabel={\footnotesize Residual CF},
axis background/.style={fill=white}
]

\addplot [color=evalSingleIt, dashed, line width=1.5pt, mark size=2pt, mark=triangle, mark options={fill=none, solid, evalSingleIt}, forget plot]
table[]{2.           0.70709639
4.           1.49998775
8.           2.54125365
16.           4.24531537
32.           6.96423534
};

\addplot [color=evalBound, dotted, line width=1.5pt, mark size=3pt, mark=star, mark options={fill=none, solid, evalBound}, forget plot]
table[]{2.           0.49999412
4.           0.75002664
8.           0.91041159
16.           1.09163188
32.           1.27938329
};

\addplot [color=upperBound, dashed, line width=1.5pt, mark size=2pt, mark=square, mark options={fill=none, solid, upperBound}, forget plot]
table[]{2.           0.49653114
4.           0.74942625
8.           0.91033966
16.           1.09158655
32.           1.27931766
};

\addplot [color=worstCF, line width=1.5pt, mark size=2pt, mark=*, mark options={fill=none, solid, worstCF}, forget plot]
table[]{2.           0.48139481
4.           0.73684041
8.           0.89937646
16.           1.07953343
32.           1.26444759
};

\addplot [color=aveCF, line width=1.5pt, mark size=3pt, mark=+, mark options={fill=none, solid, aveCF}, forget plot]
table[]{2.           0.359026440536954816
4.           0.627054457499721996
8.           0.753942081982384882
16.           0.753563115637586378
32.           0.574260291336758710
};

\addplot [color=lowerBound, dotted, line width=1.5pt, mark size=2.5pt, mark=diamond, mark options={fill=none, solid, lowerBound}, forget plot]
table[]{2.           0.49362255
4.           0.74836374
8.           0.9099803
16.           1.09135994
32.           1.27898983
};

\addplot [color=black, line width=1pt, dotted, forget plot]
table[]{
2.0 1.0
32.0 1.0
};

\end{axis}
\end{tikzpicture}
            \caption{SDIRK1, $\delta t \approx 0.1 \tfrac{dx}{c^2}$}
            \label{fig:wave0:ratio0:sdirk1}
        \end{subfigure}\qquad
        \begin{subfigure}[c]{0.3\textwidth}
            \begin{tikzpicture}

\begin{axis}[width=\figurewidth,
height=0.849\figureheight,
at={(0\figurewidth,0\figureheight)},
scale only axis,
unbounded coords=jump,
xmin=2,
xmax=32,
xtick={2, 4, 8, 16, 32},
xlabel style={font=\color{white!15!black}},
xlabel={\footnotesize Coarsening factor $k$},
ymode=log,
ymin=1e-1,
ymax=1e3,
ytick={1e-1,1e0,1e1,1e2,1e3},
yminorticks=true,
ylabel style={font=\color{white!15!black}},
axis background/.style={fill=white}
]

\addplot [color=evalSingleIt, dashed, line width=1.5pt, mark size=2pt, mark=triangle, mark options={fill=none, solid, evalSingleIt}, forget plot]
table[]{2.          1699.6696797
4.          1502.31560855
8.          1115.46932368
16.           798.31563756
32.           566.63548537
};

\addplot [color=evalBound, dotted, line width=1.5pt, mark size=3pt, mark=star, mark options={fill=none, solid, evalBound}, forget plot]
table[]{2.          1201.84795629
4.           751.15780427
8.           394.37796149
16.           199.57890939
32.           100.16794854
};

\addplot [color=upperBound, dashed, line width=1.5pt, mark size=2pt, mark=square, mark options={fill=none, solid, upperBound}, forget plot]
table[]{2.           0.25696938
 4.           1.27534195
 8.           4.96963124
16.           8.31090726
32.           7.27724094
};

\addplot [color=worstCF, line width=1.5pt, mark size=2pt, mark=*, mark options={fill=none, solid, worstCF}, forget plot]
table[]{2.           0.10891253
 4.           0.66000251
 8.           1.93081128
16.           5.29397611
32.           5.97815656
};

\addplot [color=aveCF, line width=1.5pt, mark size=3pt, mark=+, mark options={fill=none, solid, aveCF}, forget plot]
table[]{2.           0.024696404490560171
 4.           0.127148066790694597
 8.           0.357469936097139507
16.           0.654723739130270710
32.           0.554909363617867357
};

\addplot [color=lowerBound, dotted, line width=1.5pt, mark size=2.5pt, mark=diamond, mark options={fill=none, solid, lowerBound}, forget plot]
table[]{2.           0.1049086
 4.           0.52105308
 8.           2.11353211
16.           5.96482091
32.           6.2611261
};

\addplot [color=black, line width=1pt, dotted, forget plot]
table[]{
2.0 1.0
32.0 1.0
};

\end{axis}
\end{tikzpicture}
            \caption{SDIRK2, $\delta t \approx 0.1 \tfrac{dx}{c^2}$}
            \label{fig:wave0:ratio0:sdirk2}
        \end{subfigure}\quad
        \begin{subfigure}[c]{0.3\textwidth}
            \begin{tikzpicture}

\begin{axis}[width=\figurewidth,
height=0.849\figureheight,
at={(0\figurewidth,0\figureheight)},
scale only axis,
unbounded coords=jump,
xmin=2,
xmax=32,
xtick={2, 4, 8, 16, 32},
xlabel style={font=\color{white!15!black}},
xlabel={\footnotesize Coarsening factor $k$},
ymode=log,
ymin=1e-2,
ymax=1e1,
yminorticks=true,
ylabel style={font=\color{white!15!black}},
axis background/.style={fill=white}
]

\addplot [color=evalSingleIt, dashed, line width=1.5pt, mark size=2pt, mark=triangle, mark options={fill=none, solid, evalSingleIt}, forget plot]
table[]{2.           1.2432952
4.           2.0101072
8.           3.0550558
16.           5.17788586
32.          10.61582828
};

\addplot [color=evalBound, dotted, line width=1.5pt, mark size=3pt, mark=star, mark options={fill=none, solid, evalBound}, forget plot]
table[]{2.           0.87914313
4.           1.00505586
8.           1.080131
16.           1.29448601
32.           1.87667461
};

\addplot [color=upperBound, dashed, line width=1.5pt, mark size=2pt, mark=square, mark options={fill=none, solid, upperBound}, forget plot]
table[]{2.           0.03330159
4.           0.28152296
8.           0.967959
16.           1.288241
32.           1.87542531
};

\addplot [color=worstCF, line width=1.5pt, mark size=2pt, mark=*, mark options={fill=none, solid, worstCF}, forget plot]
table[]{2.00000000e+00   1.49095894e-02
4.00000000e+00   1.40531098e-01
8.00000000e+00   6.46508585e-01
1.60000000e+01   1.16805832e+00
3.20000000e+01   1.80108748e+00
};

\addplot [color=aveCF, line width=1.5pt, mark size=3pt, mark=+, mark options={fill=none, solid, aveCF}, forget plot]
table[]{2.00000000e+00   0.001722260958591609
4.00000000e+00   0.026716755030029900
8.00000000e+00   0.178984336611682171
1.60000000e+01   0.488732433786846276
3.20000000e+01   0.571719786990243728
};

\addplot [color=lowerBound, dotted, line width=1.5pt, mark size=2.5pt, mark=diamond, mark options={fill=none, solid, lowerBound}, forget plot]
table[]{2.00000000e+00   1.36034530e-02
4.00000000e+00   1.18883450e-01
8.00000000e+00   6.87104441e-01
1.60000000e+01   1.25831740e+00
3.20000000e+01   1.86921603e+00
};

\addplot [color=black, line width=1pt, dotted, forget plot]
table[]{
2.0 1.0
32.0 1.0
};

\end{axis}
\end{tikzpicture}
            \caption{SDIRK3, $\delta t \approx 0.1 \tfrac{dx}{c^2}$}
            \label{fig:wave0:ratio0:sdirk3}
        \end{subfigure}
        }
        \\
        {
        \pgfplotsset{every tick label/.append style={font=\footnotesize}}
        \begin{subfigure}[c]{0.3\textwidth}
            \begin{tikzpicture}

\begin{axis}[width=\figurewidth,
height=0.849\figureheight,
at={(0\figurewidth,0\figureheight)},
scale only axis,
unbounded coords=jump,
xmin=2,
xmax=32,
xtick={2, 4, 8, 16, 32},
xlabel style={font=\color{white!15!black}},
xlabel={\footnotesize Coarsening factor $k$},
ymode=log,
ymin=1e-1,
ymax=1e1,
yminorticks=true,
ylabel style={font=\color{white!15!black}},
ylabel={\footnotesize Residual CF},
axis background/.style={fill=white}
]

\addplot [color=evalSingleIt, dashed, line width=1.5pt, mark size=2pt, mark=triangle, mark options={fill=none, solid, evalSingleIt}, forget plot]
table[]{2.           0.70606785
4.           1.49873192
8.           2.55115271
16.           4.26035129
32.           6.96407128
};

\addplot [color=evalBound, dotted, line width=1.5pt, mark size=3pt, mark=star, mark options={fill=none, solid, evalBound}, forget plot]
table[]{2.           0.49941203
4.           0.75002645
8.           0.92783232
16.           1.10898104
32.           1.27877275
};

\addplot [color=upperBound, dashed, line width=1.5pt, mark size=2pt, mark=square, mark options={fill=none, solid, upperBound}, forget plot]
table[]{2.           0.49652995
4.           0.74941606
8.           0.92781523
16.           1.10895333
32.           1.278711
};

\addplot [color=worstCF, line width=1.5pt, mark size=2pt, mark=*, mark options={fill=none, solid, worstCF}, forget plot]
table[]{2.           0.48098566
4.           0.74091865
8.           0.92193548
16.           1.10171575
32.           1.26760118
};

\addplot [color=aveCF, line width=1.5pt, mark size=3pt, mark=+, mark options={fill=none, solid, aveCF}, forget plot]
table[]{2.           0.339943495893510028
4.           0.639118591806961955
8.           0.841246128452684205
16.           0.773609407584889497
32.           0.612154420309807157
};

\addplot [color=lowerBound, dotted, line width=1.5pt, mark size=2.5pt, mark=diamond, mark options={fill=none, solid, lowerBound}, forget plot]
table[]{2.           0.49362265
4.           0.74836379
8.           0.92772983
16.           1.1088148
32.           1.2784024
};

\addplot [color=black, line width=1pt, dotted, forget plot]
table[]{
2.0 1.0
32.0 1.0
};

\end{axis}
\end{tikzpicture}
            \caption{SDIRK1, $\delta t \approx \tfrac{dx}{c^2}$}
            \label{fig:wave0:ratio1:sdirk1}
        \end{subfigure}\qquad
        \begin{subfigure}[c]{0.3\textwidth}
            \begin{tikzpicture}

\begin{axis}[width=\figurewidth,
height=0.849\figureheight,
at={(0\figurewidth,0\figureheight)},
scale only axis,
unbounded coords=jump,
xmin=2,
xmax=32,
xtick={2, 4, 8, 16, 32},
xlabel style={font=\color{white!15!black}},
xlabel={\footnotesize Coarsening factor $k$},
ymode=log,
ymin=1e-1,
ymax=1e3,
ytick={1e-1,1e0,1e1,1e2,1e3},
yminorticks=true,
ylabel style={font=\color{white!15!black}},
axis background/.style={fill=white}
]

\addplot [color=evalSingleIt, dashed, line width=1.5pt, mark size=2pt, mark=triangle, mark options={fill=none, solid, evalSingleIt}, forget plot]
table[]{2.          170.02453013
4.          150.48136134
8.          112.3179315
16.           82.04816321
32.           62.84980723
};

\addplot [color=evalBound, dotted, line width=1.5pt, mark size=3pt, mark=star, mark options={fill=none, solid, evalBound}, forget plot]
table[]{2.          120.22549855
4.           75.24068139
8.           39.71038626
16.           20.51204161
32.           11.11038211
};

\addplot [color=upperBound, dashed, line width=1.5pt, mark size=2pt, mark=square, mark options={fill=none, solid, upperBound}, forget plot]
table[]{2.           9.85726955
4.          10.55481098
8.           9.51259576
16.           8.30629624
32.           7.27606438
};

\addplot [color=worstCF, line width=1.5pt, mark size=2pt, mark=*, mark options={fill=none, solid, worstCF}, forget plot]
table[]{2.           6.68257204
4.           6.63257905
8.           6.05050882
16.           5.69621965
32.           5.74886084
};

\addplot [color=aveCF, line width=1.5pt, mark size=3pt, mark=+, mark options={fill=none, solid, aveCF}, forget plot]
table[]{2.           2.289113537394567466
4.           4.364167422868654000
8.           3.333342441038491089
16.           1.483771054864120309
32.           1.102268632958060435
};

\addplot [color=lowerBound, dotted, line width=1.5pt, mark size=2.5pt, mark=diamond, mark options={fill=none, solid, lowerBound}, forget plot]
table[]{2.           7.95526033
4.           8.62195365
8.           7.86181722
16.           6.96266678
32.           6.26011696
};

\addplot [color=black, line width=1pt, dotted, forget plot]
table[]{
2.0 1.0
32.0 1.0
};

\end{axis}
\end{tikzpicture}            \caption{SDIRK2, $\delta t \approx \tfrac{dx}{c^2}$}
            \label{fig:wave0:ratio1:sdirk2}
        \end{subfigure}\quad
        \begin{subfigure}[c]{0.3\textwidth}
            \begin{tikzpicture}

\begin{axis}[width=\figurewidth,
height=0.849\figureheight,
at={(0\figurewidth,0\figureheight)},
scale only axis,
unbounded coords=jump,
xmin=2,
xmax=32,
xtick={2, 4, 8, 16, 32},
xlabel style={font=\color{white!15!black}},
xlabel={\footnotesize Coarsening factor $k$},
ymode=log,
ymin=1e-1,
ymax=1e2,
ytick={1e-1,1e0,1e1,1e2},
yminorticks=true,
ylabel style={font=\color{white!15!black}},
axis background/.style={fill=white}
]

\addplot [color=evalSingleIt, dashed, line width=1.5pt, mark size=2pt, mark=triangle, mark options={fill=none, solid, evalSingleIt}, forget plot]
table[]{2.           1.60502419
4.           3.78667791
8.           6.71584194
16.          10.07148374
32.          14.38300871
};

\addplot [color=evalBound, dotted, line width=1.5pt, mark size=3pt, mark=star, mark options={fill=none, solid, evalBound}, forget plot]
table[]{2.           1.14123404
4.           1.92497802
8.           2.40252975
16.           2.52369263
32.           2.5434107
};

\addplot [color=upperBound, dashed, line width=1.5pt, mark size=2pt, mark=square, mark options={fill=none, solid, upperBound}, forget plot]
table[]{2.           1.14120879
4.           1.92496878
8.           2.40251357
16.           2.5236341
32.           2.54318365
};

\addplot [color=worstCF, line width=1.5pt, mark size=2pt, mark=*, mark options={fill=none, solid, worstCF}, forget plot]
table[]{2.           1.13524499
4.           1.91811779
8.           2.58152877
16.           2.65427686
32.           2.98000011
};

\addplot [color=aveCF, line width=1.5pt, mark size=3pt, mark=+, mark options={fill=none, solid, aveCF}, forget plot]
table[]{2.           1.050400374630146461
4.           1.836431081221350459
8.           2.196678964058615957
16.           1.169494560830936614
32.           1.013159840235413833
};

\addplot [color=lowerBound, dotted, line width=1.5pt, mark size=2.5pt, mark=diamond, mark options={fill=none, solid, lowerBound}, forget plot]
table[]{2.           1.14108253
4.           1.92492259
8.           2.4024327
16.           2.52334152
32.           2.54204936
};

\addplot [color=black, line width=1pt, dotted, forget plot]
table[]{
2.0 1.0
32.0 1.0
};

\end{axis}
\end{tikzpicture}
            \caption{SDIRK3, $\delta t \approx \tfrac{dx}{c^2}$}
            \label{fig:wave0:ratio1:sdirk3}
        \end{subfigure}
        }
        \begin{tikzpicture}

\begin{axis}[width=\figurewidth,
height=0.849\figureheight,
at={(0\figurewidth,0\figureheight)},
hide axis,
xmin=2,
xmax=6,
xtick={2, 3, 4, 5, 6},
xlabel style={font=\color{white!15!black}},
xlabel={Number of levels},
ymode=log,
ymin=1e-07,
ymax=10,
yminorticks=true,
ylabel style={font=\color{white!15!black}},
ylabel={Convergence factor},
axis background/.style={fill=white},
legend style={at={(0.6,0)}, anchor=south, legend columns=3,
draw=none, fill=none, legend cell align=left}
]

\addlegendimage{color=upperBound, dashed, line width=1.5pt, mark size=2pt, mark=square, mark options={fill=none, solid, upperBound}}
\addlegendimage{color=evalBound, dotted, line width=1.5pt, mark size=3pt, mark=star, mark options={solid, evalBound}}
\addlegendimage{color=evalSingleIt, dashed, line width=1.5pt, mark size=2pt, mark=triangle, mark options={fill=none, solid, evalSingleIt}}
\addlegendimage{color=worstCF, line width=1.5pt, mark size=2pt, mark=*, mark options={solid, worstCF}}
\addlegendimage{color=aveCF, line width=1.5pt, mark size=3pt, mark=+, mark options={solid, aveCF}}
\addlegendimage{color=lowerBound, dotted, line width=1.5pt, mark size=2.5pt, mark=diamond, mark options={fill=none, solid, lowerBound}}

\addlegendentry{~Upper bound}
\addlegendentry{~Eval bound\quad~}
\addlegendentry{~Eval single it\quad~}
\addlegendentry{~Worst CF}
\addlegendentry{~Ave CF}
\addlegendentry{~Lower bound}

\end{axis}
\end{tikzpicture}    \end{center}
    \caption{Eigenvalue convergence bounds and upper/lower bounds from~\cite[Equation~(63)]{southworth19}
        compared to observed worst-case and average convergence factors.}
    \label{fig:wave0}
\end{figure}
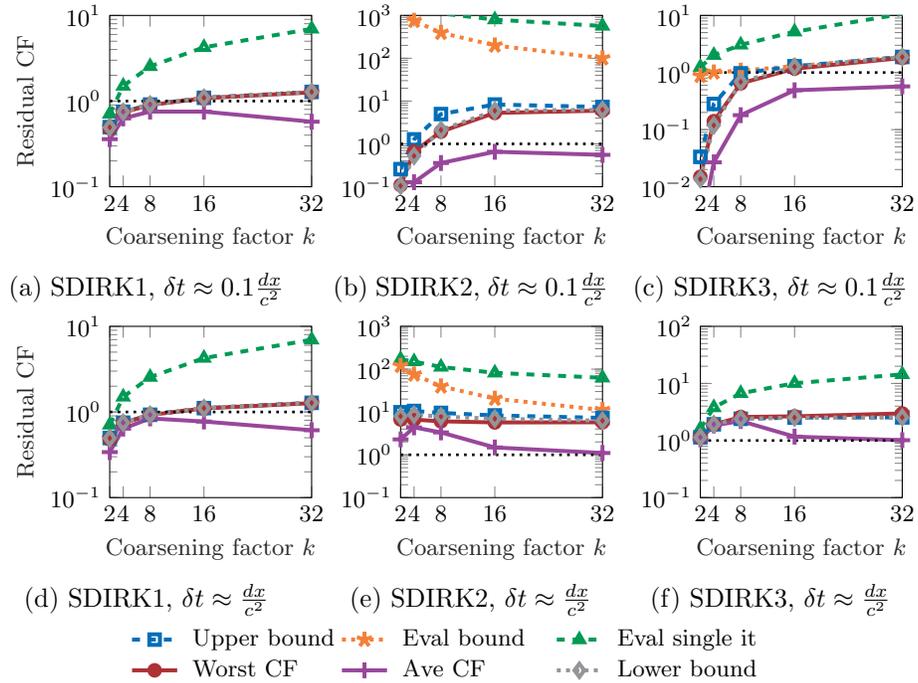

It turns out for this particular problem and discretization, convergence of
Parareal/MGRIT requires an almost explicit restriction on time-step size. To that
end, it is not likely to be competitive in practice vs. sequential time stepping.
Nevertheless, our eigenvalue-based convergence theory provides a very accurate 
measure of convergence. In all cases, the theoretical lower and upper bound from
\Cref{th:diag_tight} perfectly contain the worst observed convergence factor, with
only a very small difference between the theoretical upper and lower bounds.

There are a number of other interesting things to note from \Cref{fig:wave0}.
First, both the general eigenvalue bound (``Eval bound'') and single iteration
bound from \Cref{cor:new2grid} (''Eval single it'') are very pessimistic,
sometimes many orders of magnitude larger than the worst observed convergence
factor. As discussed previously and first demonstrated in \cite{19c_mgrit},
the size of the coarse problem, $N_c$, is particularly relevant for problems
with imaginary eigenvalues. Although the lower and upper bound converge to
the ``Eval bound'' as $N_c\to\infty$, for small to moderate $N_c$, reasonable
convergence may be feasible in practice even if the limiting bound $\gg 1$. 
Due to the relevance of $N_c$, it is difficult to comment on the single
iteration bound (because as derived, we let $N_c\to\infty$ for an upper
single-iteration bound). However, we note that the upper bound on all other
iterations (``Upper bound'') appears to bound the worst-observed convergence
factor, so the single-iteration bound on worst-case convergence appears
not to be realized in practice, at least for this problem. 
It is also worth pointing out the difference between time-integration schemes
in terms of convergence,
with backward Euler being the most robust, followed by SDIRK3, and last
SDIRK2. As discussed in \cite{19c_mgrit} for normal spatial operators, 
not all time-integration schemes are equal when it comes to convergence
of Parareal and MGRIT, and robust convergence theory provides important
information on choosing effective integration schemes.

\begin{remark}[Superlinear convergence]
In \cite{Gander:2007jq}, superlinear convergence of the Pareal algorithm is
observed and discussed. We see similar behavior in \Cref{fig:wave0}, where
the worst observed convergence factor can be more than $10\times$ larger
than the average convergence factor. In general, superlinear convergence 
would be expected as the exactness property of Parareal/MGRIT is approached.
In addition, for non-normal operators, although the slowest converging
mode may be observed in practice during some iteration, it does not
necessarily stay in the error spectrum (and thus continue to yield
slow(est) convergence) as iterations progress, due to the inherent 
rotation in powers of non-normal operators. The nonlinear setting
introduces additional complications that may explain such behavior.
\end{remark}

\subsection{Test case: DG advection (diffusion)} \label{sec:adv}

Now we consider the time-dependent advection-diffusion equation,
\begin{align}
\frac{\partial u}{\partial t} + \mathbf{v} \cdot \nabla u - \nabla\cdot(\alpha \nabla u)  = f,
\end{align}
on a 2D unit square domain discretized with linear discontinuous Galerkin elements.
The discrete spatial operator, $L$, for this problem is defined by $L = M^{-1}K$,
where $M$ is a mass matrix, and $K$ is the stiffness matrix associated with
$\mathbf{v}\cdot \nabla u - \nabla\cdot(\alpha \nabla u)$. The length of the time domain is
always chosen to maintain the appropriate relationship between the spatial and temporal
step sizes while also keeping 100 time points on the MGRIT coarse grid (for a variety
of coarsening factors). Throughout, the initial condition is $u(\mathbf{x},0) = \mathbf{0}$
and the following time-dependent forcing function is used:
\begin{align}
f(\mathbf{x},t) = \begin{cases}
\cos^2( 2\pi  t / t_{final} )\,, & \mathbf{x}\in[1/8,3/8]\times[1/8,3/8] \\
0\, , & \text{else}.
\end{cases}
\end{align}
Backward Euler and a third-order, three-stage L-stable SDIRK scheme, denoted SDIRK3
are applied in time, and FCF-relaxation is used for all tests. Three different
velocity fields are studied:
\begin{align}
\mathbf{v}_1(\mathbf{x},t) = (\sqrt{2/3}, \sqrt{1/3}), \\
\mathbf{v}_2(\mathbf{x},t) = (\cos(\pi y)^2,\cos(\pi x)^2), \\
\mathbf{v}_3(\mathbf{x},t) = (y\pi/2, -x\pi/2),
\end{align}
referred to as problems 1, 2, and 3 below. Note that problem 1 is a simple translation,
problem 2 is a curved but non-recirculating flow, and problem 3 is a recirculating flow.
The relative strength of the diffusion term is also varied in the results below, including
a diffusion-dominated case, $\alpha = 10d_x$, an advection-dominated case, $\alpha = 0.1d_x$,
and a pure advection case, $\alpha = 0$.
{When backward Euler is used, the time step is set equal to the spatial step, $d_t = d_x$, while for SDIRK3, $d_t^3 = d_x$, in order to obtain similar accuracy in both time and space.}

\Cref{fig:adv_diff_cf} shows various computed bounds compared with observed worst-case and average convergence factors (over 10 iterations) vs. MGRIT coarsening factor for each problem variation with different amounts of diffusion using backward Euler for time integration.
The bounds shown are the ``GSVD bound'' from \Cref{th:conv}, the ``Eval bound,'' an equivalent eigenvalue form of this bound (see \cite[Theorem 13]{southworth19}, and ``Eval single it,'' which is the bound from \Cref{cor:new2grid} as $N_c\to\infty$.
The problem size is $n = 16$, where the spatial mesh has $n\times n$ elements.
This small problem size allows the bounds to be directly calculated: for the GSVD bound, it is possible to compute $||(\Psi - \Phi^k)(I - e^{\mathrm{i}x}\Psi)^{-1}||$, and for the eigenvalue bounds, it is possible to compute the complete eigenvalue decomposition of the spatial operator, $L$, and apply proper transformations to the eigenvalues to obtain eigenvalues of $\Phi$ and $\Psi$.

\begin{figure}[htb!]
\centering
\includegraphics[height=0.21\textwidth]{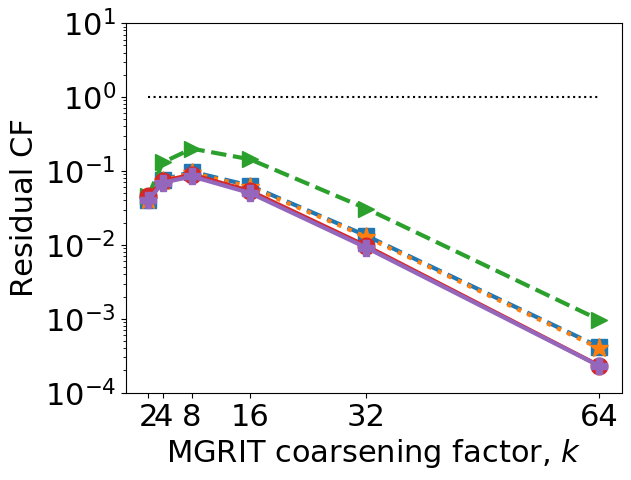}
\includegraphics[height=0.21\textwidth]{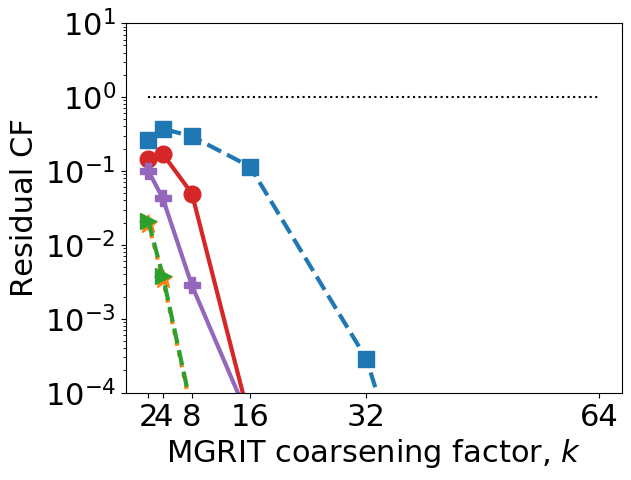}
\includegraphics[height=0.21\textwidth]{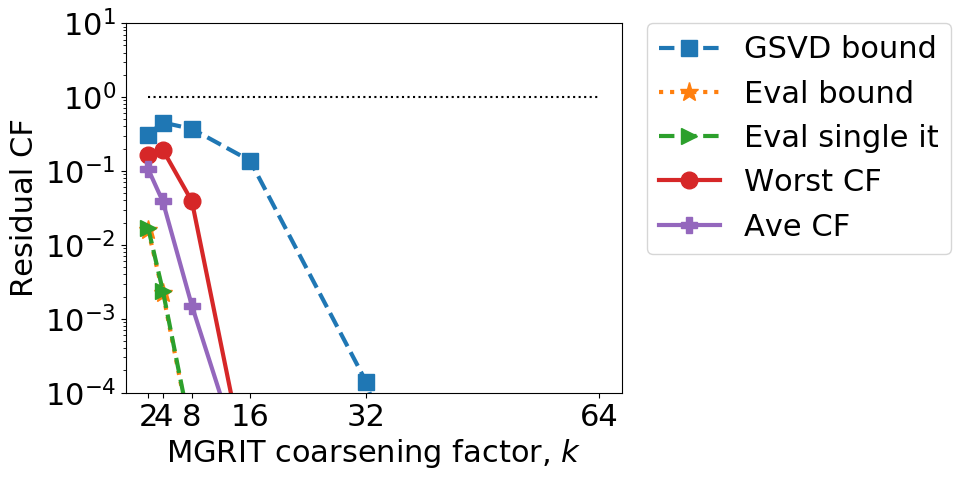}

\includegraphics[height=0.21\textwidth]{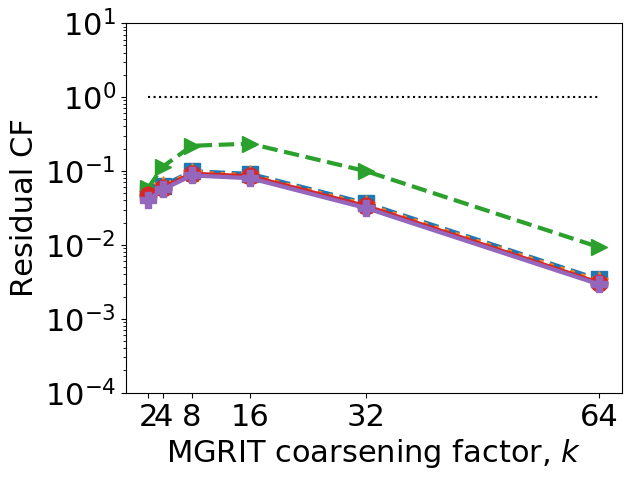}
\includegraphics[height=0.21\textwidth]{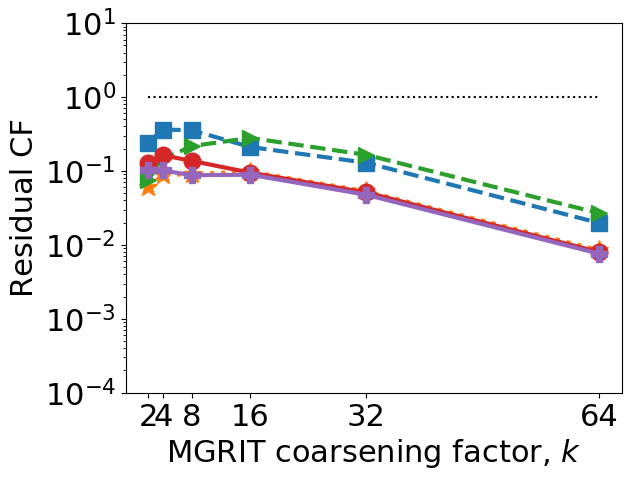}
\includegraphics[height=0.21\textwidth]{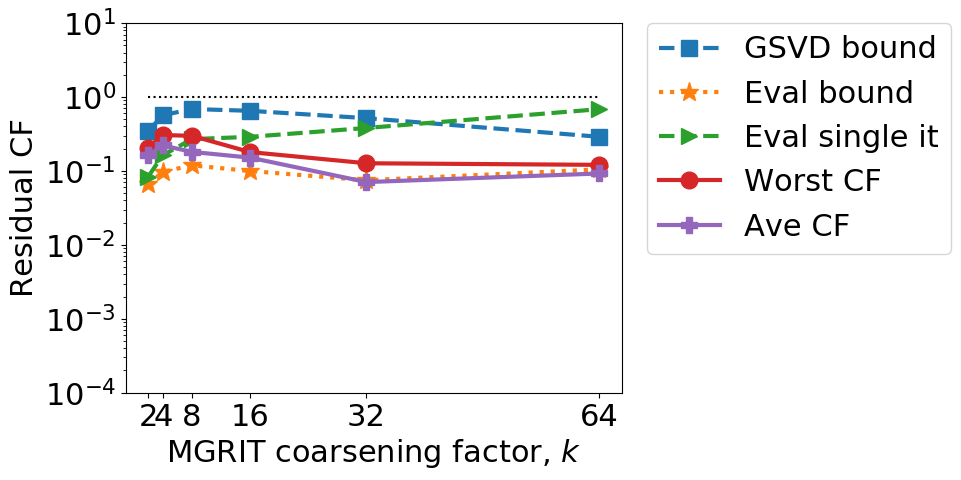}

\includegraphics[height=0.21\textwidth]{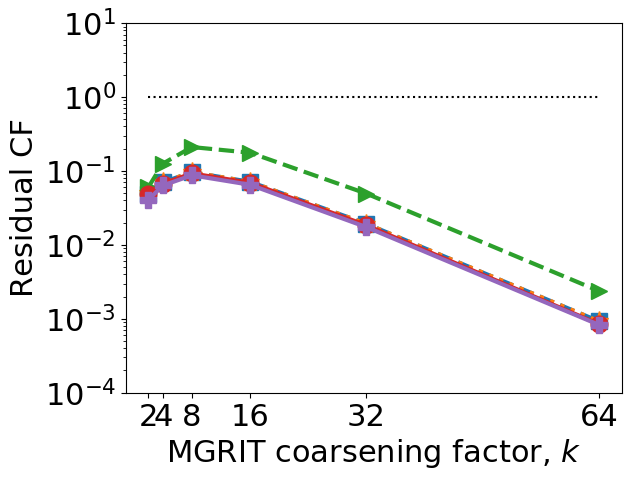}
\includegraphics[height=0.21\textwidth]{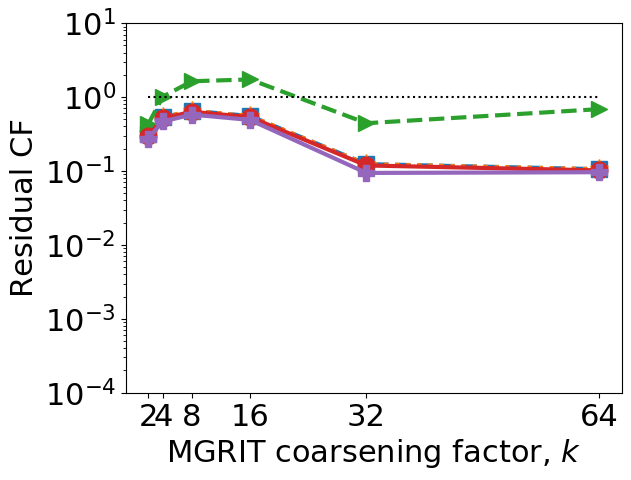}
\includegraphics[height=0.21\textwidth]{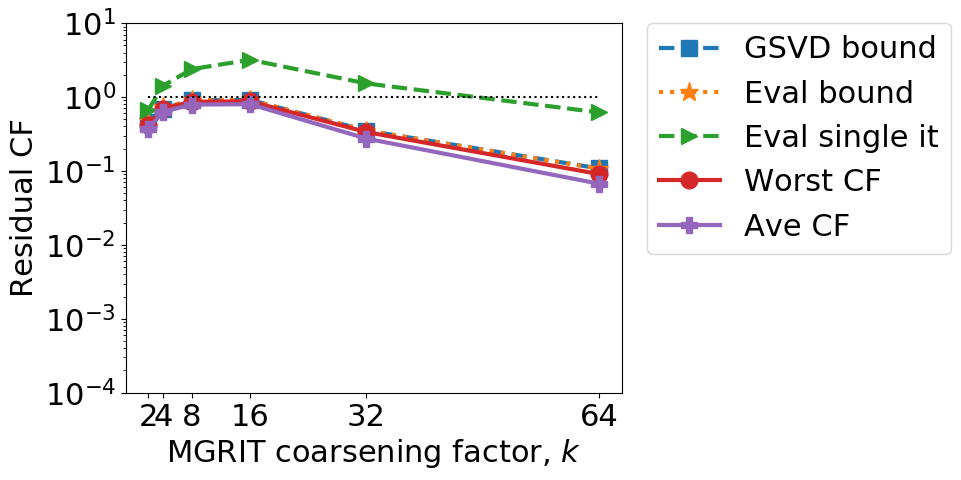}
\caption{Bounds and observed convergence factor vs. MGRIT coarsening factor with $n = 16$ using backward Euler for problem 1 (top), 2 (middle), and 3 (bottom) with $\alpha = 10d_x$ (left), $0.1d_x$ (middle), and 0 (right). }
\label{fig:adv_diff_cf}
\end{figure}

In the diffusion-dominated case (left column of \Cref{fig:adv_diff_cf}), the GSVD and eigenvalue bounds agree well with each other (because the spatial operator is nearly symmetric), accurately predicting observed residual convergence factors for all problems. Similar to \Cref{sec:wave}, the single iteration bound from \Cref{cor:new2grid} does not appear to be realized in practice.

In the advection-dominated and pure advection cases (middle and right columns of \Cref{fig:adv_diff_cf}), behavior of the bounds and observed convergence depends on the type of flow. In the non-recirculating examples, the GSVD bounds are more pessimistic compared to observed convergence, but still provide an upper bound on worst-case convergence, as expected. Conversely, the eigenvalue bounds on worst-case convergence become unreliable, sometimes undershooting the observed convergence factors by significant amounts.
Recall that the eigenvalue bounds are tight in the $(UU^*)^{-1}$-norm of the error, where $U$ is the matrix of eigenvectors. However, for the non-recirculating problems, the spatial operator $L$ is defective to machine precision, that is, the eigenvectors are nearly linearly dependent and $U$ is close to singular. Thus, tight bounds on convergence in the $(UU^*)^{-1}$-norm are unlikely to provide an accurate measure of convergence in more standard norms, such as $\ell^2$.
In the recirculating case, $UU^*$ is well conditioned. Then, similarly to the wave equation in \Cref{sec:wave}, the eigenvalue bounds should provide a good approximation to the $\ell^2$-norm, and, indeed, the GSVD and eigenvalue bounds agree well with each other accurately predict residual convergence factors.

\Cref{fig:adv_diff_cf_sdirk3} shows the same set of results but using SDIRK3 for time integration instead of backward Euler and with a large time step $d_t^3 = d_x$ to match accuracy in the temporal and spatial domains.
Results here are qualitatively similar to those of backward Euler, although MGRIT convergence (both predicted and measured) is generally much swifter, especially for larger coarsening factors.
Again, the GSVD and eigenvalue bounds accurately predict observed convergence in the diffusion-dominated case.
In the advection-dominated and pure advection cases, again the eigenvalue bounds are not reliable for the non-recirculating flows, but all bounds otherwise accurately predict the observed convergence.

\begin{figure}[htb!]
\centering
\includegraphics[height=0.21\textwidth]{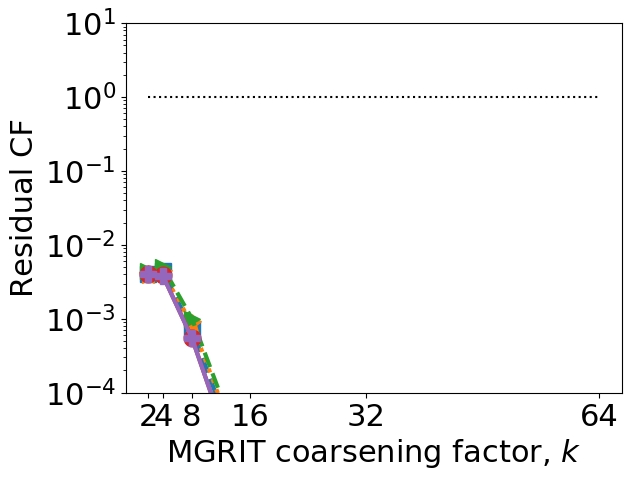}
\includegraphics[height=0.21\textwidth]{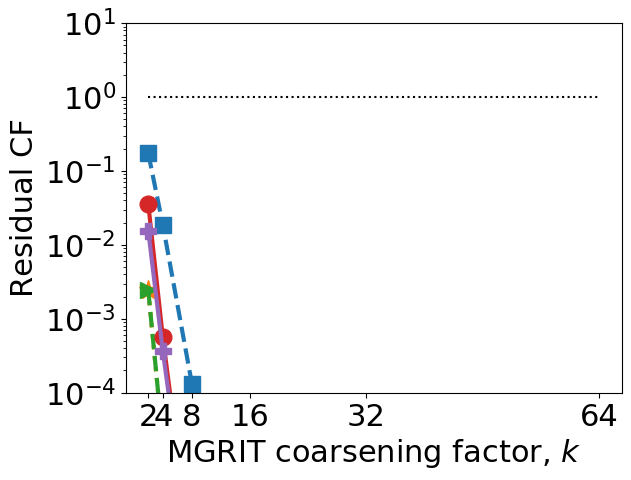}
\includegraphics[height=0.21\textwidth]{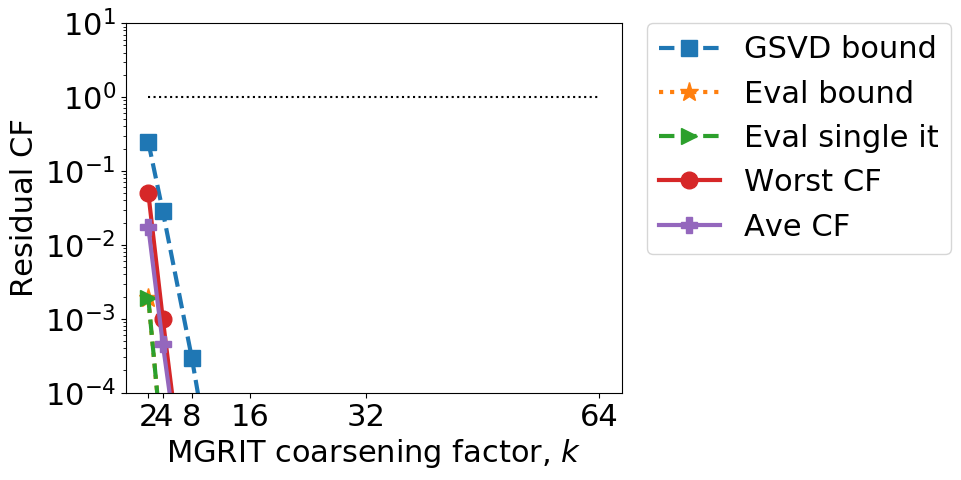}

\includegraphics[height=0.21\textwidth]{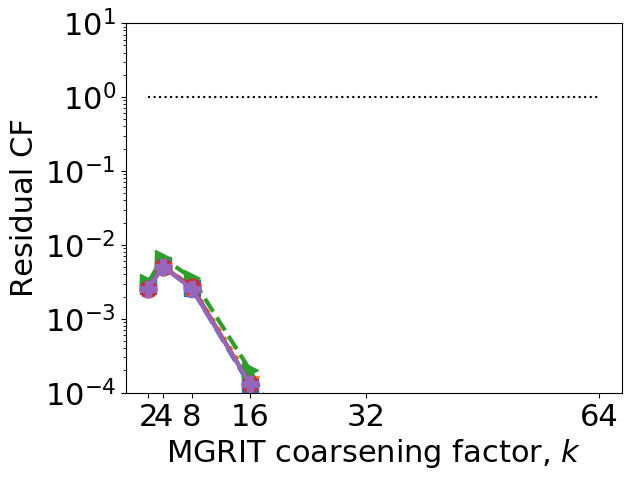}
\includegraphics[height=0.21\textwidth]{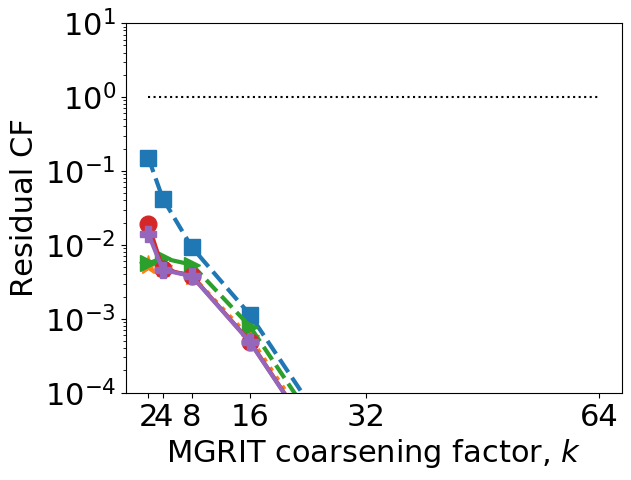}
\includegraphics[height=0.21\textwidth]{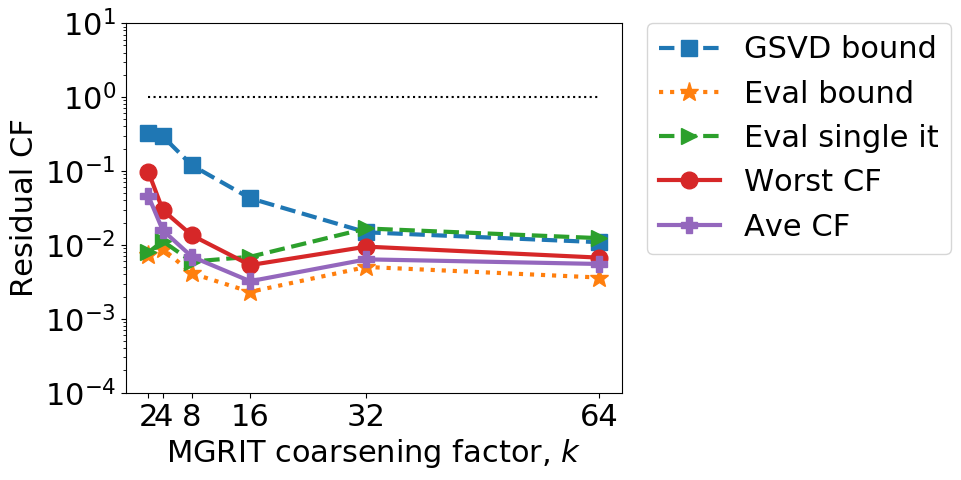}

\includegraphics[height=0.21\textwidth]{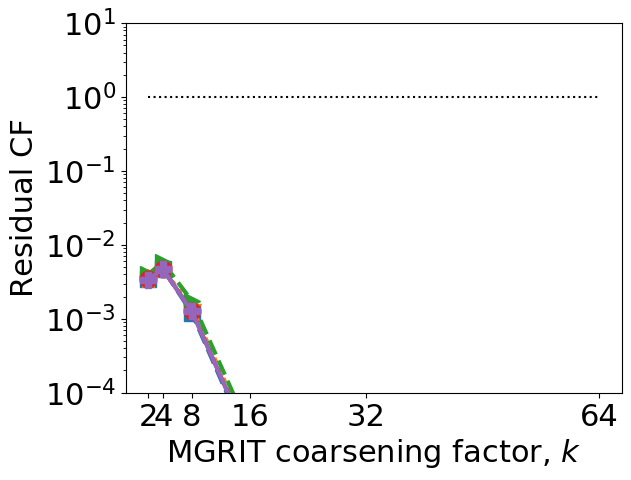}
\includegraphics[height=0.21\textwidth]{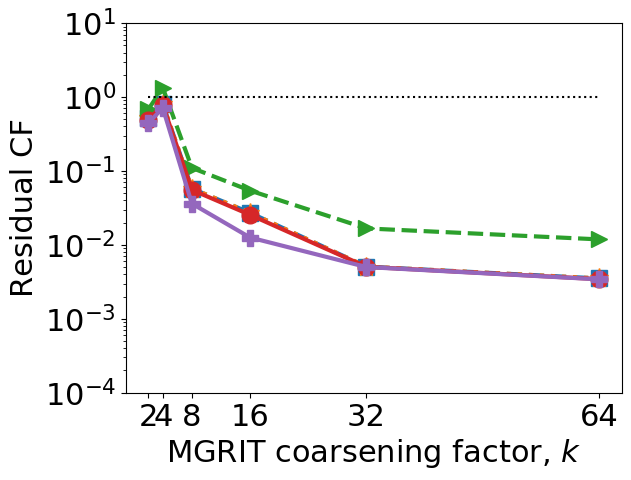}
\includegraphics[height=0.21\textwidth]{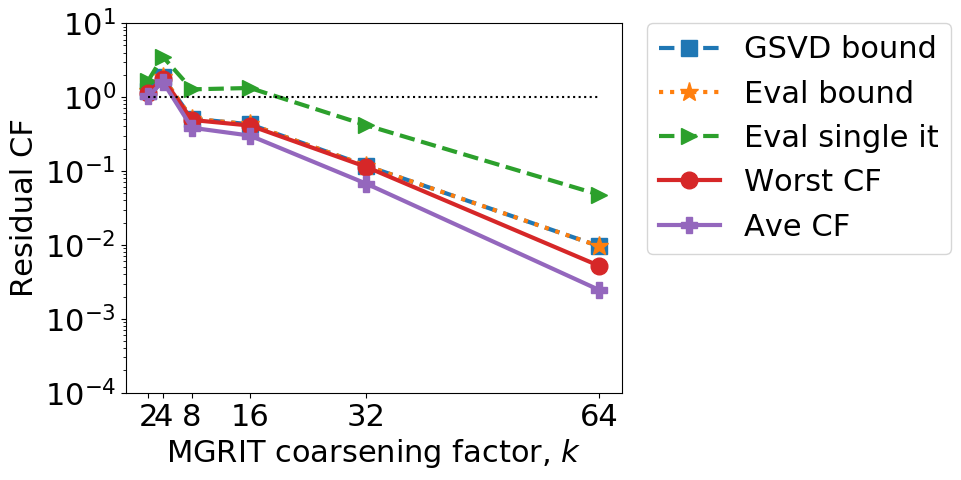}
\caption{Bounds and observed convergence factor vs. MGRIT coarsening factor with $n = 16$ using SDIRK3 {with $d_t^3 = d_x$} for problem 1 (top), 2 (middle), and 3 (bottom) with $\alpha = 10d_x$ (left), $0.1d_x$ (middle), and 0 (right). }
\label{fig:adv_diff_cf_sdirk3}
\end{figure}

\Cref{fig:adv_diff_large} shows results for a somewhat larger problem, with spatial
mesh of $512\times 512$ elements, to compare convergence estimates computed directly
on a small problem size with observed convergence on more practical problem sizes.
The length of the time domain is set according to the MGRIT coarsening factor such that there are always four levels in the MGRIT hierarchy (or again 100 time points on the coarse grid in the
case of two-level) while preserving the previously stated time step to spatial step
relationships. Although less tight than above, convergence estimates on the
small problem size provide reasonable estimates of the larger problem in many
cases, particularly for smaller coarsening factors. The difference in larger
coarsening factors is likely because, e.g., $\Phi^{16}$ for a $16\times 16$ mesh
globally couples elements, while $\Phi^{16}$ for a $512\times 512$ mesh remains a
fairly sparse matrix. That is, the mode decomposition of $\Phi^{16}$ for $n=16$
is likely a poor representation for $n=512$.

\begin{figure}[htb!]
\centering
\includegraphics[height=0.21\textwidth]{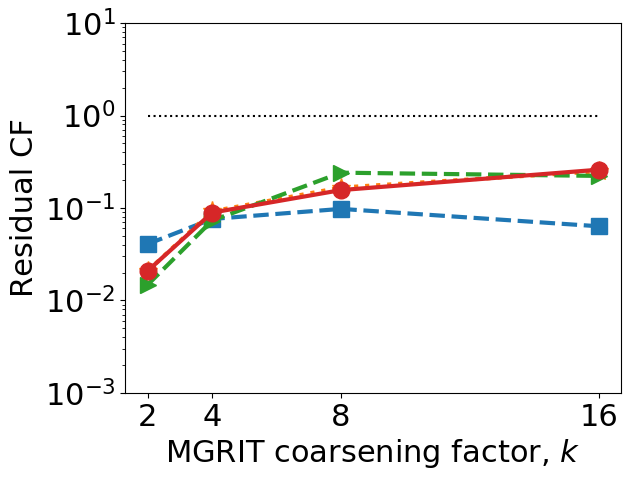}
\includegraphics[height=0.21\textwidth]{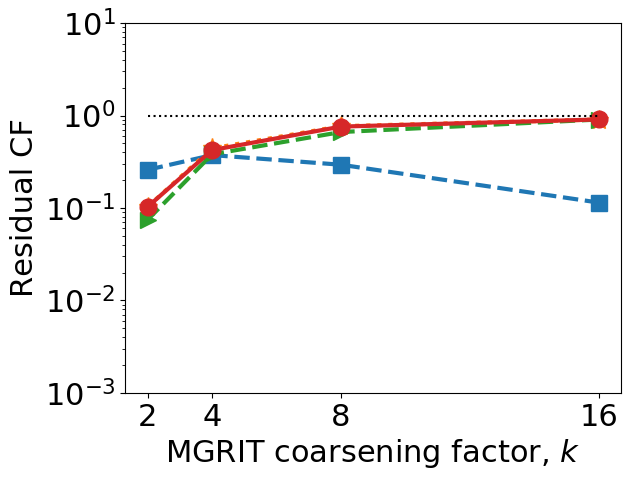}
\includegraphics[height=0.21\textwidth]{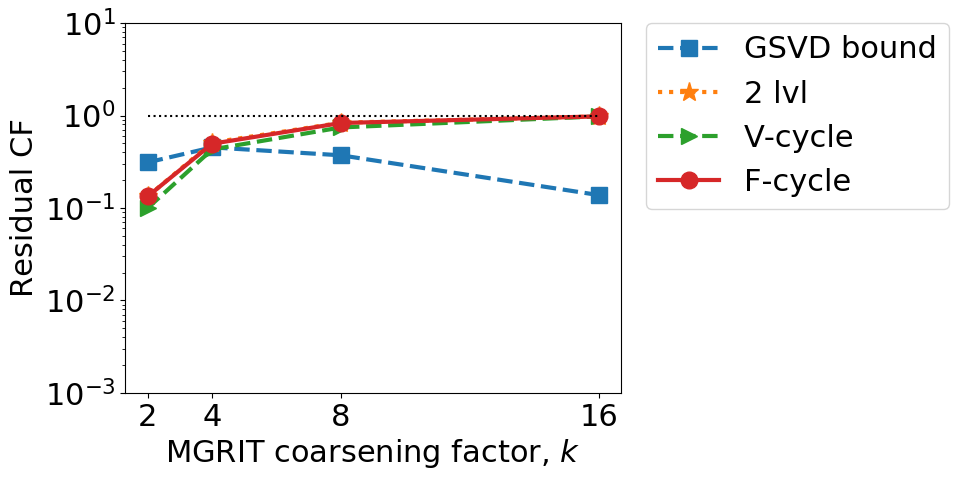}

\includegraphics[height=0.21\textwidth]{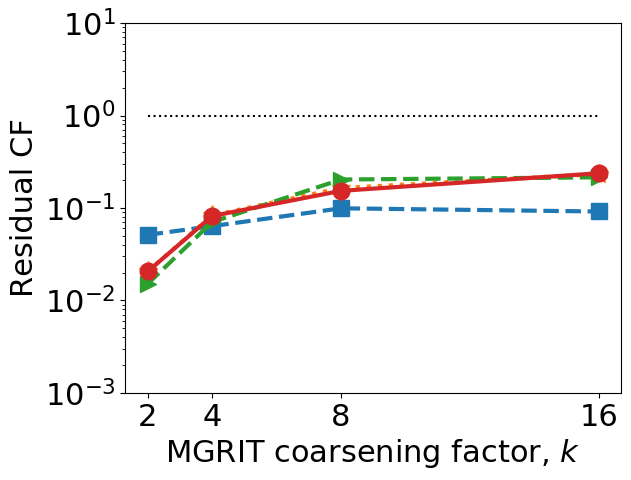}
\includegraphics[height=0.21\textwidth]{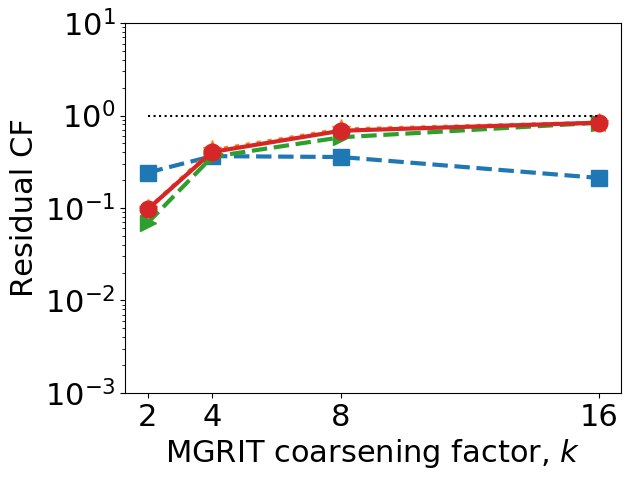}
\includegraphics[height=0.21\textwidth]{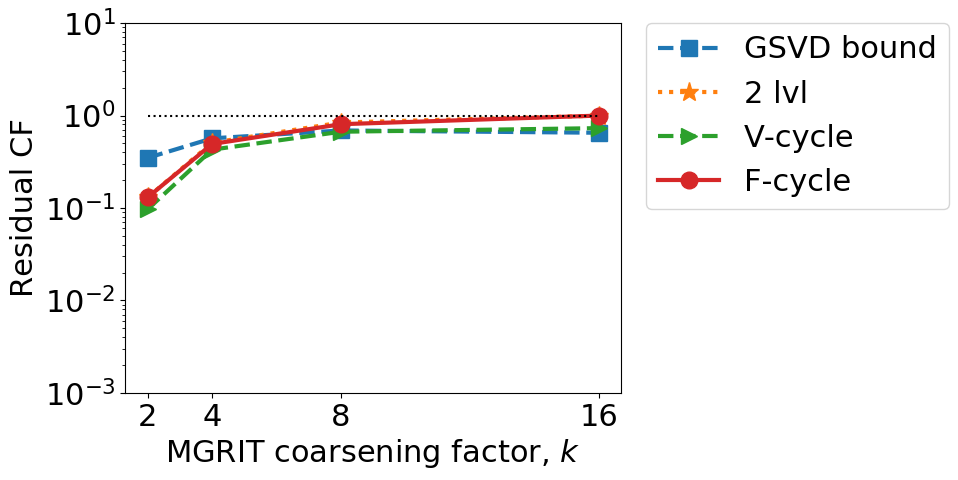}

\includegraphics[height=0.21\textwidth]{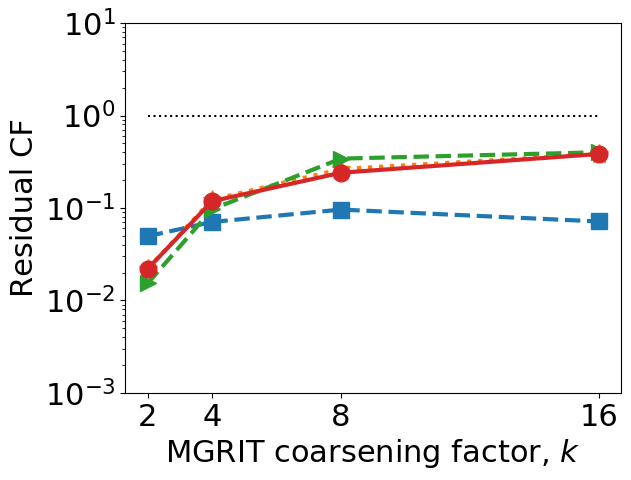}
\includegraphics[height=0.21\textwidth]{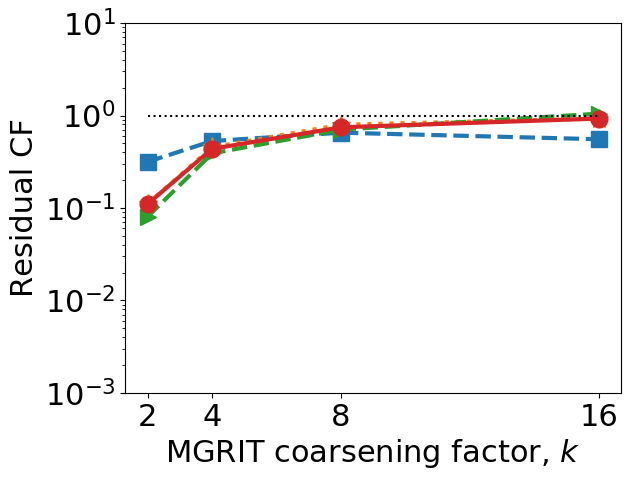}
\includegraphics[height=0.21\textwidth]{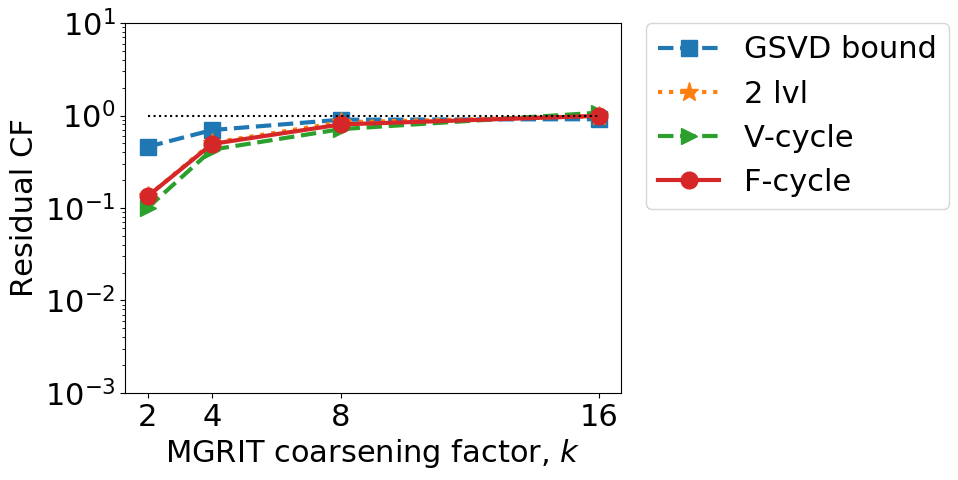}
\caption{GSVD bound for $n = 16$ vs. observed convergence factors for different cycle structures with $n = 512$ plotted against MGRIT coarsening factor using backward Euler for problem 1 (top), 2 (middle), and 3 (bottom) with $\alpha = 1$ (left), 0.01 (middle), and 0.0 (right). }
\label{fig:adv_diff_large}
\end{figure}

Finally, we give insight in how the minimum over $x$ is realized in the TAP.
\Cref{fig:adv_diff_x_11} and \Cref{fig:adv_diff_x_13} show the GSVD bounds (i.e.
$\varphi_{FCF}$) as a function of $x$, for backward Euler and SDIRK3,
respectively, and for each of the three problems and diffusion coefficients. A
downside of the GSVD bounds in practice is the cost of computing $||(\Psi -
\Phi^k)(I - e^{\mathrm{i}x}\Psi)^{-1}||$ for many values of $x$. As shown,
however, for the diffusion-dominated (nearly symmetric) problems, simply
choosing $x=0$ is sufficient. Interestingly, SDIRK3 bounds have almost no
dependence on $x$ for any problems, while backward Euler bounds tend to have a
non-trivial dependence on $x$ (and demonstrate the symmetry in $x$ as discussed
in \Cref{cor:symm}). Nevertheless, accurate bounds for nonsymmetric problems do
require sampling a moderate spacing of $x\in[0,\pi]$ to achieve a realistic
bound.

\begin{figure}[htb!]
\centering
\begin{subfigure}[t]{0.33\textwidth}
  \includegraphics[width=\textwidth]{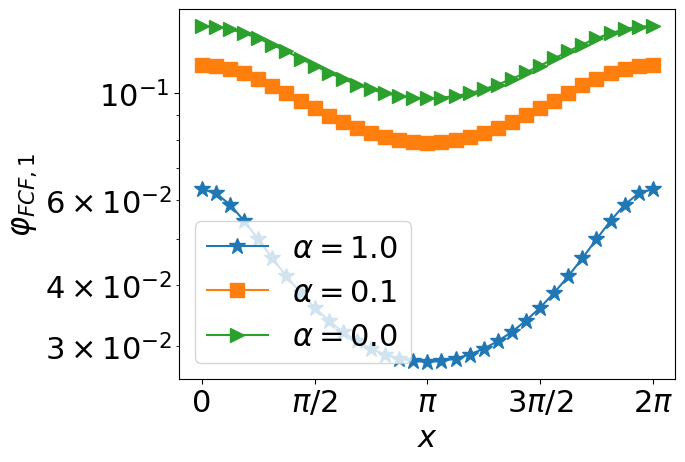}
  \caption{Problem 1}
\end{subfigure}
\begin{subfigure}[t]{0.32\textwidth}
  \includegraphics[width=\textwidth]{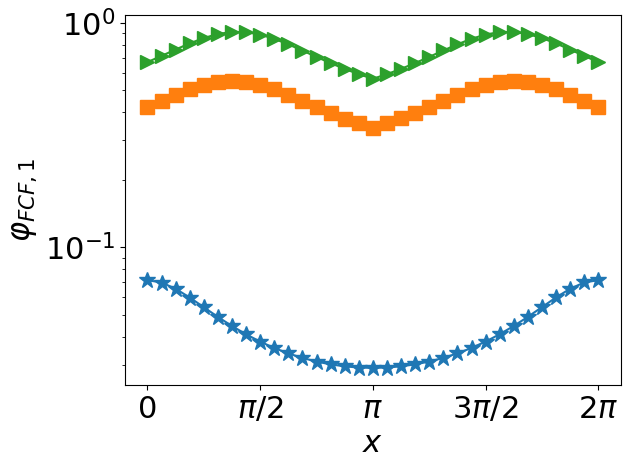}
  \caption{Problem 2}
\end{subfigure}
\begin{subfigure}[t]{0.32\textwidth}
  \includegraphics[width=\textwidth]{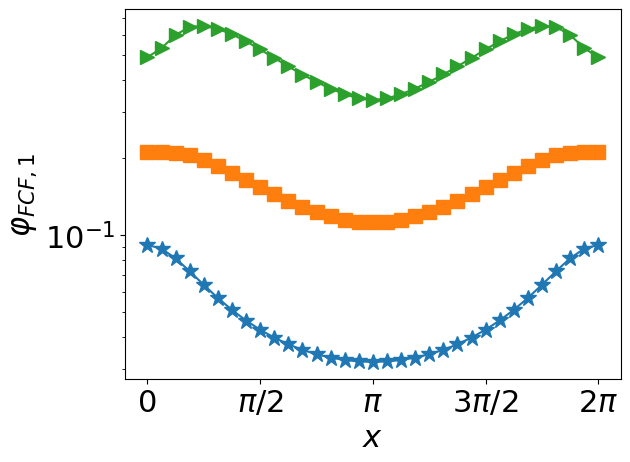}
  \caption{Problem 3}
\end{subfigure}

\caption{GSVD bounds ($\varphi_{FCF,1}$) vs. choice of $x$ with $n = 16$ and MGRIT coarsening factor 16 using backward Euler for problem 1 (top), 2 (middle), and 3 (bottom) with $\alpha = 1$ (left), 0.01 (middle), and 0.0 (right). }
\label{fig:adv_diff_x_11}
\end{figure}

\begin{figure}[htb!]
\centering
\begin{subfigure}[t]{0.33\textwidth}
   \includegraphics[width=\textwidth]{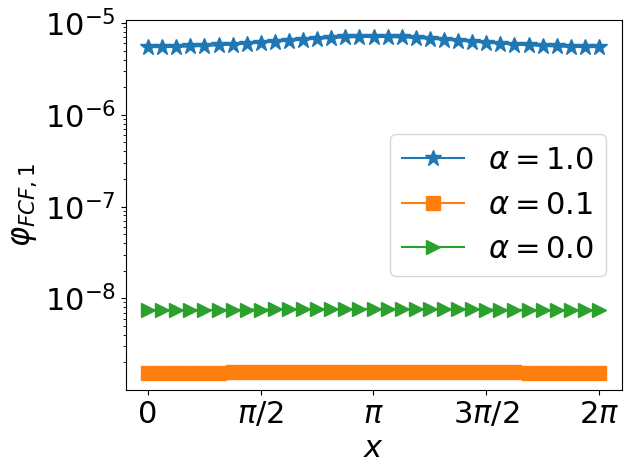}
   \caption{Problem 1}
\end{subfigure}
\begin{subfigure}[t]{0.32\textwidth}
   \includegraphics[width=\textwidth]{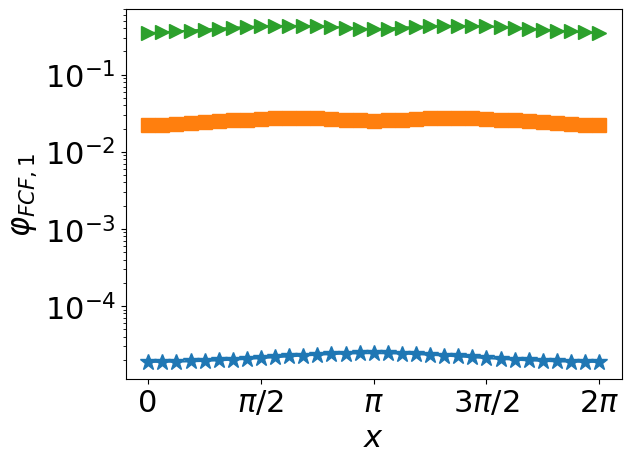}
   \caption{Problem 2}
\end{subfigure}
\begin{subfigure}[t]{0.32\textwidth}
   \includegraphics[width=\textwidth]{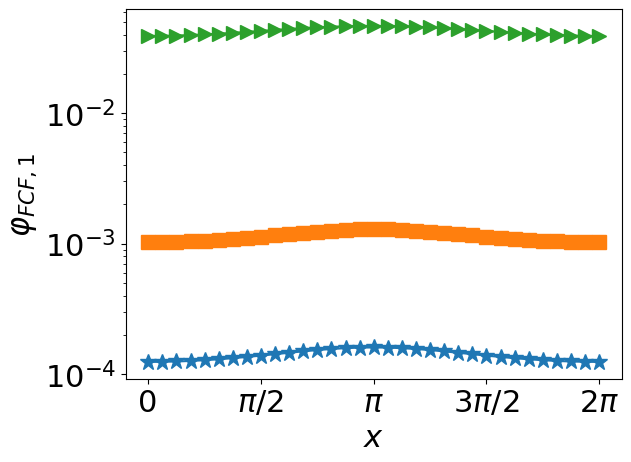}
   \caption{Problem 3}
\end{subfigure}

\caption{GSVD bounds ($\varphi_{FCF,1}$) vs. choice of $x$ with $n = 16$ and MGRIT coarsening factor 16 using SIDRK3 for problem 1 (top), 2 (middle), and 3 (bottom) with $\alpha = 1$ (left), 0.01 (middle), and 0.0 (right). }
\label{fig:adv_diff_x_13}
\end{figure}

\section{Conclusion} \label{sec:conc}

This paper furthers the theoretical understanding of convergence Parareal and MGRIT. 
A new, simpler derivation of measuring error and residual propagation operators is
provided, which applies to linear time-dependent operators, and which may be a good
starting point to develop improved convergence theory for the time-dependent setting.
Theory from \cite{southworth19} on spatial operators that are independent of time is
then reviewed, and several new results are proven, expanding the understanding of
two-level convergence of linear MGRIT and Parareal. Discretizations of the two classical
linear hyperbolic PDEs, linear advection (diffusion) and the second-order wave equation,
are then analyzed and compared with the theory. Although neither naive implementation
yields the rapid convergence desirable in practice, the theory is shown to accurately
predict convergence on highly nonsymmetric and hyperbolic operators.

\section{Appendix: Proofs} \label{sec:proofs}

\begin{proof}[Proof of \Cref{th:new2grid}]

Here we consider a slightly modified coarsening of points: let the first $k$ points be F-points, followed by a C-point,
followed by $k$ F-points, and so on, finishing with a C-point (as opposed to starting with a C-point). This is simply a
theoretical tool that permits a cleaner analysis, but is not fundamental to the result. Then,
define the so-called ideal restriction operator, $R_{\textnormal{ideal}}$ via
\begin{align*}
R_{\textnormal{ideal}} & =
\begin{bmatrix} -A_{cf}(A_{ff})^{-1} & I \end{bmatrix} \\
&=
\left[\begin{array}{@{}c c c c c c c c c : c c c c@{}}
  \Phi^{k-1} & \Phi^{k-2} &... & \Phi &&&&& & I \\
   & & & & \ddots & & & & & & \ddots \\
  &&&&& \Phi^{k-1} & \Phi^{k-2} &... & \Phi & & & I \\
\end{array}\right].
\end{align*}
Let $\widetilde{P}$ be the orthogonal (column) block permutation matrix such that
\begin{align*}
R_{\textnormal{ideal}}\widetilde{P} & =
\begin{bmatrix}
  \Phi^{k-1} &... & \Phi & I &&&&& \\
  & & & & \ddots \\
  &&&&& \Phi^{k-1} &... & \Phi & I
\end{bmatrix}
:=
\begin{bmatrix}
\mathbf{W} \\ & \ddots \\ & & \mathbf{W}
\end{bmatrix},
\end{align*}
where $\mathbf{W}$ is the block row matrix $\mathbf{W} = ( \Phi^{k-1}, ..., \Phi, I)$. Then, from \eqref{eq:err_res}
and the fact that $\|\mathcal{E}_F\|_{A^*A} = \|\mathcal{R}_F\|$ \cite{southworth19}, the norm of residual and error
propagation of MGRIT with F-relaxation is given by
\begin{align*}
\|\mathcal{E}_F\|_{A^*A} = \|\mathcal{R}_F\| & = \left\| (I - A_\Delta B_\Delta^{-1})R_{\textnormal{ideal}}\right\|
  = \left\| (I - A_\Delta B_\Delta^{-1})R_{\textnormal{ideal}}\widetilde{P}\right\|,
\end{align*}
where
\begin{align*}
    &( I - A_\Delta B_\Delta^{-1})R_{\textnormal{ideal}}\widetilde{P}
    = \textnormal{diag}\left[\Psi - \Phi^k\right]
  \begin{bmatrix} \mathbf{0} \\ I & \mathbf{0} \\ \Psi & I & \mathbf{0} \\ \Psi^2 & \Psi & I & \mathbf{0} \\
  \vdots & \ddots & \ddots & \ddots  & \ddots\end{bmatrix}
  \begin{bmatrix} \mathbf{W} \\ & \ddots \\ & & \mathbf{W}\end{bmatrix} \\
    &    = \begin{bmatrix}
  \mathbf{0} \\
  (\Psi - \Phi^k)\mathbf{W} & \mathbf{0} \\
  (\Psi - \Phi^k)\Psi\mathbf{W} & (\Psi - \Phi^k)\mathbf{W} & \mathbf{0} \\
  (\Psi - \Phi^k)\Psi^2\mathbf{W} & (\Psi - \Phi^k)\Psi\mathbf{W} & (\Psi - \Phi^k)\mathbf{W} & \mathbf{0}  \\
  \vdots & \vdots && \ddots & \ddots   \\
  (\Psi - \Phi^k)\Psi^{N_c-2}\mathbf{W} & (\Psi - \Phi^k)\Psi^{N_c-1}\mathbf{W} & ... & & (\Psi - \Phi^k)\mathbf{W} & \mathbf{0}
\end{bmatrix}.
\end{align*}
Excuse the slight abuse of notation and denote $\mathcal{R}_F : = (I - A_\Delta B_\Delta^{-1})R_{\textnormal{ideal}}\widetilde{P}$;
that is, ignore the upper zero blocks in $\mathcal{R}_F$. Define a tentative pseudoinverse, $\mathcal{R}_F^{\dagger}$, as
{\small
\begin{align*}
    \mathcal{R}_F^{\dagger} =
\begin{bmatrix}
\mathbf{0} & \widetilde{\mathbf{W}}^{-1}(\Psi-\Phi^k)^{-1}  \\
\mathbf{0} & -\widetilde{\mathbf{W}}^{-1}\Psi (\Psi-\Phi^k)^{-1} & \widetilde{\mathbf{W}}^{-1}(\Psi - \Phi^k)^{-1} \\
& & \ddots  & \ddots  \\
& & & -\widetilde{\mathbf{W}}^{-1}\Psi (\Psi-\Phi^k)^{-1} & \widetilde{\mathbf{W}}^{-1}(\Psi - \Phi^k)^{-1} \\
& & & \mathbf{0} & \mathbf{0}
\end{bmatrix}
\end{align*}
}
for some $\widetilde{\mathbf{W}}^{-1}$, and observe that
\begin{align*}
\mathcal{R}_F^{\dagger} \mathcal{R}_F & =
\begin{bmatrix} \widetilde{\mathbf{W}}^{-1}\mathbf{W} \\ & \ddots \\ & &
   \widetilde{\mathbf{W}}^{-1}\mathbf{W} \\ & & & \mathbf{0} \end{bmatrix}.
\end{align*}
Three of the four properties of a pseudoinverse require that,
\begin{align*}
    \mathcal{R}_F^{\dagger} \mathcal{R}_F \mathcal{R}_F^{\dagger} &= \mathcal{R}_F^{\dagger}, \quad
    \mathcal{R}_F \mathcal{R}_F^{\dagger} \mathcal{R}_F = \mathcal{R}_F \quad \text{and} \quad
    \left(\mathcal{R}_F^{\dagger} \mathcal{R}_F\right)^* = \mathcal{R}_F^{\dagger} \mathcal{R}_F.
\end{align*}
These three properties follow by picking $\widetilde{\mathbf{W}}^{-1}$ such that
$\left(\widetilde{\mathbf{W}}^{-1}\mathbf{W}\right)^* = \widetilde{\mathbf{W}}^{-1}\mathbf{W}$, and
$\mathbf{W}\widetilde{\mathbf{W}}^{-1}\mathbf{W} = \mathbf{W}$, $\widetilde{\mathbf{W}}^{-1}\mathbf{W} \widetilde{\mathbf{W}}^{-1}
=\widetilde{\mathbf{W}}^{-1}$.  Notice these are exactly the first three properties of a pseudoinverse of $\mathbf{W}$. To that
end, define $\widetilde{\mathbf{W}}^{-1}$ as the pseudoinverse of a full row rank operator,
\begin{align*}
\widetilde{\mathbf{W}}^{-1} = \mathbf{W}^*(\mathbf{W}\mathbf{W}^*)^{-1}.
\end{align*}
Note that here, $\mathbf{W}\widetilde{\mathbf{W}}^{-1} = I$, and the fourth property of a pseudoinverse for $\mathcal{R}_F$,
$\left(\mathcal{R}_F\mathcal{R}_F^{\dagger} \right)^* = \mathcal{R}_F\mathcal{R}_F^{\dagger}$, follows immediately.

Recall the maximum singular value of $\mathcal{R}_{F}$ is given by the minimum nonzero singular
value of $\mathcal{R}_F^\dagger$, which is equivalent to the minimum nonzero singular value of
$(\mathcal{R}_F^\dagger)^*\mathcal{R}_F^\dagger$. Following from \cite{southworth19}, the minimum nonzero eigenvalue
of $(\mathcal{R}_F^\dagger)^*\mathcal{R}_F^\dagger$ is bounded from above by the minimum
eigenvalue of a block Toeplitz matrix, with real-valued matrix generating function
{\small
\begin{align*}
F(x)
    = \left(e^{ix}\widetilde{\mathbf{W}}^{-1} \Psi(\Psi-\Phi^k) - \widetilde{\mathbf{W}}^{-1} (\Psi-\Phi^k)\right)
  \left(e^{ix}\widetilde{\mathbf{W}}^{-1} \Psi(\Psi-\Phi^k) - \widetilde{\mathbf{W}}^{-1} (\Psi-\Phi^k)\right)^*.
\end{align*}
}Let $\lambda_k(A)$ and $\sigma_k(A)$ denote the $k$th eigenvalue and singular value of some operator, $A$. Then,
\begin{align*}
\min_{\substack{x\in[0,2\pi],\\k}} \lambda_k(F(x)) & = \min_{\substack{x\in[0,2\pi],\\k}}
  \sigma_k\left(e^{ix}\widetilde{\mathbf{W}}^{-1} \Psi(\Psi-\Phi^k) - \widetilde{\mathbf{W}}^{-1} (\Psi-\Phi^k)\right)^2 \\
& = \min_{\substack{x\in[0,2\pi], \\ \mathbf{v}\neq \mathbf{0}}} \frac{\left\| \left(e^{ix}\widetilde{\mathbf{W}}^{-1} \Psi(\Psi-\Phi^k) -
  \widetilde{\mathbf{W}}^{-1} (\Psi-\Phi^k)\right) \mathbf{v}\right\|^2}{\|\mathbf{v}\|^2} \\
& =  \min_{\substack{x\in[0,2\pi], \\ \mathbf{v}\neq \mathbf{0}}} \frac{\left\| \widetilde{\mathbf{W}}^{-1}(e^{ix}\Psi - I)\mathbf{v}\right\|^2}{\|(\Psi - \Phi^k)\mathbf{v}\|^2} \\
& =  \min_{\substack{x\in[0,2\pi], \\ \mathbf{v}\neq \mathbf{0}}} \frac{\left\|
  (\mathbf{W}\mathbf{W}^*)^{-1/2}(e^{ix}\Psi - I)\mathbf{v}\right\|^2}{\|(\Psi - \Phi^k)\mathbf{v}\|^2}, \\
\sigma_{min}\left(\mathcal{R}_F^\dagger\right) & \leq \min_{\substack{x\in[0,2\pi], \\ \mathbf{v}\neq \mathbf{0}}} \frac{\left\| (\mathbf{W}\mathbf{W}^*)^{-1/2}(e^{ix}\Psi - I)\mathbf{v}\right\|^2}{\|(\Psi - \Phi^k)\mathbf{v}\|} + O(1/N_c).
\end{align*}

Then,
\begin{align*}
\left\|\mathcal{E}_{F}\right\|_{A^*A} = \left\|\mathcal{R}_{F}\right\| & \geq \frac{1}{\sqrt{\min_{\substack{x\in[0,2\pi], \\ \mathbf{v}\neq \mathbf{0}}}
  \frac{\left\|(\mathbf{W}\mathbf{W}^*)^{-1/2}(e^{ix}\Psi - I)\mathbf{v}\right\|^2}{\|(\Psi - \Phi^k)\mathbf{v}\|^2} + O(1/N_c)}} \\
& \geq \frac{1}{\min_{\substack{x\in[0,2\pi], \\ \mathbf{v}\neq \mathbf{0}}}
  \frac{\left\|(\mathbf{W}\mathbf{W}^*)^{-1/2}(e^{ix}\Psi - I)\mathbf{v}\right\|^2}{\|(\Psi - \Phi^k)\mathbf{v}\|^2}+ O(1/\sqrt{N_c})} \\
& = \max_{\mathbf{v}\neq \mathbf{0}} \frac{\|(\Psi - \Phi^k)\mathbf{v}\|}{\min_{x\in[0,2\pi]} \left\|(\mathbf{W}\mathbf{W}^*)^{-1/2}(e^{ix}\Psi - I)\mathbf{v}\right\|+O(1/\sqrt{N_c})},
\end{align*}
and the result for $\mathcal{E}_F$ follows.\footnote{More detailed steps for this proof involving the block Toeplitz matrix and
generating function can be found in similar derivations in \cite{southworth19,Serra:1998fm,Serra:1999ed,Miranda:2000is}.}

The case of $\mathcal{E}_{FCF}$ follows an identical derivation with the modified operator $\widehat{\mathbf{W}}=(\Phi^{2k-1},...,\Phi^k)$,
which follows from the right multiplication by $A_{cf}A_{ff}^{-1}A_{fc}$ in $\mathcal{R}_{FCF} = (I - A_\Delta B_\Delta^{-1})A_{cf}A_{ff}^{-1}A_{fc}R_{\textnormal{ideal}}$,
which is effectively a right-diagonal scaling by $\Phi^k$. The cases of error propagation in the $\ell^2$-norm follow a similar derivation based on $P_{ideal}$.
\end{proof}
\begin{proof}[Proof of \Cref{cor:new2grid}]
The derivations follow a similar path as those in \Cref{th:new2grid}. However, when considering Toeplitz operators
defined over the complex scalars (eigenvalues) as opposed to operators, additional results hold. In particular, the previous
lower bound (that is, necessary condition) is now a tight bound in norm, which follows from a closed form for the eigenvalues
of a perturbation to the first or last entry in a tridiagonal Toeplitz matrix \cite{Losonczi1992,DafonsecaKowalenko2019}.
Scalar values also lead to a
tighter asymptotic bound, $O(1/N_c)$ as opposed to $O(1/\sqrt{N_c})$, which is derived from the existence of a second-order
zero of $F(x) - \min_{y\in[0,2\pi]} F(y)$, when the Toeplitz generating function $F(x)$ is defined over complex scalars as
opposed to operators \cite{Serra:1998fm}. Analogous derivations for each of these steps can be found in the diagonalizable case in
\cite{southworth19}, and the steps follow easily when coupled with the pseduoinverse derived in \Cref{th:new2grid}.

Then, noting that
\begin{align*}
    \mathcal{W}_F
    &= \sqrt{\sum_{\ell=0}^{k-1} (|\lambda_i|^2)^\ell}
    = \sqrt{\frac{1-|\lambda_i|^{2 k}}{1 - |\lambda_i|^2}}, \\
    \mathcal{W}_{FCF}
    &= \sqrt{\sum_{\ell=k}^{2 k-1} (|\lambda_i|^2)^\ell}
    = |\lambda_i|^{k} \sqrt{\frac{1-|\lambda_i|^{2k}}{(1 - |\lambda_i|^2}} ,
\end{align*}
and substituting $\lambda_i$ for $\Phi$ and $\mu_i$ for $\Psi$ in \Cref{th:new2grid}, the result follows.
\end{proof}

\section*{Acknowledgments}
Los Alamos National Laboratory report number LA-UR-20-28395. This work was supported
by Lawrence Livermore National Laboratory under contract B634212, and under the
Nicholas C. Metropolis Fellowshop from the Laboratory Directed Research and
Development program of Los Alamos National Laboratory.

\bibliographystyle{plain}
\bibliography{main.bib}

\end{document}